\documentclass[12pt,twoside]{amsart}

\usepackage{amssymb}
\usepackage{aliascnt}
\usepackage[margin=1in]{geometry}
\usepackage{comment}
\usepackage{todonotes}
\usepackage{tcolorbox}
\usepackage{hyperref}
\usepackage{algorithm}
\usepackage{algpseudocode}
\usepackage{standalone,mathrsfs}
\usepackage{hyperref}
\hypersetup{colorlinks=true,allcolors=blue,pdftitle={Relaxed coparking functions and Stanley's conjecture for matroid h-vectors},pdfauthor={Suho Oh}}

\usepackage{tikz}
\usetikzlibrary{calc}

\usepackage[nameinlink]{cleveref}

\newtheorem{theorem}{Theorem}[section]

\newaliascnt{lemma}{theorem}
\newtheorem{lemma}[lemma]{Lemma}
\aliascntresetthe{lemma}

\newaliascnt{corollary}{theorem}
\newtheorem{corollary}[corollary]{Corollary}
\aliascntresetthe{corollary}

\newaliascnt{proposition}{theorem}
\newtheorem{proposition}[proposition]{Proposition}
\aliascntresetthe{proposition}

\newaliascnt{conjecture}{theorem}
\newtheorem{conjecture}[conjecture]{Conjecture}
\aliascntresetthe{conjecture}

\newaliascnt{question}{theorem}
\newtheorem{question}[question]{Question}
\aliascntresetthe{question}

\theoremstyle{definition}

\newaliascnt{definition}{theorem}
\newtheorem{definition}[definition]{Definition}
\aliascntresetthe{definition}

\newaliascnt{example}{theorem}
\newtheorem{example}[example]{Example}
\aliascntresetthe{example}

\newaliascnt{remark}{theorem}

\aliascntresetthe{remark}

\crefname{theorem}{theorem}{theorems}
\Crefname{theorem}{Theorem}{Theorems}

\crefname{lemma}{lemma}{lemmas}
\Crefname{lemma}{Lemma}{Lemmas}

\crefname{corollary}{corollary}{corollaries}
\Crefname{corollary}{Corollary}{Corollaries}

\crefname{proposition}{proposition}{propositions}
\Crefname{proposition}{Proposition}{Propositions}

\crefname{definition}{definition}{definitions}
\Crefname{definition}{Definition}{Definitions}

\crefname{example}{example}{examples}
\Crefname{example}{Example}{Examples}

\crefname{conjecture}{conjecture}{conjectures}
\Crefname{conjecture}{Conjecture}{Conjectures}

\crefname{question}{question}{questions}
\Crefname{question}{Question}{Questions}

\crefname{remark}{remark}{remarks}
\Crefname{remark}{Remark}{Remarks}

\newcommand{\CC}{\mathcal{C}}

\usepackage{booktabs}
\usepackage{enumitem}
\usepackage{array}

\DeclareMathOperator{\Ext}{Ext}
\DeclareMathOperator{\exc}{exc}
\DeclareMathOperator{\out}{out}
\newcommand{\N}{\mathbb{N}}
\providecommand{\Z}{\mathbb{Z}}
\newcommand{\PP}{\mathcal{P}}
\newcommand{\word}[1]{\mathsf{#1}}
\usepackage{amsmath,amssymb}
\usetikzlibrary{calc,fit,shapes.geometric,decorations.pathreplacing,arrows.meta,positioning,patterns}
\definecolor{redge}{RGB}{190,30,45}
\definecolor{grey}{RGB}{120,128,125}
\tikzset{
  hub/.style={draw,rectangle,minimum size=5mm,inner sep=1pt,fill=white,font=\small},
  ch/.style={draw,circle,minimum size=5mm,inner sep=1pt,fill=white,font=\small},
  ctr/.style={draw,circle,minimum size=6mm,inner sep=1pt,fill=white,thick,font=\small},
  wh/.style={draw,rectangle,minimum size=5mm,inner sep=1pt,fill=white,font=\small,dashed},
  tr/.style={grey,thick},
  rd/.style={redge,very thick},
  uu/.style={blue!70!black,line width=2.2pt,opacity=0.35},
  lab/.style={font=\footnotesize},
  rlab/.style={redge,font=\footnotesize}}
\providecommand{\word}[1]{\mathsf{#1}}
\providecommand{\exc}{\operatorname{exc}}

\title[Relaxed coparking functions and Stanley's conjecture]{Relaxed coparking functions and Stanley's conjecture for matroid $h$-vectors}

\author{Suho Oh}
\address{Texas State University}
\email{suhooh@txstate.edu}

\date{September 2026}

\begin{document}

\begin{abstract}
Stanley's conjecture asserts that the $h$-vector of a matroid is a pure $O$-sequence. Corry,
Dochtermann, McClain, Perkinson and Yi introduced cycle systems, which give a bijective proof for the
matroids that admit one, through coparking functions, and the same authors proposed
generalized cycle systems, in which an independent unique union may be a basis of the restriction to the union of
its cycles. Every
matroid admitting a cycle system is regular, and
there are graphs, among them the Wagner graph and the Petersen graph, that admit no generalized
cycle system consisting of circuits. We propose a relaxation. Where the coparking recursion breaks down, at the
strata whose unique union is independent, the coparking functions are replaced by a \emph{fibre}: a pure multicomplex with the $h$-vector of the dead node, the minor of the matroid attached to the stratum. We prove, for any
matroid with a fixed basis, that whenever the required fibres exist the resulting relaxed coparking
functions form a pure multicomplex whose degree sequence is the $h$-vector of the matroid, so Stanley's
conjecture follows. The proof rests on a
version of Dhar's burning algorithm for a matroid with a fixed basis, which gives the purity, and on a
deletion--contraction identity that computes the gap between the $h$-vector and the coparking
functions as a sum of local $h$-vectors, one per dead node. Coned, biconed and triconed graphs, the classes of graphs for which Stanley's conjecture was proved through a canonical spanning tree,
carry a fibre system with their canonical spanning trees. The Wagner and Petersen graphs carry one as well. Beyond graphs, every
basis of a matroid of corank two, of a matroid of rank at most four or of a uniform matroid
carries a fibre system, so that Stanley's conjecture follows for those classes. We also give a graph of radius two on twelve vertices which, with its breadth-first spanning tree, carries no fibre system.
\end{abstract}

\maketitle

\section{Introduction}\label{sec:intro}

Let $M$ be a matroid of rank $d$ on a finite ground set and let $h(M)=(h_0,h_1,\dots,h_d)$ be the
$h$-vector of its independence complex. Stanley proved in 1977 that matroid complexes are
Cohen--Macaulay, so that $h(M)$ is an \emph{$O$-sequence}, the degree sequence of an order
ideal of monomials, and conjectured that more is true \cite{Stanley}. See also
\cite[Conjecture~III.3.6]{StanleyBook} and, for pure $O$-sequences in general, \cite{BMMNZ}.

\begin{conjecture}[Stanley]
The $h$-vector of a matroid is a \emph{pure} $O$-sequence: the degree sequence of a finite order
ideal of monomials all of whose maximal elements have the same degree.
\end{conjecture}

The conjecture has been open for almost fifty years and has been established for a number of
classes of matroids, in almost every case by an explicit bijection between the bases of the
matroid and the monomials of a pure order ideal. It holds for cographic matroids, by Merino
\cite{MerinoTutte,Merino2001} through the chip-firing game and $G$-parking functions, and by
Chari \cite{Chari} through a decomposition of the matroid complex. It holds for lattice-path
matroids (Schweig \cite{Schweig}), for cotransversal matroids (Oh \cite{OhCotransversal}, through
generalized permutohedra) and for positroids (He, Lai and Oh \cite{Positroid}). It holds for paving
matroids (Merino, Noble, Ram\'{\i}rez-Iba\~nez and Villarroel-Flores \cite{MNRV}), for
internally perfect matroids (Dall \cite{Dall}). It also holds for matroids of rank at most three (H\`a,
Stokes and Zanello \cite{HSZ}), for matroids of rank three and of corank two and for all matroids
on at most nine elements (De Loera, Kemper and Klee \cite{DLKK}, the corank-two case recovered in Section~\ref{sec:coranktwo}). It holds for matroids of rank four
(Klee and Samper \cite{KleeSamper}, recovered together with the rank-three case in Section~\ref{sec:rankfour}) and for matroids of rank $d$ with $h_d\le5$
(Constantinescu, Kahle and Varbaro \cite{CKV}). In a different direction, the shape of $h(M)$ is
now understood much better than it was: Huh proved that the $h$-vector of a matroid
representable over a field of characteristic zero is log-concave \cite{Huh}, and Berget, Spink
and Tseng extended log-concavity to all matroids \cite{BST}. Both came in the wake of the resolution of the
Rota--Heron--Welsh conjecture on characteristic polynomials by Adiprasito, Huh and Katz
\cite{AHK}. Pure
$O$-sequences need not be log-concave or even unimodal, and log-concavity does not imply
purity. So these results constrain the conjecture without settling it. The same is true of the inequalities $h_0\le h_1\le\dots\le h_{\lfloor d/2\rfloor}$ and $h_i\le h_{d-i}$ for $i\le\lfloor d/2\rfloor$, which Hibi proved for every pure $O$-sequence \cite{Hibi} and Chari for every matroid $h$-vector \cite{Chari}.

For graphic matroids, whose bases are the spanning trees of a graph, the conjecture is open. The bijective proofs for the subclasses treated in \cite{Kook,Biconed,Triconed} all start from a small dominating structure. Kook proved it for coned graphs, the graphs dominated by one vertex \cite{Kook}.
Cranford, Dochtermann, Haithcock, Marsh, Oh and Truman proved it for biconed graphs (dominated by
an edge) \cite{Biconed}, and David, Lai, Oh and Wu for triconed graphs (dominated by a path of
length two) \cite{Triconed}. Each of these proofs is an explicit bijection from spanning trees to weighted
forests in the complement of a canonical spanning tree, the forests being read as monomials on the non-tree edges. The case analysis grows quickly with the size of the dominating structure: the triconed
bijection already needs a special treatment of ``sandwiched'' marks, and no extension of it to
graphs dominated by a claw or by a longer path is known.

Corry, Dochtermann, McClain, Perkinson and Yi \cite{CycleSystems} abstracted the mechanism of
the coned case into matroid language. A \emph{cycle system} on a matroid $M$ of corank $g$ is a
collection $\CC=\{C_1,\dots,C_g\}$ of cycles whose \emph{unique union} $U_S$, the set of elements lying in exactly one $C_i$ with $i\in S$, is dependent for every nonempty $S\subseteq[g]$.
To a cycle system they attach the \emph{coparking functions} $\PP^*(\CC)$, the vectors in $\N^g$ from which no set $S$ can be fired under the thresholds $|C_i\cap U_S|$, and they prove that
$\PP^*(\CC)$ is a pure multicomplex with degree sequence $h(M)$. The proof is a
deletion--contraction recursion on the pair $(M,\CC)$ that always has an element to expand, one lying in exactly one cycle, precisely because every unique union is dependent.

The authors of \cite{CycleSystems} also considered running the same recursion on a collection
of cycles that is \emph{not} a cycle system, and verified by computation that on one such
collection on $K_{3,3}$ it still produces the right pure multicomplex
\cite[Section~6.2]{CycleSystems}. They observed that every unique union of that collection is
dependent \emph{or} a basis of the union of the cycles it comes from. The condition, and the
question of which collections of circuits produce a complete deletion--contraction tree, arose in
the work of the authors of \cite{CycleSystems}; see \cite{Dochtermann,PerkinsonDC,Yi}. Collections
with the property are called \emph{generalized cycle systems} (Definition~\ref{def:gcs}), and it is
open whether a matroid admitting one always has a pure multicomplex of coparking functions with
degree sequence $h(M)$. Biconed graphs admit generalized cycle systems by the fundamental
circuits of their canonical tree \cite{Yi}. Two facts fix the setting of the present paper, and both are recalled in
Section~\ref{sec:cs}. First, every matroid admitting a cycle system is regular \cite{Regular},
and since cographic matroids are already covered by Merino's theorem the natural place to look
for an extension is among graphic matroids. Second, Perkinson's exhaustive search
\cite{PerkinsonDC} shows that the Wagner graph $V_8$ and the Petersen graph admit no generalized
cycle system consisting of circuits: for these two graphs the deletion--contraction recursion gets
stuck under every choice of circuits and every order of expansion.

This paper proposes a relaxation that removes the obstruction at exactly the place where it
occurs, and it does so for an arbitrary matroid. Fix a basis $B_0$ of a matroid $M$ and take
$\CC$ to be the fundamental circuits of the elements outside $B_0$. Those elements we call \emph{red},
after the usage of \cite{Biconed,Triconed} for the non-tree edges of a graph. The recursion of \cite{CycleSystems} runs as long as
some element lies in exactly one of the current cycles, as the authors point out in
\cite[Section~6.2]{CycleSystems}, and stops exactly when it has contracted the unique union
$U_\sigma$ of a set $\sigma$ of red edges, which happens exactly when $U_\sigma$ is
\emph{independent}. We call such a $\sigma$ a \emph{dead stratum}. Its \emph{dead node} is the minor $K_\sigma=(M|A_\sigma)/U_\sigma$, where $A_\sigma$ is the union of the circuits of $\sigma$. At a dead node the
coparking functions are replaced by a \emph{fibre}: a pure order ideal of monomials on the
coordinates of $\sigma$ with degree sequence $h(K_\sigma)$, placed above the \emph{coparking baseline}
$b_\sigma(i)=|C_i\cap U_\sigma|$ and subject to a compatibility condition with the strata
below $\sigma$. A stratum $\tau$ can be fired from a vector $a$ when $a|_\tau\ge b_\tau$. The relaxed coparking functions are then the vectors $a$ from which no live stratum can be fired and from which a dead stratum
$\tau$ can be fired only with an offset $a|_\tau-b_\tau$ in the fibre of $\tau$ (Definition~\ref{def:relaxed}). When no
stratum is dead this is exactly the definition of a coparking function, and when every
dead node has rank zero the fibres are trivial. The relaxation is therefore a generalisation of
both cycle systems and generalized cycle systems that preserves their philosophy, in the precise sense of Corollary~\ref{cor:specialise}. For the generalized cycle systems formed by the fundamental circuits of a basis, the question
of whether the recursion gives the correct pure multicomplex is settled by the special case of
Theorem~\ref{thm:intro-main} in which every fibre is trivial.

The main general result is that nothing is lost in the relaxation.

\begin{theorem}[Theorem~\ref{thm:main}]\label{thm:intro-main}
Let $M$ be a matroid with a basis $B_0$, and suppose that every dead stratum admits a
fibre. Then the relaxed coparking functions form a pure multicomplex whose degree sequence is
the $h$-vector of $M$. In particular Stanley's conjecture holds for $M$.
\end{theorem}

The theorem is stated for an arbitrary matroid with a fixed basis, because nothing in its proof uses
more than the private elements of the fundamental circuits and the orthogonality of circuits and
cocircuits. Most classes to which we apply it are graphic, the exceptions being the matroids of corank two, the matroids of rank at most four and the uniform matroids of Sections~\ref{sec:coranktwo} to~\ref{sec:uniform}. The argument has three parts. The
first is a decomposition of the relaxed coparking functions into products of fibres and
\emph{extension sets} (Lemma~\ref{lem:D}). The second is that every extension set is a set of
coparking functions in its own right, of a contraction of $M$, and that the coparking functions of
any matroid with a fixed basis form a pure multicomplex. The proof of the latter is a version of Dhar's burning
algorithm \cite{Dhar,Klivans}, the classical test for superstable configurations on a graph, in which the
vertices are replaced by the non-basis elements and the edges at a vertex by the elements of a
fundamental circuit (Theorem~\ref{thm:burning}). Purity survives in this version and the count does not: the coparking functions of a basis are in general fewer than the bases. The
third part computes exactly how many fewer. The deletion--contraction recursion of
\cite{CycleSystems} is now allowed to stop at dead nodes. It gives the stratum identity of Section~\ref{sec:model} (Theorem~\ref{thm:B}): $h(M)$ is the sum over the empty set and the dead
strata of $t^{|b_\sigma|}h(K_\sigma)$ times the generating function $E_\sigma(t)$ of the extension set
$\Ext(\sigma)$. So the gap between the $h$-vector and the coparking functions is a sum of
local $h$-vectors, one per dead node. A fibre system is a way of
filling that gap monomial by monomial. Stanley's conjecture for $M$ is thereby reduced to a local question: does each dead stratum admit a fibre?

The second half of the paper answers that question for the graphs for which Stanley's
conjecture was previously known through a canonical tree, for two of the graphs on which the generalized cycle system
recursion is known to fail, and for three classes of matroids not contained in the graphic matroids. For radius-two graphs with the
breadth-first tree from the centre we develop combinatorial criteria that read whether a stratum is dead, and
the rank of its dead node, from the graph (Section~\ref{sec:radius2}). As a consequence of these criteria, every dead stratum of a triconed graph has a dead node of rank at most one, and at every such stratum a fibre exists (Theorems~\ref{thm:rankone} and~\ref{thm:triconed}). The biconed and coned cases are included. For a cone with the star as tree no stratum is dead, and the model is the cycle system of \cite{CycleSystems}. We thereby recover the theorem of \cite{Triconed} by a proof that does not use the construction machinery of that paper, machinery already very heavy
for biconed and triconed graphs. The Wagner graph and the Petersen
graph both carry a fibre system: on the Petersen graph the twenty-four dead
strata fall into three types, and the two nontrivial types carry explicit fibres that are
valid independently of every other choice (Theorem~\ref{thm:petersen}).

The model is not confined to graphs, and the three non-graphic classes show how little it needs in order to
apply. For a matroid of corank two the dead node is again a matroid of corank two, so an induction recovers the corank-two theorem of \cite{DLKK} (Theorem~\ref{thm:coranktwo}). For a matroid of rank at most four the dead nodes have corank two or rank one, and a count of loops at the rank-one nodes recovers the theorems of H\`a, Stokes and Zanello \cite{HSZ} for rank at most three and of Klee and Samper \cite{KleeSamper} for rank four (Theorem~\ref{thm:rankfour}). For uniform matroids the dead nodes are again uniform, and the fibres are simplices of monomials that fit together because they are nested (Theorem~\ref{thm:uniform}).

Not every pair $(M,B_0)$ carries a fibre system. The live strata below a dead stratum may exclude too many offsets to leave a pure multicomplex filling $h(K_\sigma)$ (Section~\ref{sec:criteria}). Section~\ref{sec:outlook} gives a
radius-two graph with its breadth-first tree on which this happens: a path of eight outer
vertices winding cyclically around three hubs.

The paper is organised as follows. Section~\ref{sec:prelim} fixes the conventions on matroids,
$h$-vectors, a basis and its fundamental circuits. Section~\ref{sec:cs} reviews cycle systems and
generalized cycle systems. Section~\ref{sec:model} defines strata, dead nodes, fibres and the
relaxed coparking functions, proves the burning theorem and the stratum identity, deduces the
main theorem, and records the specialisations, the truncated $h$-polynomial, the spared offsets
and the unconditional fibres. Section~\ref{sec:radius2} develops the combinatorial criteria on radius-two graphs and the
analysis, in the terminology introduced there, of the rank-one dead nodes of the roots with two fertile hubs. Section~\ref{sec:cases} treats coned, biconed
and triconed graphs, the Wagner and Petersen graphs, matroids of corank two, matroids of rank
at most four and uniform matroids. Section~\ref{sec:outlook} gives a radius-two graph with its breadth-first tree that carries no fibre system and records what remains open.

\section{Preliminaries}\label{sec:prelim}

\subsection{Matroids, \texorpdfstring{$h$}{h}-vectors and pure \texorpdfstring{$O$}{O}-sequences}

We follow Oxley \cite{Oxley}. A matroid $M$ on a finite ground set $E$ is given by its
independent sets, and $r$ denotes its rank function. For $X\subseteq E$ and $e\in E$ we write $X\setminus e$ and $X\cup e$ for
$X\setminus\{e\}$ and $X\cup\{e\}$. A \emph{circuit} is a minimal dependent set, a
\emph{cycle} is a union of circuits, a \emph{loop} is an element in no basis and a
\emph{coloop} an element in every basis. For $S\subseteq E$ the \emph{deletion} $M\setminus S$
and the \emph{contraction} $M/S$ are the usual minors, and $M|S=M\setminus(E\setminus S)$ is the
restriction. The independent sets of $M$ form a simplicial complex whose $h$-vector
$h(M)=(h_0,\dots,h_d)$, $d=r(M)$, is determined by
\[
  T_M(x,1)=\sum_{k=0}^{d}h_kx^{d-k},
\]
where $T_M$ is the Tutte polynomial. Equivalently, $h_k$ counts the bases with exactly $k$
internally passive elements with respect to any fixed linear order of $E$ \cite{Bjorner}. We
write $h(M;t)=\sum_k h_kt^k$, and simply $h(M)$ when the variable is understood: $t$ is the
formal variable throughout the paper, and the same convention applies to $E_\tau(t)$ and
$B(M,\CC)$ below. Deleting a loop or a coloop does not change the sequence
$(h_0,\dots,h_d)$, and for $e$ neither a loop nor a coloop,
\begin{equation}\label{eq:tutte}
  h(M)=h(M\setminus e)+t\,h(M/e).
\end{equation}
The \emph{degree} of $h(M)$ is the largest $k$ with $h_k\ne0$. It equals $r(M)$ minus the number
of coloops of $M$.

A \emph{multicomplex} on a finite set of variables is a finite set of monomials closed under
division. We identify a monomial $x^a$ with its exponent vector $a$ and write $|a|$ for its degree. The multicomplex is \emph{pure} if all its maximal elements have the same degree. Its
\emph{degree sequence} counts the monomials of each degree, and a sequence arising this way from
a pure multicomplex is a \emph{pure $O$-sequence}. Stanley's conjecture \cite{Stanley} asserts
that $h(M)$ is a pure $O$-sequence for every matroid $M$.

\subsection{A basis and its fundamental circuits}

Throughout, $M$ is a matroid on a ground set $E$ with a fixed basis $B_0$, and $R=E\setminus B_0$.
For $i\in R$ the \emph{fundamental circuit} $C_i$ is the unique circuit contained in
$B_0\cup\{i\}$. We call the elements of $B_0$ \emph{tree edges} and those of $R$ \emph{red edges} as in
\cite{Biconed,Triconed}, also when $M$ is not graphic. The names come from the case that motivates everything. There $M=M(G)$ is the graphic matroid of a connected multigraph $G$, whose circuits are the edge sets of cycles and whose bases are the spanning trees. The basis $B_0$ is a spanning tree, and $C_i$ consists of $i$ and the tree path joining its ends. Nothing below uses that $B_0$ is a tree. The graphs return in Section~\ref{sec:radius2}, where the radius-two structure of $G$ is used. Three facts, valid for an arbitrary matroid with a fixed basis, are used throughout:
\begin{enumerate}[label=(F\arabic*)]
\item $i\in C_i$ and $i\notin C_j$ for $j\ne i$ (the \emph{private element} of $C_i$).
\item For a tree edge $e$ let $D_e=\{e\}\cup\{i\in R:e\in C_i\}$, the fundamental cocircuit of
      $e$ with respect to $B_0$. A circuit and a cocircuit never meet in exactly one element
      \cite[Prop.~2.1.11]{Oxley}, so for every circuit $C$ the following holds: a tree edge of
      $C$ lies in $C_i$ for some $i\in C\setminus B_0$, and a tree edge outside $C$ lies in $C_i$
      for none or for at least two such $i$.
\item Let $C$ be a circuit with red set $S=C\setminus B_0$. Every element of $C$ lies in
      some $C_j$ with $j\in S$, so $C\subseteq A_S$ in the notation below, and an element
      of $E$ lying in exactly one $C_j$ with $j\in S$ belongs to $C$.
\end{enumerate}
Fact (F3) follows from the first two. For the first claim, a red element of $C$ is its own index, and a tree edge of $C$
lies in some $C_j$ with $j\in S$ by (F2). For the second, a red element lying in some $C_j$ with $j\in S$ is $j$
itself by (F1), and a tree edge outside $C$ lies in $C_j$ for none or for at least two $j\in S$
by (F2), never for exactly one.

The red edges index the variables $x_i$, $i\in R$, of every multicomplex constructed in this
paper. For $S\subseteq R$ we write $A_S=\bigcup_{i\in S}C_i$. Then
\begin{equation}\label{eq:nullity}
  r(A_S)=|A_S|-|S| .
\end{equation}
Indeed $A_S\setminus S=B_0\cap A_S$ is independent, and it spans $A_S$, since each $i\in S$ lies
in the circuit $C_i\subseteq(B_0\cap A_S)\cup\{i\}$, so $r(A_S)=|A_S\setminus S|=|A_S|-|S|$ by (F1).

\subsection{Unique unions}

\begin{definition}[\cite{CycleSystems}]
For subsets $C_1,\dots,C_k$ of a set $E$ and $S\subseteq[k]$ the \emph{unique union}
$\ast\{C_i:i\in S\}$ is the set of elements of $E$ lying in exactly one $C_i$ with $i\in S$.
\end{definition}

For a set $\sigma\subseteq R$ of red edges and $e\in E$ let $m_\sigma(e)=\#\{i\in\sigma:e\in C_i\}$
be the \emph{multiplicity} of $e$, and put
\[
  U_\sigma=\{e: m_\sigma(e)=1\}=\ast\{C_i:i\in\sigma\},\qquad
  A_\sigma=\{e:m_\sigma(e)\ge1\}.
\]
Although the definition is stated with the star, we write $U_\sigma$ throughout the paper
whenever the family is the fundamental circuits $C_i$, $i\in\sigma$. The star is reserved for
the few places where the family is a modified one, such as the sets $C_j\setminus A_\tau$ in
Lemma~\ref{lem:extcontract}.
By (F1) every $i\in\sigma$ lies in $U_\sigma$, so $\sigma\subseteq U_\sigma\subseteq A_\sigma$.
For an element $e$ the set $\{i\in\sigma:e\in C_i\}$ is its \emph{owner set}, of size
$m_\sigma(e)$. The dependence on $\sigma$ is suppressed when it is clear.

\begin{example}[The running example]\label{ex:running}
Figure~\ref{fig:running} shows the graph $G_0$ on the vertices $0,\dots,5$ with edges
$01,02,03,14,25$ (the tree $B_0$) and $12,24,35,45$ (the red edges). Its fundamental circuits
are
\begin{gather*}
  C_{12}=\{12,01,02\},\qquad C_{24}=\{24,14,01,02\},\\
  C_{35}=\{35,25,02,03\},\qquad C_{45}=\{45,14,01,02,25\}.
\end{gather*}
For $\sigma=\{24,35,45\}$ the multiplicities are $m_\sigma(02)=3$, $m_\sigma(01)=m_\sigma(14)=
m_\sigma(25)=2$ and $m_\sigma(03)=1$, so $U_\sigma=\{24,35,45,03\}$ and $A_\sigma$ is everything except $12$. The set $U_\sigma$ is a path $0\,3\,5\,4\,2$, a forest, so this
$\sigma$ is dead in the sense of Definition~\ref{def:strata} below. For $\tau=\{24,45\}$ one
finds $U_\tau=\{24,45,25\}$, the triangle $2\,4\,5$, so $\tau$ is live. This graph is used throughout the paper. It is triconed
(dominated by the path $1\,0\,2$), has $66$ spanning trees and $h(M(G_0))=(1,4,10,17,20,14)$.
\end{example}

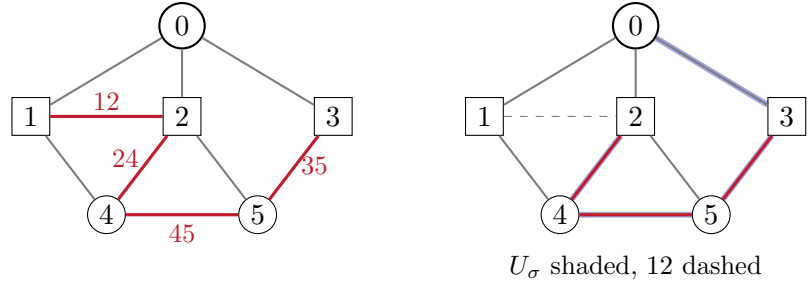
\begin{figure}[h]\centering
\begin{tikzpicture}[scale=1]
\node[ctr] (o) at (0,2.2) {$0$};
\node[hub] (h1) at (-2,1) {$1$}; \node[hub] (h2) at (0,1) {$2$}; \node[hub] (h3) at (2,1) {$3$};
\node[ch] (c4) at (-1,-0.3) {$4$}; \node[ch] (c5) at (1,-0.3) {$5$};
\draw[tr] (o)--(h1); \draw[tr] (o)--(h2); \draw[tr] (o)--(h3); \draw[tr] (h1)--(c4); \draw[tr] (h2)--(c5);
\draw[rd] (h1)--(h2); \draw[rd] (h2)--(c4); \draw[rd] (h3)--(c5); \draw[rd] (c4)--(c5);
\node[rlab] at (-1,1.2) {$12$}; \node[rlab] at (-0.75,0.45) {$24$}; \node[rlab] at (1.75,0.35) {$35$}; \node[rlab] at (0,-0.55) {$45$};
\begin{scope}[xshift=6cm]
\node[ctr] (o) at (0,2.2) {$0$};
\node[hub] (h1) at (-2,1) {$1$}; \node[hub] (h2) at (0,1) {$2$}; \node[hub] (h3) at (2,1) {$3$};
\node[ch] (c4) at (-1,-0.3) {$4$}; \node[ch] (c5) at (1,-0.3) {$5$};
\draw[uu] (o)--(h3); \draw[uu] (h2)--(c4); \draw[uu] (h3)--(c5); \draw[uu] (c4)--(c5);
\draw[tr] (o)--(h1); \draw[tr] (o)--(h2); \draw[tr] (o)--(h3); \draw[tr] (h1)--(c4); \draw[tr] (h2)--(c5);
\draw[rd] (h2)--(c4); \draw[rd] (h3)--(c5); \draw[rd] (c4)--(c5);
\draw[grey,dashed] (h1)--(h2);
\node[lab] at (0,-1.0) {$U_\sigma$ shaded, $12$ dashed};
\end{scope}
\end{tikzpicture}
\caption{The running example $G_0$ (left): the centre $0$ is a thick circle, hubs are squares
and children are circles (terms of Section~\ref{sec:radius2}), tree edges are grey and red edges are red. Right: the unique union
$U_\sigma=\{24,35,45,03\}$ of the stratum $\sigma=\{24,35,45\}$, shaded, a forest. The dashed
edge $12$ is the one edge outside $A_\sigma$.}\label{fig:running}
\end{figure}

Table~\ref{tab:notation} collects the notation introduced here and in
Section~\ref{sec:model}.

\begin{table}[h]
\centering
\begin{tabular}{@{}ll@{}}
\toprule
$B_0$, $R$ & the fixed basis (tree edges) and its complement (red edges) \\
$C_i$ & fundamental circuit of $i\in R$, the unique circuit in $B_0\cup i$ \\
$m_\sigma(e)$ & multiplicity of $e$ among $\{C_i:i\in\sigma\}$, the size of its owner set \\
$U_\sigma$, $A_\sigma$ & unique union and union of the $C_i$, $i\in\sigma$ \\
$\sigma$ dead, live & $U_\sigma$ is independent in $M$ (a forest of $G$), or dependent \\
$b_\sigma\in\N^\sigma$ & coparking baseline, $b_\sigma(i)=|C_i\cap U_\sigma|$ \\
$K_\sigma=(M|A_\sigma)/U_\sigma$ & dead node of a dead $\sigma$ \\
$r(K_\sigma)$ & rank of the dead node, the degree of $h(K_\sigma)$, equal to $r(A_\sigma)-|U_\sigma|$ \\
$F(\sigma)\subseteq\N^\sigma$ & fibre at $\sigma$: offsets above $b_\sigma$ \\
$\sigma$ obstructed & dead stratum at which no fibre exists, given the fibres below \\
$\mathring\sigma$ & the interior of $\sigma$, the set of proper strata $\tau\subsetneq\sigma$ \\
$a$ coparking on $\mathcal T$ & every $\tau\in\mathcal T$ that can be fired from $a$ is dead, with $a|_\tau-b_\tau\in F(\tau)$ \\
$\Delta^\sigma_\tau$, $\exc_\tau$ & elements owned in $\tau\subseteq\sigma$ but not in $\sigma$, and the excess $b_\tau-b_\sigma|_\tau$ \\
$\PP^*(\CC)$ & coparking functions of a collection $\CC$ \\
$\Ext(\sigma)\subseteq\N^{R\setminus\sigma}$, $E_\sigma(t)$ & extension set relative to $\sigma$ and its generating function \\
$a|_\tau\ge b_\tau$ & $\tau$ \emph{can be fired} from $a$ ($a$ \emph{dominates} $b_\tau$) \\
$\mathrm{Dom}(a)$, $\max\mathrm{Dom}(a)$ & the strata that can be fired from $a$, and the largest of them \\
$\PP^*(F)$ & relaxed coparking functions of a fibre system $F$ \\
\bottomrule
\end{tabular}
\caption{Notation.}\label{tab:notation}
\end{table}

\section{Cycle systems and generalized cycle systems}\label{sec:cs}

\subsection{Cycle systems and coparking functions}\label{sec:coparking}

\begin{definition}[\cite{CycleSystems}]\label{def:cs}
Let $M$ be a matroid of corank $g$. A \emph{cycle system} on $M$ is a collection
$\CC=\{C_1,\dots,C_g\}$ of cycles of $M$ such that $U_S=\ast\{C_i:i\in S\}$ is dependent for every
nonempty $S\subseteq[g]$.
\end{definition}

For a cycle system $\CC$ and a nonempty $S\subseteq[g]$ put $b_S(i)=|C_i\cap U_S|$ for $i\in S$. We say that $S$ \emph{can be fired} from $a\in\N^g$ if $a_i\ge b_S(i)$ for every $i\in S$. A vector $a\in\N^g$ is a \emph{coparking function} of $\CC$ if no nonempty $S$ can be fired from $a$, that is, if for every nonempty $S$ some $i\in S$ has $a_i<b_S(i)$
\cite[Definition 4.1]{CycleSystems}. The name comes from the set-firing of chip-firing, where $b_S(i)$ is the number of chips $i$ sends out when all of $S$ fires at once. Let $\PP^*(\CC)$ be the set of coparking functions, the notation of \cite{CycleSystems}, where the star refers to the dual matroid. A set that can be fired from $a'\le a$ can be fired from $a$, so $\PP^*(\CC)$ is closed under division.

The definition makes sense for any collection $\CC$ of subsets of $E$, not only for cycle systems, and we use it in that generality. Only Theorem~\ref{thm:cdmpy} needs $\CC$ to be a cycle system.

\begin{theorem}[{\cite[Proposition 4.3, Theorem 4.5]{CycleSystems}}]\label{thm:cdmpy}
If $\CC$ is a cycle system on $M$ then $\PP^*(\CC)$ is a pure multicomplex whose degree sequence
is $h(M)$.
\end{theorem}

The proof is a deletion--contraction recursion on the pair $(M,\CC)$. An element $e$ lying in
exactly one $C_i$ is \emph{admissible}. The vectors with $a_i=0$ correspond to the pair
$(M\setminus e,\CC\setminus\{C_i\})$ and the vectors with $a_i\ge1$, shifted by $-\mathbf e_i$,
to the pair $(M/e,\{C_j\setminus e: j\in[g]\})$ \cite[Proposition 3.6]{CycleSystems}. Both children are
again pairs carrying a cycle system, and the recursion runs until no coordinates remain. The point of the
definition is that an admissible element always exists: the total unique union $U_{[g]}$
is dependent, in particular nonempty. Cographic matroids (with the vertex cut sets),
planar graphs (with the bounded faces), coned graphs (with the triangles through the cone
vertex) and $K_{3,3}$-minor-free graphs admit cycle systems. $M(K_{3,3})$ does not
\cite{CycleSystems}. For a connected matroid every cycle system consists of circuits and is a
basis of the circuit space \cite[Theorem 3.10]{CycleSystems}, \cite[Corollary 4.4]{Regular}.

\begin{example}[A cycle system on a cone]\label{ex:cone}
Let $G$ be the cone over the path $1\,2\,3\,4$ with apex $0$ (Figure~\ref{fig:cone}). With the
star at $0$ as spanning tree the red edges are $12,23,34$, and their fundamental circuits are the
triangles $C_{12}=\{01,12,02\}$, $C_{23}=\{02,23,03\}$, $C_{34}=\{03,34,04\}$. The unique union of
two adjacent triangles is the $4$-cycle through the apex, $U_{\{12,23\}}=\{01,12,23,03\}$,
and the unique union of all three is the $5$-cycle $\{01,12,23,34,04\}$. The two non-adjacent
triangles have unique union $C_{12}\cup C_{34}$. Every unique union is dependent, so the triangles
form a cycle system, and the coparking functions $(a_{12},a_{23},a_{34})$ are the vectors from which no set can be fired: each coordinate is at most $2$, $(a_{12},a_{23})\not\ge(2,2)$,
$(a_{23},a_{34})\not\ge(2,2)$, $(a_{12},a_{34})\not\ge(3,3)$ (vacuous), and
$(a_{12},a_{23},a_{34})\not\ge(2,1,2)$. They number $21$, as do the spanning trees of $G$, with degree sequence $(1,3,6,7,4)=h(M(G))$.
\end{example}

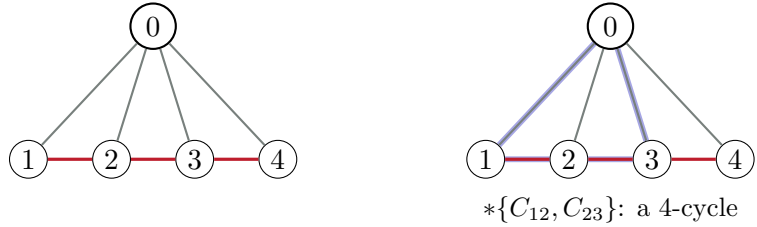
\begin{figure}[h]\centering
\begin{tikzpicture}[scale=1.1]
\node[ctr] (o) at (1.5,1.6) {$0$};
\node[ch] (a) at (0,0) {$1$}; \node[ch] (b) at (1,0) {$2$}; \node[ch] (c) at (2,0) {$3$}; \node[ch] (d) at (3,0) {$4$};
\foreach \x in {a,b,c,d} \draw[tr] (o)--(\x);
\draw[rd] (a)--(b); \draw[rd] (b)--(c); \draw[rd] (c)--(d);
\begin{scope}[xshift=5.5cm]
\node[ctr] (o) at (1.5,1.6) {$0$};
\node[ch] (a) at (0,0) {$1$}; \node[ch] (b) at (1,0) {$2$}; \node[ch] (c) at (2,0) {$3$}; \node[ch] (d) at (3,0) {$4$};
\draw[uu] (o)--(a); \draw[uu] (a)--(b); \draw[uu] (b)--(c); \draw[uu] (c)--(o);
\foreach \x in {a,b,c,d} \draw[tr] (o)--(\x);
\draw[rd] (a)--(b); \draw[rd] (b)--(c); \draw[rd] (c)--(d);
\node[lab] at (1.5,-0.6) {$\ast\{C_{12},C_{23}\}$: a $4$-cycle};
\end{scope}
\end{tikzpicture}
\caption{Example~\ref{ex:cone}. The cone over a path and its fundamental triangles. The unique
union of two adjacent triangles is the cycle through the apex.}\label{fig:cone}
\end{figure}

\begin{theorem}[\cite{Regular}]\label{thm:regular}
Every matroid admitting a cycle system is regular.
\end{theorem}

\subsection{Generalized cycle systems}

The question raised in \cite{CycleSystems} and pursued in \cite{PerkinsonDC} is what happens to
the recursion when $\CC$ is an arbitrary collection of $g$ cycles. The recursion stops at a pair $(M',\CC')$ as soon as the
total unique union of $\CC'$ is empty, and then the multicomplex must be supplied at that node
by other means. The following intermediate notion, due to the authors of \cite{CycleSystems}, describes a class in which the stopping nodes are trivial. The condition was observed in
\cite[Section~6.2]{CycleSystems} on a collection of four-cycles of $K_{3,3}$, and the definition was proposed by the authors \cite{Dochtermann,PerkinsonDC,Yi}.

\begin{definition}[\cite{Dochtermann,PerkinsonDC,Yi}]\label{def:gcs}
Let $M$ be a matroid of corank $g$. A collection $\CC=\{C_1,\dots,C_g\}$ of cycles of $M$ is a
\emph{generalized cycle system} if for every nonempty $S\subseteq[g]$ the unique union
$U_S$ is either dependent or a basis of the restriction $M|\bigcup_{i\in S}C_i$.
\end{definition}

Every cycle system is a generalized cycle system. The condition is preserved by the two moves of
\cite[Proposition 3.6]{CycleSystems}: deleting an admissible $e\in C_i$ removes $C_i$ and leaves
every other unique union and every other union unchanged. Contracting a non-loop $e$ keeps
dependent sets dependent. For a basis $U_S$ of $\bigcup_{i\in S}C_i$ there are three cases. If
$e\in U_S$, the contraction removes $e$ from $U_S$ and leaves a basis. If $e$ lies outside the
closure of $U_S$, the set $U_S$ stays a basis of the union in $M/e$. If $e$ lies in that closure,
which is the closure of the union, $U_S$ becomes dependent in $M/e$. Under a generalized cycle system the
recursion of \cite{CycleSystems} therefore still terminates at nodes whose multicomplex is the single
monomial $1$. At a pair whose total unique union is empty, every element of the union of the
cycles lies in at least two of them, and the condition ``dependent or a basis'' forces that union
to have rank zero. The node then contributes the single monomial $1$, provided every circuit of the node
lies in the union, as it does for the fundamental circuits of a basis (Lemma~\ref{lem:invariants}(d) below).

\begin{example}\label{ex:biconed-gcs}
Let $G$ be biconed with dominating edge $01$ and let $B_0$ be the breadth-first tree from $0$
in which every vertex is attached to $0$ if possible and to $1$ otherwise. The fundamental
circuits of $B_0$ form a generalized cycle system of $M(G)$ \cite{Yi}. In the language of Section~\ref{sec:radius2}, every dead node of a
biconed graph with this tree has rank zero.
\end{example}

\begin{example}[A generalized cycle system, and a graph that has none from its tree]\label{ex:gcs-two}
Let $G_1$ be the biconed graph on $0,1,2,3,4$ with edges $01,02,03,14$ (the tree) and $12,24,34$
(red), dominated by the edge $01$ (Figure~\ref{fig:gcstwo}). Its only dead stratum is $\sigma=\{12,24,34\}$, with
$U_\sigma=\{12,24,34,03\}$ spanning all five vertices, so $U_\sigma$ is a basis of the union of
the three circuits: the fundamental circuits of $B_0$ form a generalized cycle system, as
Example~\ref{ex:biconed-gcs} predicts. In the running example $G_0$ of Figure~\ref{fig:running},
on the other hand, the stratum $\sigma=\{24,35,45\}$ has $U_\sigma$ a path on the five vertices
$0,2,3,4,5$ while the union of its three circuits spans all six vertices. So $U_\sigma$ is
independent but not a basis of $M|A_\sigma$, the fundamental circuits of the tree $B_0$ of
$G_0$ do not form a generalized cycle system, and the recursion of \cite{CycleSystems} stops at
$\sigma$ with a nontrivial node. The next section treats exactly this situation.
\end{example}

\begin{figure}[h]\centering
\begin{tikzpicture}[scale=1]
\node[ctr] (o) at (0,2.2) {$0$};
\node[hub] (h1) at (-2,1) {$1$}; \node[hub] (h2) at (0,1) {$2$}; \node[hub] (h3) at (2,1) {$3$};
\node[ch] (c4) at (-1,-0.3) {$4$};
\draw[tr] (o)--(h1); \draw[tr] (o)--(h2); \draw[tr] (o)--(h3); \draw[tr] (h1)--(c4);
\draw[rd] (h1)--(h2); \draw[rd] (h2)--(c4); \draw[rd] (h3)--(c4);
\node[rlab] at (-1,1.2) {$12$}; \node[rlab] at (-0.75,0.45) {$24$}; \node[rlab] at (0.75,0.2) {$34$};
\begin{scope}[xshift=6cm]
\node[ctr] (o) at (0,2.2) {$0$};
\node[hub] (h1) at (-2,1) {$1$}; \node[hub] (h2) at (0,1) {$2$}; \node[hub] (h3) at (2,1) {$3$};
\node[ch] (c4) at (-1,-0.3) {$4$};
\draw[uu] (o)--(h3); \draw[uu] (h1)--(h2); \draw[uu] (h2)--(c4); \draw[uu] (h3)--(c4);
\draw[tr] (o)--(h1); \draw[tr] (o)--(h2); \draw[tr] (o)--(h3); \draw[tr] (h1)--(c4);
\draw[rd] (h1)--(h2); \draw[rd] (h2)--(c4); \draw[rd] (h3)--(c4);
\node[lab] at (0,-1.0) {$U_\sigma=\{12,24,34,03\}$ shaded};
\end{scope}
\end{tikzpicture}
\caption{Example~\ref{ex:gcs-two}. Left: the biconed graph $G_1$ with its tree. Right: the
unique union of the dead stratum $\sigma=\{12,24,34\}$, shaded, a spanning tree of $G_1$.}\label{fig:gcstwo}
\end{figure}
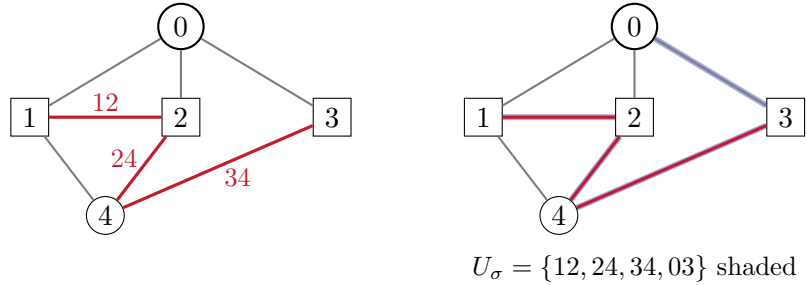

We now recall the two facts announced in the introduction. The first is Theorem~\ref{thm:regular}: by
Seymour's decomposition theorem \cite{Seymour}, matroids with cycle systems are built by $1$-, $2$- and $3$-sums from graphic matroids, cographic matroids and $R_{10}$.
The cographic case is Merino's theorem, so we look for an extension among graphic matroids. The second
is the following.

\begin{proposition}[Perkinson \cite{PerkinsonDC}]\label{prop:nogcs}
The Wagner graph $V_8$ and the Petersen graph admit no generalized cycle system consisting of
circuits. More strongly, no set of $g$ circuits of either graph admits a completable
deletion--contraction tree under any choice of admissible element at any node: this holds for all
$118{,}755$ five-element sets of circuits of $V_8$ and for all $36{,}288{,}252$ six-element sets of
circuits of the Petersen graph.
\end{proposition}

By Proposition~\ref{prop:nogcs}, generalized cycle systems of circuits do not cover all graphs. We design the relaxation of the next section with these graphs in mind, keeping the recursion of \cite{CycleSystems} on the strata
where it works and replacing the multicomplex only at the nodes where it stops.

\section{Relaxed coparking functions}\label{sec:model}

In this section $M$ is a matroid with a basis $B_0$, red set $R=E\setminus B_0$ and
fundamental circuits $C_i$. The reader may keep $M=M(G)$ and $B_0$ a spanning tree in mind, and
the running example is graphic, but only (F1) to (F3) and \eqref{eq:nullity} are used. Graphs
enter only in Section~\ref{sec:radius2}.

\subsection{Strata, coparking baselines and dead nodes}

\begin{definition}\label{def:strata}
A nonempty set $\sigma\subseteq R$ is a \emph{stratum}. It is \emph{dead} if $U_\sigma$ is
independent in $M$ (a forest of $G$ when $M=M(G)$) and \emph{live} otherwise. The words
independent and dependent are reserved for sets of elements of $M$, so ``$\sigma$ is dead'' is a
statement about $U_\sigma$, while ``$\sigma$ is a forest'' is a statement about $\sigma$ itself. The \emph{coparking baseline} of
$\sigma$, or baseline for short, is the vector $b_\sigma\in\N^\sigma$ with
\[
  b_\sigma(i)=|C_i\cap U_\sigma|\qquad(i\in\sigma).
\]
We identify $\N^\sigma$ with the set of vectors in $\N^R$ that vanish outside $\sigma$. As in Section~\ref{sec:coparking}, the stratum $\sigma$ \emph{can be fired} from a vector $a\in\N^R$ if $a|_\sigma\ge b_\sigma$
coordinatewise, and we then also say that $a$ \emph{dominates} $b_\sigma$.
\end{definition}

Since each element of $U_\sigma$ is counted in exactly one
$b_\sigma(i)$,
\begin{equation}\label{eq:bsize}
  |b_\sigma|:=\sum_{i\in\sigma}b_\sigma(i)=|U_\sigma| .
\end{equation}

Coparking baselines are weakly monotone in the stratum: shrinking $\sigma$ can make the baseline larger, because elements shared by several circuits of $\sigma$ may become uniquely owned inside a smaller stratum $\tau$. For
$\tau\subseteq\sigma$ write
\[
  \Delta^\sigma_\tau=U_\tau\setminus U_\sigma
  =\{e\in E: m_\tau(e)=1,\ m_\sigma(e)\ge2\},
  \qquad
  \exc_\tau=b_\tau-b_\sigma|_\tau\in\Z^\tau ,
\]
the elements owned in $\tau$ but not in $\sigma$, and the \emph{excess}. The excess is the room
left before $\tau$ can be fired: at the baseline $b_\sigma$ the stratum $\sigma$ can be fired but
$\tau$ is short by $\exc_\tau$, and $\tau$ can be fired from $b_\sigma+m$ exactly when the offset $m$ is large enough to
cover the excess, $m|_\tau\ge\exc_\tau$. The superscript on $\Delta$ is
suppressed when $\sigma$ is clear.

\begin{lemma}[Monotonicity of $b_\sigma$]\label{lem:mono}
Let $\sigma\subseteq R$ be a stratum and $\tau\subseteq\sigma$ a stratum inside it.
\begin{enumerate}[label=(\alph*)]
\item $b_\sigma(i)\ge1$ for every $i\in\sigma$. Consequently only strata $\tau$ inside the support of
      $a$ can be fired from $a$.
\item $U_\tau=(U_\sigma\cap A_\tau)\ \sqcup\ \Delta_\tau$.
\item $\exc_\tau(i)=|\Delta_\tau\cap C_i|\ge0$ for every $i\in\tau$, in particular
      $b_\sigma|_\tau\le b_\tau$.
\item The following are equivalent: $\tau$ can be fired from $b_\sigma$, $b_\sigma|_\tau=b_\tau$,
      $\Delta_\tau=\emptyset$, and $U_\tau\subseteq U_\sigma$. When they hold and $\sigma$ is
      dead, $\tau$ is dead.
\end{enumerate}
\end{lemma}
\begin{proof}
Part (a) is (F1): $i\in C_i\cap U_\sigma$.

For (b), if $e \in U_\sigma \cap A_\tau$ then the unique owner of $e$ in $\sigma$ lies in $\tau$. Hence $U_\sigma \cap A_\tau$ is exactly the set of elements of $U_\tau$ with $m_\sigma(e) = 1$. As every $e$ in $U_\tau$ has $m_\sigma(e) \geq 1$, the set $\Delta_\tau$ is exactly the set of elements of $U_\tau$ with $m_\sigma(e) \geq 2$, so $U_\sigma \cap A_\tau$ and $\Delta_\tau$ partition $U_\tau$.

For (c), fix $i \in \tau$. From $C_i \subseteq A_\tau$, we get $C_i \cap U_\sigma \cap A_\tau = C_i \cap U_\sigma$. Together with (b) we get
$$b_\tau(i) = |C_i \cap U_\tau| = |C_i \cap U_\sigma| + |C_i \cap \Delta_\tau| = b_\sigma(i) + |C_i \cap \Delta_\tau|.$$

For (d), part (c) shows that $\tau$ can be fired from $b_\sigma$ if and only if $b_\sigma|_\tau = b_\tau$. By definition each $e \in \Delta_\tau$ lies in a unique $C_i$ with $i \in \tau$, so
$$\sum_{i\in\tau}|\Delta_\tau\cap C_i|=|\Delta_\tau|.$$
Hence the excess is $0$ if and only if $\Delta_\tau = \emptyset$, which by definition means $U_\tau \subseteq U_\sigma$. Finally, a subset of an independent set is independent.
\end{proof}

The converse of the last clause of Lemma~\ref{lem:mono}(d) is false: $\tau$ may be dead inside a dead $\sigma$ with $\exc_\tau\neq0$. In $G_0$ the strata $\sigma=\{12,24,35,45\}$ and $\tau=\{12,35,45\}$ are both dead. Their baselines are $b_\sigma=(1,1,2,1)$ and $b_\tau=(1,2,2)$, so $\exc_\tau=(0,0,1)$. The tree edge $14$ lies in $C_{24}$ and $C_{45}$, so inside $\tau$ it is owned by $45$ alone and $14\in\Delta_\tau$.

Every singleton $\{i\}$ is live because $U_{\{i\}}=C_i$ is a circuit. For a graph
(or any binary matroid) every pair $\{i,j\}$ is live as well, because $U_{\{i,j\}}=C_i\,\triangle\,C_j$
is a nonzero element of the cycle space. In a general matroid a pair can be dead, as
the uniform matroid $U_{2,4}$ with basis $\{1,2\}$ shows, where the unique union of $\{3,4\}$ is
$\{3,4\}$ itself. The live strata are the ones that must never be fired, as in the coparking condition of Section~\ref{sec:cs}, and
the dead strata are exactly where the recursion of \cite{CycleSystems} stops. At a dead
node every element lies in at least two of the circuits $C_i\setminus U_\sigma$ with $i\in\sigma$.
Definition~\ref{def:gcs} allows this situation only when the rank is zero, and the relaxation
of this paper allows any rank.

\begin{definition}\label{def:deadnode}
For a dead stratum $\sigma$ the \emph{dead node} is the matroid
\[
  K_\sigma=(M|A_\sigma)/U_\sigma
\]
on the ground set $A_\sigma\setminus U_\sigma$, the elements lying in at least two of the
circuits $C_i$, $i\in\sigma$.
\end{definition}

\begin{lemma}\label{lem:rank}
For a dead stratum $\sigma$, the matroid $K_\sigma$ has no coloops, so the degree of
$h(K_\sigma)$ is $r(K_\sigma)$, and
\[
  r(K_\sigma)=r(A_\sigma)-|U_\sigma|=|A_\sigma|-|\sigma|-|U_\sigma| ,
  \qquad\text{so}\qquad |b_\sigma|+r(K_\sigma)=r(A_\sigma).
\]
\end{lemma}
\begin{proof}
Since $\sigma$ is dead, $U_\sigma$ is independent, and contracting an independent set lowers
the rank by exactly its size, which gives the first equality. The second equality is
\eqref{eq:nullity} and the last is \eqref{eq:bsize}. For the coloops, note that
$r_{K_\sigma}(X)=r(X\cup U_\sigma)-|U_\sigma|$. So an element $y$ is a coloop of $K_\sigma$ if and
only if it is a coloop of $M|A_\sigma$. But $y$ lies in the circuit $C_i$ of $M|A_\sigma$ for
some $i\in\sigma$, so it is not one. The degree of the $h$-vector of a
matroid is its rank minus the number of coloops (Section~\ref{sec:prelim}).
\end{proof}

\begin{example}[A dead node of rank one]\label{ex:deadnode}
In $G_0$ take $\sigma=\{24,35,45\}$ (Figure~\ref{fig:running}). The coparking baseline is
$b_\sigma=(b_{24},b_{35},b_{45})=(1,2,1)$: for instance $C_{35}\cap U_\sigma=\{35,03\}$. The
ground set of $K_\sigma$ is $A_\sigma\setminus U_\sigma=\{01,02,14,25\}$, the tree edges lying in
at least two of the three circuits. Contracting the path $U_\sigma$ identifies the vertices
$0,2,3,4,5$ into one vertex $w$ (Figure~\ref{fig:deadnode}): the edges $02$ and $25$ become
loops, and $01$ and $14$ become two parallel edges between $w$ and $1$. So $K_\sigma$ has rank
one, $h(K_\sigma)=(1,1)$, and $|b_\sigma|+r(K_\sigma)=4+1=5=r(A_\sigma)$ as
Lemma~\ref{lem:rank} says. The other two dead strata of $G_0$, namely $\{12,35,45\}$ and
the full set $\{12,24,35,45\}$, have dead nodes of rank zero.
\end{example}

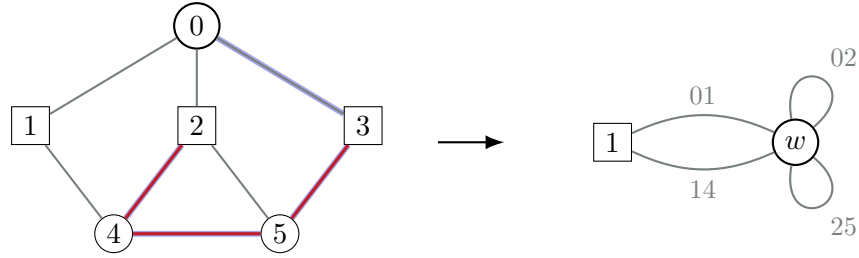
\begin{figure}[h]\centering
\begin{tikzpicture}[scale=1.1]
\node[ctr] (o) at (0,2.2) {$0$};
\node[hub] (h1) at (-2,1) {$1$}; \node[hub] (h2) at (0,1) {$2$}; \node[hub] (h3) at (2,1) {$3$};
\node[ch] (c4) at (-1,-0.3) {$4$}; \node[ch] (c5) at (1,-0.3) {$5$};
\draw[uu] (o)--(h3); \draw[uu] (h2)--(c4); \draw[uu] (h3)--(c5); \draw[uu] (c4)--(c5);
\draw[tr] (o)--(h1); \draw[tr] (o)--(h2); \draw[tr] (o)--(h3); \draw[tr] (h1)--(c4); \draw[tr] (h2)--(c5);
\draw[rd] (h2)--(c4); \draw[rd] (h3)--(c5); \draw[rd] (c4)--(c5);
\node[lab] at (0,-1.0) {contract the shaded forest $U_\sigma$, delete the edge $12$ outside $A_\sigma$};
\draw[-{Latex[length=2.5mm]},thick] (2.9,0.8)--(3.7,0.8);
\begin{scope}[xshift=7.2cm]
\node[ctr] (S) at (0,0.8) {$w$};
\node[hub] (h1) at (-2.2,0.8) {$1$};
\draw[tr] (h1) to[bend left=25] node[above,lab] {$01$} (S);
\draw[tr] (h1) to[bend right=25] node[below,lab] {$14$} (S);
\draw[tr] (S) to[out=40,in=100,looseness=8] node[above right,lab] {$02$} (S);
\draw[tr] (S) to[out=-40,in=-100,looseness=8] node[below right,lab] {$25$} (S);
\node[lab] at (-1.1,-1.0) {$K_\sigma$: rank $1$, $h=(1,1)$};
\end{scope}
\end{tikzpicture}
\caption{Example~\ref{ex:deadnode}. The dead node $K_\sigma$ of $\sigma=\{24,35,45\}$ in $G_0$:
one parallel class of size two and two loops.}\label{fig:deadnode}
\end{figure}

\subsection{Fibres and the relaxed coparking functions}

At each dead stratum $\sigma$ we are going to choose a set $F(\sigma)$ of exponent vectors, the
\emph{fibre} at $\sigma$, which describes the vectors that are allowed to sit above the coparking
baseline $b_\sigma$. The definition is recursive along inclusion: when the fibre at
$\sigma$ is chosen, fibres have already been chosen at every dead $\tau\subsetneq\sigma$.

The compatibility condition uses the following notion. Let $\mathcal T$ be a family of strata at whose dead members fibres $F(\tau)$ have been chosen. A vector $a\in\N^R$ is \emph{coparking on $\mathcal T$} if every $\tau\in\mathcal T$ that can be fired from $a$
is dead, with $a|_\tau-b_\tau\in F(\tau)$. When the strata in $\mathcal T$ are live this is the coparking condition of Section~\ref{sec:coparking}, that none of them can be fired from $a$, and in general it is that condition relaxed at the dead strata by the fibres chosen there. For a stratum $\sigma$ write $\mathring\sigma$ for the set of proper strata $\tau\subsetneq\sigma$, the \emph{interior} of $\sigma$.

\begin{definition}\label{def:fibre}
Let $\sigma$ be a dead stratum and suppose fibres $F(\tau)$ have been chosen at all dead strata
$\tau\subsetneq\sigma$. A \emph{fibre} at $\sigma$ is a set $F(\sigma)\subseteq\N^\sigma$ of
\emph{offsets} such that
\begin{enumerate}[label=(\roman*)]
\item $F(\sigma)$ is a pure multicomplex with degree sequence $h(K_\sigma)$, in particular $0\in F(\sigma)$ and every maximal offset has degree $r(K_\sigma)$.
\item (\emph{coparking condition}) for every $m\in F(\sigma)$ the vector $b_\sigma+m$ is coparking on $\mathring\sigma$: every proper stratum $\tau\subsetneq\sigma$ that can be fired from $b_\sigma+m$ is dead, with $(b_\sigma+m)|_\tau-b_\tau\in F(\tau)$.
\end{enumerate}
\end{definition}

The vector $b_\sigma+m$ is supported in $\sigma$, so by Lemma~\ref{lem:mono}(a) only $\sigma$ and its subsets can be fired from it, and $\sigma$ itself is fired with offset $m\in F(\sigma)$. Condition (ii) therefore says that the fibre written down at $\sigma$ is compatible with the fibres already chosen inside $\sigma$, and it refers to nothing else. Whether a set is a fibre at $\sigma$ depends only on the fibres chosen strictly inside $\sigma$, and the fibres may be chosen in any order compatible with inclusion, for instance in order of increasing $|\sigma|$. Figure~\ref{fig:fibre} shows a fibre with $h(K_\sigma)=(1,2,2)$.

\begin{figure}[h]\centering
\begin{tikzpicture}[scale=0.85]
\begin{scope}
\fill[grey!18] (-0.5,-0.5) rectangle (4.5,1.5);
\fill[grey!18] (-0.5,1.5) rectangle (1.5,4.5);
\fill[pattern=north east lines,pattern color=redge!60] (2.5,2.5) rectangle (4.5,4.5);
\draw[redge,thick] (2.5,2.5) -- (4.5,2.5) (2.5,2.5) -- (2.5,4.5);
\foreach \x in {0,...,4} \foreach \y in {0,...,4} \fill[grey!60] (\x,\y) circle (1.2pt);
\draw[-{Stealth}] (-0.5,-0.5) -- (4.9,-0.5) node[right,lab] {$a_1$};
\draw[-{Stealth}] (-0.5,-0.5) -- (-0.5,4.9) node[above,lab] {$a_4$};
\foreach \x in {0,...,4} \node[lab,below] at (\x,-0.5) {$\x$};
\foreach \y in {0,...,4} \node[lab,left] at (-0.5,\y) {$\y$};
\draw[grey,dashed] (2,3) -- (3,2); \draw[grey,dashed] (2,4) -- (4,2);
\fill[black] (2,2) circle (3pt); \node[lab,below left=-1pt] at (2,2) {$b_\sigma$};
\foreach \p/\l/\pos in {(3,2)/$x_1$/below, (4,2)/$x_1^2$/below, (2,3)/$x_4$/left, (2,4)/$x_4^2$/left}
  {\fill[blue!70!black] \p circle (3pt); \node[lab,\pos=2pt] at \p {\l};}
\node[lab,align=left,fill=white,inner sep=2pt] at (3.5,3.5) {$\tau=\{e_1,e_4\}$ can be fired};
\node[lab,grey!50!black,align=center] at (3.0,0.5) {$\sigma$ cannot be fired};
\node[lab,grey!50!black,rotate=90] at (0.5,3.0) {$\sigma$ cannot be fired};
\node[lab] at (2.5,-1.4) {a fibre $F(\sigma)$ above $b_\sigma$ (Section~\ref{sec:petersen})};
\end{scope}
\end{tikzpicture}
\caption{A fibre at the stratum $\sigma=\word{ABCAB}$ of the Petersen root
(Section~\ref{sec:petersen}), with red edges $e_1,\dots,e_4$, coparking baseline
$b_\sigma=(2,1,1,2)$ and $h(K_\sigma)=(1,2,2)$, drawn in the coordinates $a_1,a_4$ of a vector
with $a_2=a_3=1$. From the vectors in the grey region $\sigma$ cannot be fired. The hatched region is excluded by condition (ii), since the live stratum
$\tau=\{e_1,e_4\}$ can be fired from its vectors. The blue points are the offsets from $b_\sigma$ forming the fibre $F(\sigma)=\{1,x_1,x_4,x_1^2,x_4^2\}$, a pure multicomplex with degree sequence $(1,2,2)$.}
\label{fig:fibre}
\end{figure}
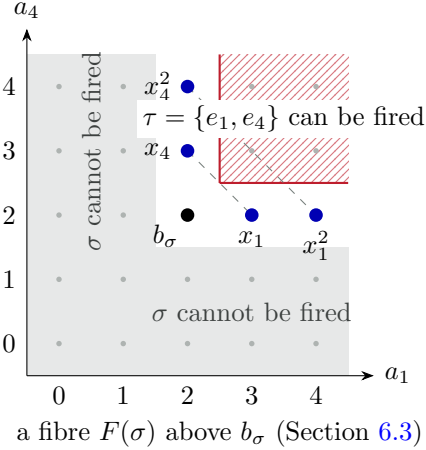

A choice of a fibre at every dead stratum is a \emph{fibre system} $F$ on $(M,B_0)$. A dead stratum at which no fibre exists, for the fibres chosen at the dead strata strictly inside it, is called \emph{obstructed}. So $(M,B_0)$ carries a fibre system if and only if the fibres can be chosen along inclusion without meeting an obstructed stratum.

The following lemma gives a sufficient condition for a fibre that does not depend on the fibres
chosen below $\sigma$.

\begin{lemma}[Baseline-only lemma]\label{lem:baseline}
Let $\sigma$ be a dead stratum and let $F(\sigma)\subseteq\N^\sigma$ be a pure multicomplex with
degree sequence $h(K_\sigma)$ such that for every $m\in F(\sigma)$ and every nonempty $\tau\subsetneq\sigma$
that can be fired from $b_\sigma+m$, the stratum $\tau$ is dead and
$(b_\sigma+m)|_\tau=b_\tau$.
Then $F(\sigma)$ is a fibre at $\sigma$ for every fibre system on the strata below $\sigma$.
\end{lemma}
\begin{proof}
Condition (i) is the hypothesis. For (ii), let $m\in F(\sigma)$ and let $\tau\subsetneq\sigma$ be a proper stratum that can be fired from $b_\sigma+m$. By assumption $\tau$ is dead and the offset $(b_\sigma+m)|_\tau-b_\tau$ is $0$, which lies in $F(\tau)$ whatever fibre was chosen there. So $b_\sigma+m$ is coparking on $\mathring\sigma$.
\end{proof}

We call an $F(\sigma)$ satisfying the hypothesis of Lemma~\ref{lem:baseline} an
\emph{unconditional fibre} at $\sigma$: every proper stratum that can be fired from $b_\sigma+m$ is dead and is fired exactly, with offset $0$.

\begin{lemma}[Rank-zero strata]\label{lem:zerofibre}
Let $\sigma$ be a dead stratum with $r(K_\sigma)=0$. Then $\{0\}$ is an unconditional fibre at
$\sigma$, and it is the only fibre there.
\end{lemma}
\begin{proof}
As $r(K_\sigma)=0$ we have $h(K_\sigma)=(1)$, and $\{0\}$ is the only pure multicomplex with
that degree sequence, so it is the only candidate. It satisfies the hypothesis of
Lemma~\ref{lem:baseline}: if $\tau$ can be fired from $b_\sigma+0$ then $b_\sigma|_\tau=b_\tau$ and
$\tau$ is dead by Lemma~\ref{lem:mono}(d).
\end{proof}

On $G_0$ the two dead strata of rank zero receive the unconditional fibre $\{0\}$. The third
dead stratum $\sigma=\{24,35,45\}$ has $h(K_\sigma)=(1,1)$ by Example~\ref{ex:deadnode}, and
$\{0,\mathbf e_{35}\}$ is an unconditional fibre there. Indeed the singletons have baselines
$|C_i|\ge4$ and the pairs inside $\sigma$ have baselines $(2,3)$ on $\{35,45\}$, $(1,2)$ on
$\{24,45\}$ and $(3,3)$ on $\{24,35\}$, so no proper stratum can be fired from
$b_\sigma=(1,2,1)$ or from $b_\sigma+\mathbf e_{35}=(1,3,1)$.

\begin{definition}\label{def:relaxed}
Let $F$ be a fibre system on $(M,B_0)$. The set $\PP^*(F)\subseteq\N^R$ of \emph{relaxed coparking functions} consists of the vectors that are coparking on the set of all strata: the vectors $a$ such that every stratum $\tau$ that can be fired from $a$ is dead, with $a|_\tau-b_\tau\in F(\tau)$.
\end{definition}

On the live strata this is the coparking condition of \cite{CycleSystems}, that none of them can be fired from $a$, and on the dead strata it is the relaxation. The main theorem of the paper is the following.

\begin{theorem}\label{thm:main}
Let $M$ be a matroid with a basis $B_0$ and let $F$ be a fibre system on $(M,B_0)$. Then
$\PP^*(F)$ is a pure multicomplex with degree sequence $h(M)$. Consequently, if $(M,B_0)$
carries a fibre system, Stanley's conjecture holds for $M$.
\end{theorem}

The ingredients of the proof occupy the rest of this section: closure under division is Lemma~\ref{lem:C}, the degree sequence is Theorem~\ref{thm:B} together with Lemma~\ref{lem:D}, and purity is Corollary~\ref{cor:extpure}.

\subsection{Closure and the decomposition}

\begin{lemma}[Closure]\label{lem:C}
For every family $\mathcal T$ of strata, the vectors coparking on $\mathcal T$ are closed under division. In particular $\PP^*(F)$ is closed under division, and a multicomplex $F(\sigma)\subseteq\N^\sigma$ satisfies condition (ii) of Definition~\ref{def:fibre} if and only if $b_\sigma+m$ is coparking on $\mathring\sigma$ for every maximal $m\in F(\sigma)$.
\end{lemma}
\begin{proof}
Let $a$ be coparking on $\mathcal T$ and $a'\le a$. A stratum $\tau\in\mathcal T$ that can be fired from $a'$ can be fired from $a$, so it is dead with $a|_\tau-b_\tau\in F(\tau)$, and $0\le a'|_\tau-b_\tau\le a|_\tau-b_\tau$ lies in $F(\tau)$ because $F(\tau)$ is closed under division. Hence $a'$ is coparking on $\mathcal T$. The case in which $\mathcal T$ is the family of all strata says that $\PP^*(F)$ is closed under division. For the last claim, every $m\in F(\sigma)$ lies below a maximal $m'$, and $b_\sigma+m\le b_\sigma+m'$.
\end{proof}

The next lemma shows that membership in $\PP^*(F)$ is decided by a single stratum, the largest one
that can be fired from $a$.

\begin{lemma}\label{lem:D}
Let $a\in\N^R$ and let $\mathrm{Dom}(a)$ be the set of strata $\tau$ that can be fired from $a$.
\begin{enumerate}[label=(\alph*)]
\item $\mathrm{Dom}(a)$ is closed under union, so when it is nonempty it has a largest member
      $\max\mathrm{Dom}(a)$, the union of all of its members.
\item $a\in\PP^*(F)$ if and only if $\mathrm{Dom}(a)$ is empty, or $\mu=\max\mathrm{Dom}(a)$ is dead and
      $a|_{\mu}-b_{\mu}\in F(\mu)$.
\item For $\tau\in\mathrm{Dom}(a)$ and a nonempty $\rho\subseteq R\setminus\tau$, the stratum
      $\tau\cup\rho$ can be fired from $a$ if and only if $a|_\rho\ge b_{\tau\cup\rho}|_\rho$.
\end{enumerate}
\end{lemma}
\begin{proof}
For (a), let $\tau_1,\tau_2\in\mathrm{Dom}(a)$. For $i\in\tau_1$,
Lemma~\ref{lem:mono} applied to $\tau_1\subseteq\tau_1\cup\tau_2$ gives
$b_{\tau_1\cup\tau_2}(i)\le b_{\tau_1}(i)\le a_i$, and similarly for $i\in\tau_2$. So $\tau_1\cup\tau_2$ can be fired from $a$.

For (c), if $\tau\in\mathrm{Dom}(a)$ then
$b_{\tau\cup\rho}|_\tau\le b_\tau\le a|_\tau$ by Lemma~\ref{lem:mono}, so whether $\tau\cup\rho$ can be fired from $a$ is decided on $\rho$ alone.

For (b), if $\mathrm{Dom}(a)$ is empty then $a\in\PP^*(F)$ vacuously, so assume it is not and put $\mu=\max\mathrm{Dom}(a)$. If $a\in\PP^*(F)$ then Definition~\ref{def:relaxed} applied to $\mu$ says that $\mu$ is dead with $a|_\mu-b_\mu\in F(\mu)$. Conversely, put $m=a|_\mu-b_\mu\in F(\mu)$. For any $\tau\in\mathrm{Dom}(a)$ we have $\tau\subseteq\mu$ and $a|_\tau=(b_\mu+m)|_\tau$. If $\tau=\mu$ then $\tau$ is dead with offset $m\in F(\mu)$. Otherwise $\tau$ is a proper stratum of $\mu$ that can be fired from $b_\mu+m$, so condition (ii) at $\mu$ gives that $\tau$ is dead with $a|_\tau-b_\tau\in F(\tau)$. As this holds for every $\tau\in\mathrm{Dom}(a)$, $a\in\PP^*(F)$.
\end{proof}

\begin{example}\label{ex:dom}
Let $F$ be any fibre system on $(M(G_0),B_0)$, write vectors as $(a_{12},a_{24},a_{35},a_{45})$ and let $\sigma=\{24,35,45\}$, so that $b_\sigma=(1,2,1)$ and $b_R=(1,1,2,1)$. From $a=(1,1,2,1)$ both $\sigma$ and $R$ can be fired, so $\mathrm{Dom}(a)=\{\sigma,R\}$ and $\max\mathrm{Dom}(a)=R$, and $a\in\PP^*(F)$ as $a-b_R=0\in F(R)$. For $a'=(2,1,2,1)$ the strata that can be fired are the same, but $a'-b_R=\mathbf e_{12}\notin F(R)=\{0\}$, so $a'\notin\PP^*(F)$ although its offset at $\sigma$ is still $0\in F(\sigma)$. Part (c) with $\tau=\sigma$, $\rho=\{12\}$ says that when $\sigma$ can be fired from a vector, $R$ can be fired from it if and only if $a_{12}\ge b_R(12)=1$, which is why $\max\mathrm{Dom}$ is $\sigma$ for $(0,1,2,1)$ and $R$ for $(1,1,2,1)$.
\end{example}

\begin{definition}\label{def:ext}
For $\tau\subseteq R$ (possibly empty) the \emph{extension set} $\Ext(\tau)$ is the set of
$c\in\N^{R\setminus\tau}$ such that no stratum strictly containing $\tau$ can be fired from the
vector $(b_\tau,c)$, which agrees with $b_\tau$ on $\tau$ and with $c$ on $R\setminus\tau$. Its
generating function is $E_\tau(t)=\sum_{c\in\Ext(\tau)}t^{|c|}$.
\end{definition}

So $\Ext(\tau)$ is the set of coparking functions relative to $\tau$: sit at the baseline of
$\tau$ and ask which values on the remaining coordinates let no larger stratum be fired. By
Lemma~\ref{lem:D}(c), or by Definition~\ref{def:strata} when $\tau=\emptyset$, the condition
reads
\[
  \Ext(\tau)=\bigl\{\,c\in\N^{R\setminus\tau}:\ \text{no nonempty }\rho\subseteq R\setminus\tau
  \text{ has } c|_\rho\ge b_{\tau\cup\rho}|_\rho\,\bigr\},
\]
so it depends only on the coparking baselines of the supersets of $\tau$, not on any fibre.
Lemma~\ref{lem:D} gives
\begin{equation}\label{eq:decomp}
  \PP^*(F)=\bigsqcup_{\tau}\ \bigl(b_\tau+F(\tau)\bigr)\times\Ext(\tau),
\end{equation}
the disjoint union over $\tau=\emptyset$ and all dead strata $\tau$, where $a\in\PP^*(F)$ is
sent to the pair $(a|_\tau,a|_{R\setminus\tau})$ with $\tau=\max\mathrm{Dom}(a)$, and $\tau=\emptyset$ when $\mathrm{Dom}(a)$ is empty. The map is onto: for $m\in F(\tau)$ and $c\in\Ext(\tau)$ put $a=(b_\tau+m,c)$. If
$\tau=\emptyset$ then $a=c$, and $c\in\Ext(\emptyset)$ says that no stratum can be fired from $a$,
so $\mathrm{Dom}(a)$ is empty and $a\in\PP^*(F)$. Otherwise $\tau\in\mathrm{Dom}(a)$, and for
$\tau'\in\mathrm{Dom}(a)$ we put $\rho=\tau'\setminus\tau$. By Lemma~\ref{lem:D}(a) the union
$\tau\cup\rho=\tau\cup\tau'$ lies in $\mathrm{Dom}(a)$, so if $\rho$ were nonempty,
Lemma~\ref{lem:D}(c) would give $c|_\rho\ge b_{\tau\cup\rho}|_\rho$, which $c\in\Ext(\tau)$
forbids. Hence every member of $\mathrm{Dom}(a)$ lies inside $\tau$, so $\max\mathrm{Dom}(a)=\tau$
and $a\in\PP^*(F)$ by Lemma~\ref{lem:D}(b). In both cases $a$ is sent to $(b_\tau+m,c)$. Consequently
\begin{equation}\label{eq:degvec}
  \sum_{a\in\PP^*(F)}t^{|a|}=\sum_{\tau}t^{|b_\tau|}\,h(K_\tau)\,E_\tau(t),
\end{equation}
where $h(K_\emptyset)=1$ and $b_\emptyset=0$.

\begin{example}[An extension set]\label{ex:ext}
On $G_0$ the three dead strata all have $\Ext(\tau)=\{0\}$, because every red edge outside $\tau$ has coparking baseline $1$ in the full stratum, so $G_0$ itself shows nothing here and we enlarge it for this example. Add the red edge $04$ to $G_0$, keeping the
tree, so that $C_{04}=\{04,01,14\}$. The stratum $\tau=\{04,12,24\}$ has $U_\tau=\{04,12,24\}$, the path $1\,2\,4\,0$, so it is dead (Figure~\ref{fig:ext}), and $R\setminus\tau=\{35,45\}$. The three supersets give the conditions: $b_{\tau\cup\{35\}}(35)=3$ (as $C_{35}\cap U_{\tau\cup\{35\}}=\{35,25,03\}$),
$b_{\tau\cup\{45\}}(45)=2$ (as $C_{45}\cap U_{\tau\cup\{45\}}=\{45,25\}$), and
$b_{\tau\cup\{35,45\}}|_{\{35,45\}}=(2,1)$ (the edge $25$ is now shared by $C_{35}$ and $C_{45}$). Hence
\begin{align*}
  \Ext(\tau)&=\{(c_{35},c_{45}):\ c_{35}\not\ge 3,\ c_{45}\not\ge 2,\ (c_{35},c_{45})\not\ge(2,1)\}\\
  &=\{(0,0),(0,1),(1,0),(1,1),(2,0)\},
\end{align*}
and $E_\tau(t)=1+2t+2t^2$.
\end{example}

\begin{figure}[h]\centering
\begin{tikzpicture}[scale=0.85]
\begin{scope}[yshift=0.2cm]
\node[ctr] (o) at (0,2.2) {$0$};
\node[hub] (h1) at (-2,1) {$1$}; \node[hub] (h2) at (0,1) {$2$}; \node[hub] (h3) at (2,1) {$3$};
\node[ch] (c4) at (-1,-0.3) {$4$}; \node[ch] (c5) at (1,-0.3) {$5$};
\draw[uu] (h1)--(h2); \draw[uu] (h2)--(c4); \draw[uu] (o) to[bend right=40] (c4);
\draw[tr] (o)--(h1); \draw[tr] (o)--(h2); \draw[tr] (o)--(h3); \draw[tr] (h1)--(c4); \draw[tr] (h2)--(c5);
\draw[rd] (h1)--(h2); \draw[rd] (h2)--(c4); \draw[rd] (h3)--(c5); \draw[rd] (c4)--(c5); \draw[rd] (o) to[bend right=40] (c4);
\node[rlab] at (-1,1.2) {$12$}; \node[rlab] at (-0.75,0.45) {$24$}; \node[rlab] at (1.75,0.35) {$35$}; \node[rlab] at (0,-0.55) {$45$}; \node[rlab] at (-1.55,1.85) {$04$};
\node[lab] at (0,-1.3) {$\tau=\{04,12,24\}$, $U_\tau=\tau$ shaded};
\end{scope}
\begin{scope}[xshift=7.2cm,yshift=-0.9cm]
\fill[pattern=north east lines,pattern color=redge!60] (2.5,-0.5) rectangle (4.5,3.5);   
\fill[pattern=north east lines,pattern color=redge!60] (-0.5,1.5) rectangle (4.5,3.5);   
\fill[pattern=north east lines,pattern color=redge!60] (1.5,0.5) rectangle (4.5,3.5);    
\draw[redge,thick] (2.5,-0.5) -- (2.5,0.5) -- (1.5,0.5) -- (1.5,1.5) -- (-0.5,1.5);
\foreach \x in {0,...,4} \foreach \y in {0,...,3} \fill[grey!60] (\x,\y) circle (1.2pt);
\draw[-{Stealth}] (-0.5,-0.5) -- (4.9,-0.5) node[right,lab] {$c_{35}$};
\draw[-{Stealth}] (-0.5,-0.5) -- (-0.5,3.9) node[above,lab] {$c_{45}$};
\foreach \x in {0,...,4} \node[lab,below] at (\x,-0.5) {$\x$};
\foreach \y in {0,...,3} \node[lab,left] at (-0.5,\y) {$\y$};
\foreach \p in {(0,0),(0,1),(1,0),(1,1),(2,0)} \fill[blue!70!black] \p circle (3pt);
\node[lab,fill=white,inner sep=2pt] at (3.5,2.6) {$\{35\}$, $\{45\}$ or $\{35,45\}$ fires};
\node[lab] at (2.0,-1.4) {the extension set $\Ext(\tau)$ (Example~\ref{ex:ext})};
\end{scope}
\end{tikzpicture}
\caption{Example~\ref{ex:ext}. Left: $G_0$ with the red edge $04$ added and the dead stratum
$\tau=\{04,12,24\}$ shaded. Right: the extension set
$\Ext(\tau)$ in the coordinates $(c_{35},c_{45})$, with the excluded region hatched.}
\label{fig:ext}
\end{figure}
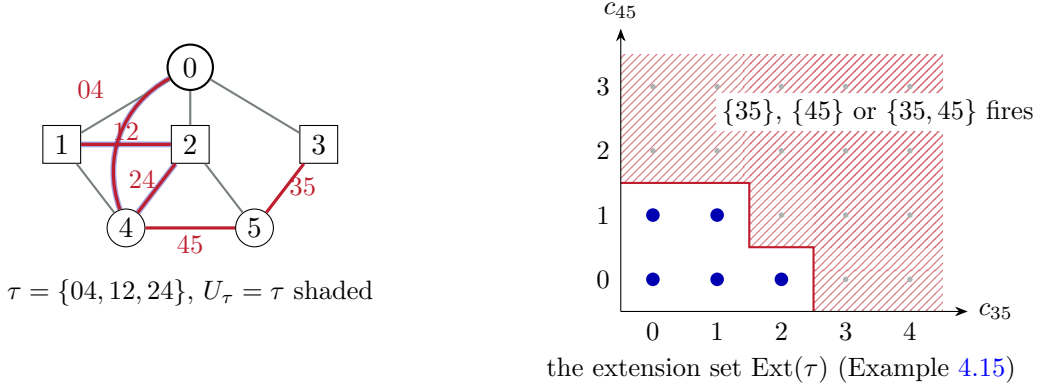

\subsection{Extension sets are coparking sets}\label{sec:extpure}

By definition,
$\Ext(\emptyset)=\PP^*(\CC)$ is the set of coparking functions of $\CC$ in the sense of
Section~\ref{sec:cs}, whether or not $\CC$ is a cycle system. We call a set of coparking functions of this kind a \emph{coparking set}, and the first lemma says that every
extension set is one.

\begin{lemma}[Extension sets by contraction]\label{lem:extcontract}
Let $\tau\subseteq R$ and put $M'=M/A_\tau$, $B_0'=B_0\setminus A_\tau$. Then $B_0'$ is a basis
of $M'$ with red set $R\setminus\tau$, its fundamental circuits are $\CC'=(C_j\setminus A_\tau)_{j\in R\setminus\tau}$, and
\[
  \Ext_{(M,B_0)}(\tau)=\PP^*(\CC').
\]
\end{lemma}
\begin{proof}
By \eqref{eq:nullity}, $B_0\cap A_\tau$ is a basis of
$M|A_\tau$, so $B_0'=B_0\setminus A_\tau$ is a basis of $M'$ \cite[Prop.~3.1.7]{Oxley}, with red
set $R\setminus\tau$. For $j\in R\setminus\tau$, \eqref{eq:nullity} applied to $\tau$ and to
$\tau\cup\{j\}$ gives $r_{M'}(C_j\setminus A_\tau)=r(A_{\tau\cup\{j\}})-r(A_\tau)=|C_j\setminus A_\tau|-1$,
so $C_j\setminus A_\tau$ contains a unique circuit $C'$ of $M'$, which is the fundamental circuit
of $j$. As $C'\cup(B_0\cap A_\tau)$ is dependent in $M$ and lies in $B_0\cup\{j\}$, it contains
$C_j$, so $C'=C_j\setminus A_\tau$.

Finally let $j\in\rho\subseteq R\setminus\tau$. The set $C_j\cap U_{\tau\cup\rho}$ consists of the elements of $C_j$ that lie in no $C_i$ with $i\in\tau$ and in exactly one $C_{j'}$ with $j'\in\rho$. That is,
$$C_j\cap U_{\tau\cup\rho}=(C_j\setminus A_\tau)\cap\ast\{C_{j'}\setminus A_\tau:j'\in\rho\}.$$
So the coparking baselines of $(M',B_0')$ are the restrictions $b_{\tau\cup\rho}|_\rho$, and the two sets coincide.
\end{proof}

So an extension set is the set of coparking functions of a smaller matroid with a fixed basis, and it remains to establish purity and the maximal degree for such sets.

\begin{figure}[h]\centering
\begin{tikzpicture}
\begin{scope}
\draw[line width=5pt,grey!25,line cap=round] (0,2.2)--(0,1);
\draw[line width=5pt,grey!25,line cap=round] (0,2.2)--(2,1);
\draw[line width=5pt,grey!25,line cap=round] (0,1)--(1,-0.3);
\draw[line width=5pt,grey!25,line cap=round] (2,1)--(1,-0.3);
\node[ctr] (o) at (0,2.2) {$0$};
\node[hub] (h1) at (-2,1) {$1$}; \node[hub] (h2) at (0,1) {$2$}; \node[hub] (h3) at (2,1) {$3$};
\node[ch] (c4) at (-1,-0.3) {$4$}; \node[ch] (c5) at (1,-0.3) {$5$};
\draw[tr] (o)--(h1); \draw[tr] (o)--(h2); \draw[tr] (o)--(h3); \draw[tr] (h1)--(c4); \draw[tr] (h2)--(c5);
\draw[rd] (h1)--(h2); \draw[rd] (h2)--(c4); \draw[rd] (h3)--(c5); \draw[rd] (c4)--(c5);
\node[rlab] at (-1,1.2) {$12$}; \node[rlab] at (-0.75,0.45) {$24$}; \node[rlab] at (1.75,0.35) {$35$}; \node[rlab] at (0,-0.55) {$45$};
\node[lab] at (0,-1.2) {$G_0$ with $A_\tau=C_{35}$ shaded};
\end{scope}
\draw[-{Latex[length=2mm]},thick] (3.0,0.9) -- node[above,font=\footnotesize] {$/A_\tau$} (4.2,0.9);
\begin{scope}[xshift=7.2cm]
\node[ctr] (w) at (0,2.2) {$w$};
\node[hub] (h1) at (-1.6,0.8) {$1$};
\node[ch] (c4) at (0.4,-0.3) {$4$};
\draw[tr] (w)--(h1); \draw[tr] (h1)--(c4);
\draw[rd] (h1) to[bend left=35] (w);
\draw[rd] (c4) to[bend left=25] (w);
\draw[rd] (c4) to[bend right=25] (w);
\node[lab] at (-0.55,1.1) {$01$}; \node[lab] at (-0.95,0.05) {$14$};
\node[rlab] at (-1.45,2.0) {$12$}; \node[rlab] at (-0.25,0.7) {$24$}; \node[rlab] at (0.9,0.9) {$45$};
\node[lab] at (-0.5,-1.2) {$M(G_0)/A_\tau$, tree $\{01,14\}$};
\end{scope}
\end{tikzpicture}
\caption{Example~\ref{ex:extcontract}. Left: $G_0$ with $A_\tau=C_{35}$ shaded. Right: the contraction $M(G_0)/A_\tau$, in which $0,2,3,5$ have become one vertex $w$. Its tree is $\{01,14\}$.}\label{fig:contract}
\end{figure}
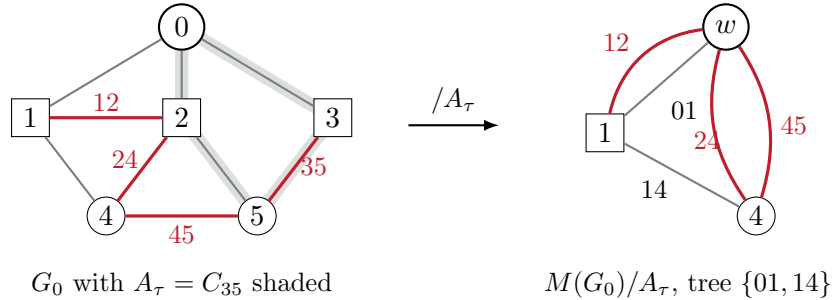

\begin{example}[Contracting $A_\tau$ on the running example]\label{ex:extcontract}
In $G_0$ take $\tau=\{35\}$, so $A_\tau=C_{35}=\{35,25,02,03\}$. Contracting $A_\tau$ merges the vertices $0,2,3,5$ into one vertex $w$ (Figure~\ref{fig:contract}) and leaves the five edges $01,14$ (the tree $B_0'$) and $12,24,45$ (red), with $12$ now joining $1$ to $w$ and $24,45$ joining $4$ to $w$. The fundamental circuits of $B_0'$ in $M'=M(G_0)/A_\tau$ are the $2$-cycle $\{12,01\}$ and the two triangles $\{24,14,01\}$, $\{45,14,01\}$, which are $C_{12}\setminus A_\tau$, $C_{24}\setminus A_\tau$, $C_{45}\setminus A_\tau$ as the lemma says. The coparking baselines of $(M',B_0')$ are the restrictions to $\rho$ of the baselines $b_{\{35\}\cup\rho}$ of $G_0$, for instance $b_{\{24,35,45\}}|_{\{24,45\}}=(1,1)$, so $\Ext_{G_0}(\{35\})=\Ext_{M'}(\emptyset)$: eight vectors, whose four maximal elements all have degree $2=|A_R\setminus A_\tau|-|R\setminus\tau|$.
\end{example}

\begin{lemma}[Deletion--contraction of coparking sets]\label{lem:dc-coparking}
Let $\CC=(C_j)_{j\in\sigma}$ be a family of sets, $\sigma\ne\emptyset$, and let $e$ lie
in $C_i$ only. Then
\[
  \PP^*(\CC)=\{a:\ a_i=0,\ a|_{\sigma\setminus i}\in\PP^*(\CC\setminus C_i)\}
  \ \sqcup\ \bigl(\mathbf e_i+\PP^*((C_j\setminus e)_{j\in\sigma})\bigr).
\]
If no element lies in exactly one $C_j$ then $\PP^*(\CC)=\emptyset$.
\end{lemma}
\begin{proof}
We have $e\in U_S$ if and only if $i\in S$, so removing $e$ from $C_i$ lowers $b_S(i)$ by one for every $S\ni i$ and changes no other coparking baseline, while removing $C_i$ leaves the baselines $b_S$ with $S\not\ni i$ unchanged.

Let $a\in\N^\sigma$. If $a_i=0$, then every $S\ni i$ has $b_S(i)\ge1$: the element $e$ lies in no other $C_j$, so $e\in C_i\cap U_S$. Thus $a_i=0<b_S(i)$ and no such $S$ can be fired. Therefore $a\in\PP^*(\CC)$ if and only if no $S\not\ni i$ can be fired, if and only if
$a|_{\sigma\setminus i}\in\PP^*(\CC\setminus C_i)$. If $a_i\ge1$, then for $S\not\ni i$ the firing condition is the same for $\CC$ and for $(C_j\setminus e)_{j\in\sigma}$, and for $S\ni i$ the condition $a|_S\ge b_S$ for $\CC$ is the condition $(a-\mathbf e_i)|_S\ge b'_S$ for $(C_j\setminus e)_{j\in\sigma}$, whose baselines $b'_S$ differ from $b_S$ only in $b'_S(i)=b_S(i)-1$.
Hence $a\in\PP^*(\CC)$ if and only if $a-\mathbf e_i\in\PP^*((C_j\setminus e)_{j\in\sigma})$.

Finally, if no element lies in exactly one $C_j$, then $U_\sigma=\emptyset$, so $b_\sigma=0$ and $\sigma$ can be fired
from every $a$.
\end{proof}

The characterisation in the next theorem is Dhar's burning algorithm \cite{Dhar}. On a graph
$G$ with a sink, burning identifies the superstable configurations. These are the $G$-parking
functions of Postnikov and Shapiro \cite{PostnikovShapiro}, and the maximal ones
correspond to acyclic orientations with a unique source. That the maximal $G$-parking functions
all have degree $|E|-|V|+1$ is proved this way in \cite{BCT}. See also \cite[Sec.~4]{PPW}. The
translation to a matroid with a fixed basis is as follows. The non-sink vertices are replaced by the
red edges $R$, the chips on a vertex by the coordinate $c_i$, and the edges at a vertex by the
elements of the fundamental circuit $C_i$. The edges from a vertex to the unburnt vertices become
the elements of $C_i$ lying in the circuit of some unburnt coordinate. A vertex burns when it
has fewer chips than edges to the burnt part, and a coordinate burns when
$c_i<|C_i\setminus A_{R'\setminus i}|$, with $R'$ the set of unburnt coordinates. The input is a
matroid together with a basis $B_0$, and the coparking functions depend on $B_0$ through its
fundamental circuits, as everything in this paper does. The same kind of choice is made for
$G$-parking functions when a sink is fixed.

\begin{theorem}[Burning algorithm for a matroid with a fixed basis]\label{thm:burning}
Let $M$ be a matroid with a basis $B_0$, red set $R$ of size $g$, and fundamental circuits
$\CC=(C_i)_{i\in R}$, so that a set $S\subseteq R$ can be fired from $c\in\N^R$ when
$c_i\ge|C_i\cap U_S|$ for every $i\in S$. Then the coparking functions $\PP^*(\CC)$ form a nonempty multicomplex, and every maximal
coparking function has degree
\[
  |A_R|-g=r(A_R),
\]
which is $r(M)$ minus the number of coloops of $M$. Moreover, $c\in\PP^*(\CC)$ if and only if
the coordinates can be burnt one at a time, that is, there is an ordering $s_1,\dots,s_g$ of
$R$ such that $\{s_m,\dots,s_g\}$ cannot be fired from $c$ at $s_m$ for each $m$:
\[
  c_{s_m}<\bigl|C_{s_m}\setminus A_{\{s_{m+1},\dots,s_g\}}\bigr|\qquad(1\le m\le g).
\]
\end{theorem}
\begin{proof}
We first prove the characterisation by orderings, following \cite[Theorem~5.15]{BDHOZ}.
Let $c$ be an element of $\PP^*(\CC)$. Since $R$ cannot be fired from $c$, there is some $s_1$ such that $c_{s_1}<b_R(s_1)$. Similarly, since $R\setminus s_1$ cannot be fired from $c$, there is some $s_2$ such that $c_{s_2}<b_{R\setminus s_1}(s_2)$. Repeating this process gives an ordering
$s_1,\dots,s_g$ of $R$ such that for all $m$,
\[
  c_{s_m}\ <\ b_{R_m}(s_m)=|C_{s_m}\setminus A_{R_m\setminus s_m}|,\qquad R_m=\{s_m,\dots,s_g\},
\]
with $R_{g+1}=\emptyset$. For the converse, suppose an ordering $s_1,\dots,s_g$ satisfies the displayed inequalities and let $\rho\subseteq R$ be nonempty. Let $s_m$ denote the smallest element of $\rho$ under that ordering, so that $\rho\subseteq R_m$ and $\rho\setminus s_m\subseteq R_{m+1}$. Then $C_{s_m}\setminus A_{\rho\setminus s_m}\supseteq C_{s_m}\setminus A_{R_{m+1}}$, so $b_\rho(s_m)\ge b_{R_m}(s_m)>c_{s_m}$, and $\rho$ cannot be fired from $c$. Hence $c\in\PP^*(\CC)$.
This is a burning algorithm in the usual sense. The threshold $b_{R'}(i)=|C_i\setminus A_{R'\setminus i}|$
can only increase as the unburnt set $R'$ shrinks (Lemma~\ref{lem:mono}(c)), so a coordinate that
can be burnt at some step can still be burnt at any later step. The ordering can therefore be found
greedily, burning at each step any $i\in R'$ with $c_i<b_{R'}(i)$. Then $c\in\PP^*(\CC)$ if and
only if this process burns all of $R$.

It remains to prove that $\PP^*(\CC)$ is a nonempty multicomplex and that its maximal elements have degree $r(A_R)$. It is closed under division (Section~\ref{sec:coparking}) and it contains $0$ because $b_S(i)\ge1$ for every $i\in S$ (Lemma~\ref{lem:mono}(a)). For an ordering $s=(s_1,\dots,s_g)$ of $R$ define $c^{s}\in\N^R$ by
$c^{s}_{s_m}=b_{R_m}(s_m)-1$ for every $m$. It satisfies the displayed inequalities, so
$c^{s}\in\PP^*(\CC)$ by the converse just proved. The sets $C_{s_m}\setminus A_{R_{m+1}}$,
$m=1,\dots,g$, partition $A_R$: an element of $A_R$ lies in the one indexed by the last $s_m$
in the ordering whose circuit contains it. Hence
\[
  |c^{s}|=\sum_{m=1}^{g}\bigl(|C_{s_m}\setminus A_{R_{m+1}}|-1\bigr)=|A_R|-g=r(A_R),
\]
the last equality by \eqref{eq:nullity}. Every $c\in\PP^*(\CC)$ satisfies $c\le c^{s}$ for the ordering $s$ constructed from $c$ above, so every maximal element of $\PP^*(\CC)$ is some $c^{s}$ and has degree $r(A_R)$. Finally, an element outside $A_R$ is a tree edge lying in no circuit by (F2), hence a coloop. A coloop lies in no $C_i$, so $r(A_R)$ is $r(M)$ minus the number of coloops of $M$.
\end{proof}

For a graph $G$ with sink $q$, Dhar's burning algorithm does three things \cite{Klivans}. It decides whether a configuration is superstable. It identifies the maximal superstable configurations with the acyclic orientations of $G$ having $q$ as unique source. Through the toppling matrix, it sits inside the theory of the sandpile group. Only the first survives in Theorem~\ref{thm:burning}, together with the purity that comes with it. Without a Laplacian there is no toppling matrix, hence no recurrent configurations, no duality and no sandpile group. The number of elements of $\PP^*(\CC)$ is the number of bases of $M$ only when $\CC$ is a cycle system (Theorem~\ref{thm:cdmpy}). In general it is smaller, and the difference is exactly what the fibres supply, one dead node at a time (Theorem~\ref{thm:B}). We also know no formula for the number of maximal elements of $\PP^*(\CC)$ beyond the cycle-system case, where it is the leading coefficient of $h(M)$.

The theorem is more general in one direction: an element may lie in any number of the $C_i$, whereas an edge of $G$ has two ends. Suppose every element of $A_R$ lies in at most two of the $C_i$, and let $G'$ be the graph on $R\cup\{q\}$ with an edge $ij$ for each element with owner set $\{i,j\}$ and an edge $iq$ for each element with owner set $\{i\}$. Then $\PP^*(\CC)$ is the set of $G'$-parking functions, and the theorem is the classical statement. Beyond that case the elements with three or more owners play the role of hyperedges, and $\PP^*(\CC)$ is a relative of the hypergraph parking functions of \cite{BDHOZ}, with thresholds defined differently, so that the two pure multicomplexes differ in general.

\begin{corollary}\label{cor:extpure}
For every $\tau\subseteq R$, $\Ext(\tau)$ is a nonempty multicomplex, and every maximal element
of $\Ext(\tau)$ has degree $|A_R\setminus A_\tau|-|R\setminus\tau|=r(A_R)-r(A_\tau)$.
\end{corollary}
\begin{proof}
By Lemma~\ref{lem:extcontract}, $\Ext(\tau)=\PP^*(\CC')$ for the fundamental circuits $\CC'$ of
the basis $B_0\setminus A_\tau$ of $M/A_\tau$. Its red set is $R\setminus\tau$ and the union
of its circuits $C_j\setminus A_\tau$ is $A_R\setminus A_\tau$. Theorem~\ref{thm:burning} applied to $(M/A_\tau,B_0\setminus A_\tau)$
gives the claim, with $r_{M/A_\tau}(A_R\setminus A_\tau)=r(A_R)-r(A_\tau)$.
\end{proof}

\begin{example}\label{ex:extpure}
Take the stratum $\tau=\{04,12,24\}$ of Example~\ref{ex:ext} (Figure~\ref{fig:ext}). Contracting
$A_\tau=\{04,12,24,01,02,14\}$ identifies the vertices $0,1,2,4$ and leaves the tree $\{03,25\}$ with
the red edges $35,45$. The circuits are $\{35,03,25\}$ and $\{45,25\}$, and their thresholds are the
three conditions found there. The two orderings give $c^{(35,45)}=(|\{35,03\}|-1,\,|\{45,25\}|-1)=(1,1)$
and $c^{(45,35)}=(|\{35,03,25\}|-1,\,|\{45\}|-1)=(2,0)$, the two maximal elements, both of degree
$|\{03,25,35,45\}|-2=2=r(A_R)-r(A_\tau)=5-3$. On $G_0$ itself the $24$ orderings of $R$ produce
the $11$ maximal elements of $\Ext(\emptyset)$, all of degree $9-4=5=r(M(G_0))$.
\end{example}

\subsection{The stratum identity}

Equation \eqref{eq:degvec} expresses the degree sequence of $\PP^*(F)$ as a sum over the dead strata in which the fibres enter only through their degree sequences $h(K_\tau)$. The baselines, the dead nodes and the extension sets are all determined by $(M,B_0)$. So the degree sequence of $\PP^*(F)$ is the same for every fibre system, and Theorem~\ref{thm:main} reduces to the following identity, in which no fibre appears.

\begin{theorem}[Stratum identity]\label{thm:B}
For every matroid $M$ with a basis $B_0$,
\[
  h(M)=\sum_{\tau}t^{|b_\tau|}\,h(K_\tau)\,E_\tau(t),
\]
the sum over $\tau=\emptyset$ and all dead strata.
\end{theorem}

The identity is proved by a deletion--contraction induction of the same shape as that of Theorem~\ref{thm:cdmpy}, carried out on the following objects. To avoid a clash of letters we write $M_0$ for the matroid of the theorem in this subsection and the next. A \emph{pair} $(M,\CC)$ is a matroid $M$ together with a list $\CC=(C_j)_{j\in\sigma}$ of subsets of its ground set indexed by a set $\sigma$ of \emph{surviving coordinates}. An element $e$ of $M$ is \emph{admissible}, as in Section~\ref{sec:coparking}, if it lies in
exactly one $C_i$. Its children are the pair $(M\setminus e,\CC\setminus\{C_i\})$, which drops the coordinate $i$, and the pair $(M/e,(C_j\setminus e)_{j\in\sigma})$, which keeps all coordinates. The second child exists only when $e$ is not a loop.

Starting from the root pair $(M_0,(C_j)_{j\in R})$ and expanding admissible elements produces the \emph{labelled minors}: pairs $(\mathrm{Del},\mathrm{Con})$ of disjoint subsets of the ground set of $M_0$ with
$M=(M_0\setminus\mathrm{Del})/\mathrm{Con}$, surviving coordinates the $j$ with
$C_j\cap\mathrm{Del}=\emptyset$, and $C_j(M)=C_j\setminus\mathrm{Con}$. As the elements of $M$ lie outside $\mathrm{Con}$,
\begin{equation}\label{eq:membership}
  y\in C_j(M)\ \text{ if and only if }\ y\in C_j\qquad(y\in E(M),\ j\in\sigma),
\end{equation}
so $m_\tau(y)$ is the same in $M$ and in $M_0$ for every $\tau\subseteq\sigma$ and every element $y$ of $M$. The definitions of $U_\tau$, $b_\tau$, $A_\tau$, $K_\tau$ and $\Ext(\tau)$ apply verbatim to a pair. The one convention is that a coordinate set $\tau$ is \emph{dead} when $U_\tau(M_0)$ is
independent in $M_0$, so deadness always refers to the root. For a pair we write
\[
  B(M,\CC)=\sum_\tau t^{|b_\tau|}\,h(K_\tau)\,E_\tau(t),
\]
and when the pair has to be named we write $b_\tau(M)$, $K_\tau(M)$ and $E_\tau(M)$ for its
baseline, dead node and generating function $E_\tau(t)$.

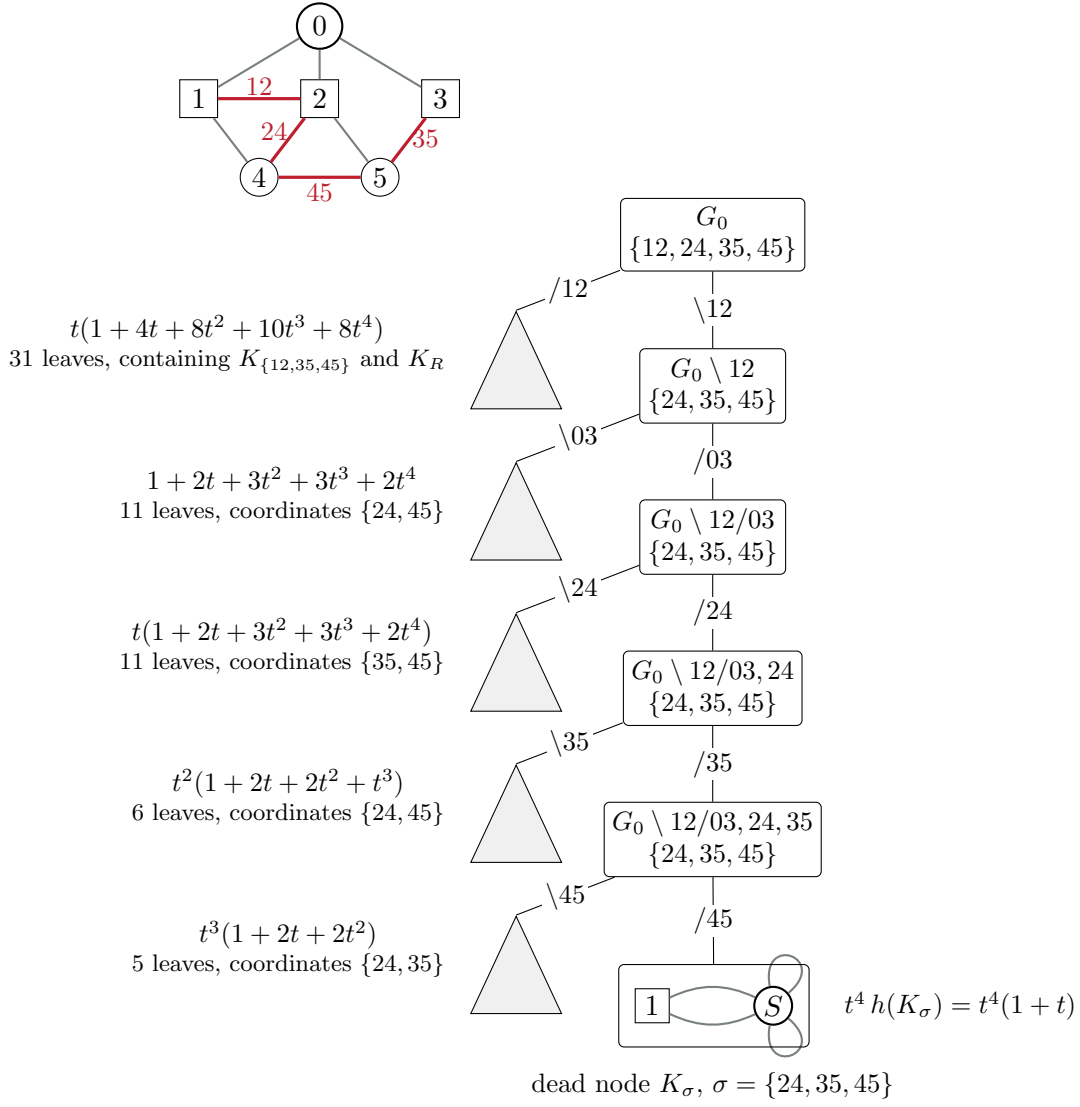
\begin{figure}[h]\centering
\begin{tikzpicture}[
  pair/.style={draw,rounded corners=2pt,fill=white,align=center,font=\footnotesize,inner sep=3pt},
  sub/.style={draw,isosceles triangle,isosceles triangle apex angle=50,shape border rotate=90,
              anchor=apex,minimum height=8mm,minimum width=12mm,inner sep=0pt,fill=grey!12},
  elab/.style={font=\footnotesize,fill=white,inner sep=1pt},
  plab/.style={font=\footnotesize,align=center}]
\begin{scope}[xshift=-5.2cm,yshift=1.3cm,scale=0.8]
\node[ctr] (o) at (0,2.2) {$0$};
\node[hub] (h1) at (-2,1) {$1$}; \node[hub] (h2) at (0,1) {$2$}; \node[hub] (h3) at (2,1) {$3$};
\node[ch] (c4) at (-1,-0.3) {$4$}; \node[ch] (c5) at (1,-0.3) {$5$};
\draw[tr] (o)--(h1); \draw[tr] (o)--(h2); \draw[tr] (o)--(h3); \draw[tr] (h1)--(c4); \draw[tr] (h2)--(c5);
\draw[rd] (h1)--(h2); \draw[rd] (h2)--(c4); \draw[rd] (h3)--(c5); \draw[rd] (c4)--(c5);
\node[rlab] at (-1,1.2) {$12$}; \node[rlab] at (-0.75,0.45) {$24$}; \node[rlab] at (1.75,0.35) {$35$}; \node[rlab] at (0,-0.55) {$45$};
\end{scope}
\node[pair] (n0) at (0,0.3)   {$G_0$\\ $\{12,24,35,45\}$};
\node[pair] (n1) at (0,-1.7)  {$G_0\setminus12$\\ $\{24,35,45\}$};
\node[pair] (n2) at (0,-3.7)  {$G_0\setminus12/03$\\ $\{24,35,45\}$};
\node[pair] (n3) at (0,-5.7)  {$G_0\setminus12/03,24$\\ $\{24,35,45\}$};
\node[pair] (n4) at (0,-7.7)  {$G_0\setminus12/03,24,35$\\ $\{24,35,45\}$};
\node[ctr,minimum size=5mm] (S) at (0.8,-9.9) {$S$};
\node[hub,minimum size=4.5mm] (h1d) at (-0.8,-9.9) {$1$};
\draw[tr] (h1d) to[bend left=25] (S);
\draw[tr] (h1d) to[bend right=25] (S);
\draw[tr] (S) to[out=40,in=100,looseness=7] (S);
\draw[tr] (S) to[out=-40,in=-100,looseness=7] (S);
\node[draw,rounded corners=2pt,fit=(S)(h1d),inner xsep=6pt,inner ysep=8pt] (n5) {};
\node[plab] at (0,-10.95) {dead node $K_\sigma$, $\sigma=\{24,35,45\}$};
\draw (n0)--node[elab]{$\setminus12$}(n1);
\draw (n1)--node[elab]{$/03$}(n2);
\draw (n2)--node[elab]{$/24$}(n3);
\draw (n3)--node[elab]{$/35$}(n4);
\draw (n4)--node[elab]{$/45$}(n5);
\foreach \i/\y/\e in {0/-0.7/{/12},1/-2.7/{\setminus03},2/-4.7/{\setminus24},3/-6.7/{\setminus35},4/-8.7/{\setminus45}}{
  \node[sub] (s\i) at (-2.6,\y) {};
  \draw (n\i)--node[elab]{$\e$}(s\i.apex);}
\node[plab,anchor=east] at (-3.4,-1.15) {$t(1+4t+8t^2+10t^3+8t^4)$\\[-1pt]{\scriptsize $31$ leaves, containing $K_{\{12,35,45\}}$ and $K_R$}};
\node[plab,anchor=east] at (-3.4,-3.15) {$1+2t+3t^2+3t^3+2t^4$\\[-1pt]{\scriptsize $11$ leaves, coordinates $\{24,45\}$}};
\node[plab,anchor=east] at (-3.4,-5.15) {$t(1+2t+3t^2+3t^3+2t^4)$\\[-1pt]{\scriptsize $11$ leaves, coordinates $\{35,45\}$}};
\node[plab,anchor=east] at (-3.4,-7.15) {$t^2(1+2t+2t^2+t^3)$\\[-1pt]{\scriptsize $6$ leaves, coordinates $\{24,45\}$}};
\node[plab,anchor=east] at (-3.4,-9.15) {$t^3(1+2t+2t^2)$\\[-1pt]{\scriptsize $5$ leaves, coordinates $\{24,35\}$}};
\node[plab,anchor=west] at (1.6,-9.9) {$t^4\,h(K_\sigma)=t^4(1+t)$};
\end{tikzpicture}
\caption{Example~\ref{ex:dctree}. The branch of the deletion--contraction tree of $G_0$ leading to the dead node $K_\sigma$, $\sigma=\{24,35,45\}$. Each box is a labelled minor with its surviving coordinates. Edges are labelled by the expanded element, $\setminus e$ for deletion
and $/e$ for contraction.}\label{fig:dctree}
\end{figure}

\begin{example}[The deletion--contraction tree of $G_0$]\label{ex:dctree}
Figure~\ref{fig:dctree} shows part of the tree of the root pair of $G_0$, expanding at each
node the first admissible element in the order $12,03,24,35,45,01,02,14,25$. Deleting $12$
drops the coordinate $12$, and contracting $03,24,35,45$, the four elements of $U_\sigma$ for
$\sigma=\{24,35,45\}$, leaves the pair with circuits $C_{24}=\{14,01,02\}$, $C_{35}=\{25,02\}$,
$C_{45}=\{14,01,02,25\}$. Every element now lies in at least two of them, so the node is the dead node
$K_\sigma$ of Example~\ref{ex:deadnode}, reached after $|U_\sigma|=4$ contractions, and it
contributes $t^4h(K_\sigma)=t^4(1+t)$. The other two dead nodes are at $\{12,35,45\}$ and at $R$, and both have rank zero and contribute $t^5$ each. The remaining leaves have minors consisting of coloops and supply $E_\emptyset(t)$.
\end{example}

Two lemmas about labelled minors are needed. Parts (a) to (c) of the first collect the invariants that are
preserved along the path from the root, and their proof is the only place where that path is traversed. The
remaining parts and the second lemma are local consequences, resting on (F2), (F3) and on the description of the circuits of a
contraction: the circuits of $(M_0\setminus\mathrm{Del})/\mathrm{Con}$ are the minimal nonempty sets
$C\setminus\mathrm{Con}$ with $C$ a circuit of $M_0\setminus\mathrm{Del}$ \cite[Prop.~3.1.10]{Oxley}.

\begin{lemma}[Labelled minors]\label{lem:invariants}
Let $(M,\CC)$ be a labelled minor with surviving coordinates $\sigma$. Then:
\begin{enumerate}[label=(\alph*)]
\item $\mathrm{Con}$ is independent in $M_0$.
\item Every $y\in\mathrm{Con}$ has $m_\sigma(y)\le1$ and every $y\in\mathrm{Del}$ has $m_\sigma(y)=0$.
      Hence $\mathrm{Con}\cap A_\tau(M_0)\subseteq U_\tau(M_0)$ and $\mathrm{Del}\cap A_\tau(M_0)=\emptyset$
      for every $\tau\subseteq\sigma$.
\item If $C$ is a circuit of $M_0$ with $C\cap\mathrm{Del}=\emptyset$, then every red edge of $C$
      survives.
\item Every circuit of $M$ has the form $C\setminus\mathrm{Con}$ with $C$ a circuit of $M_0$
      avoiding $\mathrm{Del}$ whose red edges all survive. If $S$ is the set of red edges of $C$, then $C\setminus\mathrm{Con}\subseteq\bigcup_{j\in S}C_j(M)$ and $U_S(M)\subseteq C\setminus\mathrm{Con}$. Every
      element of $M$ outside $A_\sigma(M)=\bigcup_{j\in\sigma}C_j(M)$ is a coloop.
\item An admissible element $e\in C_i(M)$ is not a coloop of $M$.
\end{enumerate}
\end{lemma}
\begin{proof}
For (a) to (c) we argue by induction on the number of expansions leading from the root to the labelled minor.
At the root, $\mathrm{Del}=\mathrm{Con}=\emptyset$ and $\sigma=R$, so (a) and (b) hold
trivially and (c) says that every red edge survives, which is the case. For the inductive step,
let $(M,\CC)$ be a labelled minor satisfying (a), (b) and (c), let $e\in C_i(M)$ be admissible,
and let $(M',\CC')$ be one of the two children obtained by expanding $e$. We show that
$(M',\CC')$ satisfies (a), (b) and (c). The argument for (c) is that of \cite[Lemma~6.4]{Regular},
with (F2) in place of the unique-union representation of a circuit.

Suppose $e$ is contracted, so that $\mathrm{Con}'=\mathrm{Con}\cup\{e\}$ while $\mathrm{Del}$ and
$\sigma$ are unchanged. Since $e$ is not a loop of $M=(M_0\setminus\mathrm{Del})/\mathrm{Con}$, it
does not lie in the closure of $\mathrm{Con}$ in $M_0$, so $\mathrm{Con}\cup\{e\}$ is independent
in $M_0$, which is (a). For (b), the elements of $\mathrm{Con}$ have $m_\sigma\le1$ by hypothesis,
and $e$ lies in exactly one surviving circuit, so $m_\sigma(e)=1$ by \eqref{eq:membership}.
The second half of (b) and statement (c) refer only to $\mathrm{Del}$ and $\sigma$, which are unchanged.

Suppose $e$ is deleted, so that $\mathrm{Del}'=\mathrm{Del}\cup\{e\}$, $\sigma'=\sigma\setminus i$
and $\mathrm{Con}$ is unchanged. Statement (a) is unchanged, and (b) persists because
$m_{\sigma\setminus i}\le m_\sigma$, while $e$ itself has $m_{\sigma\setminus i}(e)=0$, since $C_i(M)$ is the only surviving circuit of $(M,\CC)$ containing $e$. For (c), let $C$ be a circuit of $M_0$ avoiding
$\mathrm{Del}\cup\{e\}$ and let $S$ be its set of red edges. In particular $C$ avoids
$\mathrm{Del}$, so by the induction hypothesis every $j\in S$ lies in $\sigma$. The deletion
drops the single coordinate $i$, so it remains to show that $i\notin S$. Suppose $i\in S$. As
$e\in C_i$ and $e\notin C$, we have $e\neq i$, so $e$ is a tree edge of $C_i$ outside $C$, and
by (F2) it lies in $C_j$ for some $j\in S$ with $j\neq i$. Since $j\in\sigma$ and
$e\notin\mathrm{Con}$, also $e\in C_j(M)$, a second surviving circuit of $(M,\CC)$ containing $e$,
contradicting the admissibility of $e$. Hence $i\notin S$ and $S\subseteq\sigma\setminus i=\sigma'$.

(d) By \cite[Prop.~3.1.10]{Oxley}, every circuit of $M=(M_0\setminus\mathrm{Del})/\mathrm{Con}$
has the form $Z=C\setminus\mathrm{Con}$ with $C$ a circuit of $M_0\setminus\mathrm{Del}$. Such a $C$ is a circuit of $M_0$ avoiding $\mathrm{Del}$, so by (c) every red
edge of $C$ survives. This is the first claim. For the second, let $y\in Z$. By (F3),
$y\in C_j$ for some $j\in S$, and $j$ survives, so $y\in C_j(M)$ by \eqref{eq:membership}.
Hence $Z\subseteq\bigcup_{j\in S}C_j(M)$. For the third claim, which is the argument of \cite[Lemma~3.5]{Regular}, an element $y$ of $U_S(M)$ lies in exactly one $C_j(M)$ with $j\in S$, hence by \eqref{eq:membership} in exactly one $C_j$ with $j\in S$, so $y\in C$ by (F3), and $y\notin\mathrm{Con}$ gives $y\in Z$. Finally, an element of $M$ outside $A_\sigma(M)$ lies in no circuit of $M$, so it is a coloop.

(e) Write $N=M_0\setminus\mathrm{Del}$, so that $M=N/\mathrm{Con}$ and
$r_M(X)=r_N(X\cup\mathrm{Con})-r_N(\mathrm{Con})$ for $X\subseteq E(M)=E(N)\setminus\mathrm{Con}$.
Since $e\notin\mathrm{Con}$, the element $e$ is a coloop of $M$ if and only if
$r_M(E(M)\setminus e)<r_M(E(M))$. This in turn is equivalent to
$r_N(E(N)\setminus e)<r_N(E(N))$, and also equivalent to $e$ being a coloop of $N$. Since $i$
survives, $C_i\cap\mathrm{Del}=\emptyset$, so $C_i$ is a circuit of $N$ containing $e$, and $e$
is not a coloop of $N$.
\end{proof}

The second lemma restricts the extension sets of a labelled minor in two situations. Both are
used in the proof of Theorem~\ref{thm:B} to make a term of $B(M,\CC)$ vanish.

\begin{lemma}\label{lem:E}
Let $(M,\CC)$ be a labelled minor with surviving coordinates $\sigma$, let $\tau\subseteq\sigma$, and let $e\in C_i(M)$ be admissible.
\begin{enumerate}[label=(\alph*)]
\item If $i\notin\tau$ and $e$ lies in a circuit of $M$ contained in $A_\tau(M)\cup\{e\}$, then every
      $c\in\Ext_M(\tau)$ has $c_i=0$.
\item If $i\in\tau$, $\tau$ is dead in $M_0$, and $e$ is a loop of $M$, then
      $\Ext_M(\tau)=\emptyset$.
\end{enumerate}
\end{lemma}
\begin{proof}
Let $Z$ be a circuit of $M$ containing $e$, namely the given one in (a) and $\{e\}$ in (b).
By Lemma~\ref{lem:invariants}(d) we have $Z=C\setminus\mathrm{Con}$ for a circuit $C$ of $M_0$
avoiding $\mathrm{Del}$ whose red set $S$ survives, and $e\in C_j(M)$ for some $j\in S$. Since $C_i(M)$ is the
only surviving circuit containing $e$, we get $i\in S$. Put $\rho=S\setminus\tau$ and let
$j\in\rho$. An element of $C_j(M)\cap U_{\tau\cup\rho}(M)$ lies in exactly one $C_k(M)$ with
$k\in\tau\cup\rho$, hence in exactly one with $k\in S$ and in none with $k\in\tau$. So by
Lemma~\ref{lem:invariants}(d)
\[
  C_j(M)\cap U_{\tau\cup\rho}(M)\subseteq U_S(M)\setminus A_\tau(M)\subseteq Z\setminus A_\tau(M)\qquad(j\in\rho),
\]
and the coparking baseline $b_{\tau\cup\rho}(M)(j)=|C_j(M)\cap U_{\tau\cup\rho}(M)|$ is at most the
number of elements of $Z\setminus A_\tau(M)$ in $C_j(M)$.

(a) As $i\notin\tau$ and $C_i(M)$ is the only surviving circuit containing $e$, we have
$e\notin A_\tau(M)$, so $Z\setminus A_\tau(M)=\{e\}$, and $e$ lies in $C_j(M)$ only for $j=i$.
Hence $b_{\tau\cup\rho}(M)|_\rho\le\mathbf e_i$, with $i\in\rho$. If some $c\in\Ext_M(\tau)$ had
$c_i\ge1$, then $c|_\rho\ge b_{\tau\cup\rho}(M)|_\rho$ and the stratum $\tau\cup\rho$ could be
fired from $(b_\tau,c)$, contradicting $c\in\Ext_M(\tau)$.

(b) As $Z=\{e\}$, $e\notin C_j(M)$ for $j\ne i$, and $i\in\tau$, we get
$b_{\tau\cup\rho}(M)|_\rho=0$. If $\rho\ne\emptyset$ then $\rho$ can be fired from $0$, so
$0\notin\Ext_M(\tau)$, and $\Ext_M(\tau)$ is empty because it is closed under division. It
remains to rule out $\rho=\emptyset$. Suppose $S\subseteq\tau$. By (F3), $C\subseteq A_S\subseteq A_\tau(M_0)$.
The elements of $C$ other than $e$ lie in $\mathrm{Con}$ because $C\setminus\mathrm{Con}=\{e\}$,
so Lemma~\ref{lem:invariants}(b) gives
\[
  C\setminus e\subseteq\mathrm{Con}\cap A_\tau(M_0)\subseteq U_\tau(M_0).
\]
As for $e$ itself, it lies in exactly one surviving circuit, namely $C_i(M)$ with $i\in\tau$. So
$m_\tau(e)=1$ in $M$ and by \eqref{eq:membership} also in $M_0$, that is, $e\in U_\tau(M_0)$. Hence
$C\subseteq U_\tau(M_0)$, so $U_\tau(M_0)$ is dependent, which contradicts the deadness of
$\tau$ in $M_0$.
\end{proof}

\subsection{Proofs of the stratum identity and the main theorem}
Using the tools developed in the previous subsections, we prove Theorems~\ref{thm:B} and~\ref{thm:main}.

\begin{proof}[Proof of Theorem~\ref{thm:B}]
We prove the stronger claim that $B(M,\CC)=h(M)$ for every labelled minor $(M,\CC)$, by induction from the leaves of the deletion--contraction tree towards its root. The theorem is the case of the root pair $(M_0,\CC)$. Throughout, $\mathrm{Con}$ is independent in $M_0$ by Lemma~\ref{lem:invariants}(a).

Given a labelled minor $(M,\CC)$ with surviving coordinates $\sigma$, suppose first that there is no admissible element. Then every element of $A_\sigma(M)$ lies in at least two of the $C_j(M)$, so $U_\sigma(M)=\emptyset$ and $K_\sigma(M)=M|A_\sigma$, while every element
outside $A_\sigma$ is a coloop by Lemma~\ref{lem:invariants}(d). Deleting coloops does not change $h$, so $h(M)=h(K_\sigma(M))$. Moreover $\sigma$ is dead in $M_0$, since $U_\sigma(M)=U_\sigma(M_0)\setminus\mathrm{Con}$ by Lemma~\ref{lem:invariants}(b), which gives $U_\sigma(M_0)\subseteq\mathrm{Con}$. For $B(M,\CC)$, the term $\tau=\sigma$ in the defining sum is $t^{0}\,h(K_\sigma)\,E_\sigma(t)$ with $E_\sigma(t)=1$, as $\Ext(\sigma)=\N^{\emptyset}=\{0\}$. For every $\tau\subsetneq\sigma$, the
set $\rho=\sigma\setminus\tau$ has $b_\sigma|_\rho=0$, so $\rho$ can be fired from $0$ and $\Ext(\tau)=\emptyset$. Hence $B(M,\CC)=h(K_\sigma(M))=h(M)$. (If $\sigma=\emptyset$ the same reading gives $B(M,\CC)=1=h(M)$, all elements of $M$ being coloops.)

Now suppose there is an admissible element $e\in C_i(M)$. If $e$ is a loop then $h(M)=h(M\setminus e)$. Otherwise $e$ is not a loop and not a coloop by Lemma~\ref{lem:invariants}(e), so \eqref{eq:tutte} gives $h(M)=h(M\setminus e)+t\,h(M/e)$. Write $$B_\tau(M)=t^{|b_\tau(M)|}\,h(K_\tau(M))\,E_\tau(M)$$ for the $\tau$-term of $B(M,\CC)$, with $B_\tau(M\setminus e)=0$ when $i\in\tau$, since the coordinate $i$ is dropped in $M\setminus e$. We show that each $B_\tau$ satisfies the same recursion as $h$, that is, $B_\tau(M)=B_\tau(M\setminus e)$ if $e$ is a loop and $B_\tau(M)=B_\tau(M\setminus e)+t\,B_\tau(M/e)$ otherwise. Summing over $\tau$ and applying the induction hypothesis to the children then gives $B(M,\CC)=h(M)$.

Consider first a dead $\tau\ni i$. As $B_\tau(M\setminus e)=0$, the relation to be matched is $B_\tau(M)=0$ if $e$ is a loop and $B_\tau(M)=t\,B_\tau(M/e)$ otherwise. If $e$ is a loop, $B_\tau(M)=0$ by
Lemma~\ref{lem:E}(b). Otherwise $e\in U_\tau(M)$, by admissibility and $i\in\tau$. Then $U_\tau(M/e)=U_\tau(M)\setminus e$, so
$|b_\tau(M/e)|=|b_\tau(M)|-1$ by \eqref{eq:bsize}. We also have $A_\tau(M/e)=A_\tau(M)\setminus e$, so
$K_\tau(M/e)=\bigl((M/e)|(A_\tau\setminus e)\bigr)/(U_\tau\setminus e)=(M|A_\tau)/U_\tau=K_\tau(M)$.
Moreover $\Ext_{M/e}(\tau)=\Ext_M(\tau)$, because $e\notin C_j(M)$ for $j\notin\tau$, so every
coparking baseline $b_{\tau\cup\rho}|_\rho$ with $\rho\subseteq\sigma\setminus\tau$ is the same
in $M$ and $M/e$. Hence $B_\tau(M)=t\,B_\tau(M/e)$.

Consider next a dead $\tau\not\ni i$. This includes $\tau=\emptyset$, for which $A_\emptyset=U_\emptyset=\emptyset$ and $h(K_\emptyset)=1$. As $e\notin A_\tau(M)$, the pairs $M$, $M\setminus e$ and $M/e$ (when $e$ is not a loop) have the same $A_\tau$, the same $U_\tau$ and the same $|b_\tau|$. Their dead nodes are $K_\tau(M)=(M|A_\tau)/U_\tau$, then $K_\tau(M\setminus e)=((M\setminus e)|A_\tau)/U_\tau=K_\tau(M)$, since deleting an element outside $A_\tau$ does not change the restriction to $A_\tau$, and $K_\tau(M/e)=((M/e)|A_\tau)/U_\tau$, where $(M/e)|A_\tau=(M|(A_\tau\cup e))/e$ may differ from $M|A_\tau$. From the last paragraph of the proof of Lemma~\ref{lem:extcontract}, which uses only the definition of the coparking baselines and applies to any family of sets, we get
$$\Ext_M(\tau)=\PP^*((C_j(M)\setminus A_\tau)_{j\in\sigma\setminus\tau}),$$
and in this family $e$ lies in $C_i(M)\setminus A_\tau$ only. So Lemma~\ref{lem:dc-coparking} gives $E_\tau(M)=E_\tau(M\setminus e)+t\,E'$, where $t\,E'$ is the generating function of the $c\in\Ext_M(\tau)$ with $c_i\ge1$, and $E'=E_\tau(M/e)$ when $e$ is not a loop.

The relation to be matched is $B_\tau(M)=B_\tau(M\setminus e)$ if $e$ is a loop and $B_\tau(M)=B_\tau(M\setminus e)+t\,B_\tau(M/e)$ otherwise, and we split according to whether $e$ lies in a circuit of $M$ contained in $A_\tau\cup e$, which is the case when $e$ is a loop. If it does, then by Lemma~\ref{lem:E}(a) every $c\in\Ext_M(\tau)$ has $c_i=0$, so $E'=0$ and $B_\tau(M)=B_\tau(M\setminus e)$. For a loop $e$ this is the relation to be matched. For $e$ not a loop, $E_\tau(M/e)=E'=0$ gives $B_\tau(M/e)=0$, so $B_\tau(M)=B_\tau(M\setminus e)+t\,B_\tau(M/e)$, again the relation to be matched. If $e$ lies in no circuit of $M$ inside $A_\tau\cup e$, then $e$ is not a loop, and it is a coloop of $M|(A_\tau\cup e)$, so $(M|(A_\tau\cup e))/e=M|A_\tau$ and $K_\tau(M/e)=K_\tau(M)$. Hence $E_\tau(M)=E_\tau(M\setminus e)+t\,E_\tau(M/e)$ gives $B_\tau(M)=B_\tau(M\setminus e)+t\,B_\tau(M/e)$, the relation to be matched. Together with the case $\tau\ni i$ this shows that every $B_\tau$ satisfies the recursion of $h(M)$, and the induction is complete.
\end{proof}

\begin{proof}[Proof of Theorem~\ref{thm:main}]
The set $\PP^*(F)$ is closed under division by Lemma~\ref{lem:C}, and its degree sequence is $h(M)$ by \eqref{eq:degvec} and Theorem~\ref{thm:B}. For purity let $a\in\PP^*(F)$, put $\tau=\max\mathrm{Dom}(a)$ (or $\tau=\emptyset$ if $\mathrm{Dom}(a)$ is empty), and write $a=(b_\tau+m,c)$ with $m\in F(\tau)$ and $c\in\Ext(\tau)$ as in \eqref{eq:decomp}. Choose
$m'\ge m$ maximal in $F(\tau)$ and $c'\ge c$ maximal in $\Ext(\tau)$, and let
$a'=(b_\tau+m',c')$, which lies in $\PP^*(F)$ with $\max\mathrm{Dom}(a')=\tau$ by the remark after \eqref{eq:decomp}. Thus every element of $\PP^*(F)$ lies below an element of degree
\[
  |b_\tau|+r(K_\tau)+\bigl(r(A_R)-r(A_\tau)\bigr)=r(A_R)
\]
by Definition~\ref{def:fibre}(i), Corollary~\ref{cor:extpure} and Lemma~\ref{lem:rank}. Therefore every maximal element of $\PP^*(F)$ has that degree, which is the degree of $h(M)$ and also $r(M)$ minus the number of coloops (Section~\ref{sec:prelim}).
\end{proof}

\subsection{Specialisation to cycle systems and generalized cycle systems}

\begin{proposition}\label{prop:gcs-trivial}
Let $M$ be a matroid with a basis $B_0$ and fundamental circuits $\CC=(C_i)_{i\in R}$. The
following are equivalent.
\begin{enumerate}[label=(\alph*)]
\item $\CC$ is a generalized cycle system.
\item Every dead node $K_\sigma$ has rank zero.
\item At every dead stratum $\sigma$ the trivial set $\{0\}$ is a fibre.
\end{enumerate}
When they hold, $\{0\}$ is the only fibre at every dead stratum, so $(M,B_0)$ carries
exactly one fibre system.
\end{proposition}
\begin{proof}
For a live stratum the condition of Definition~\ref{def:gcs} holds automatically. So $\CC$ is a generalized cycle system if and only if every dead $\sigma$ has $U_\sigma$ a basis of $M|A_\sigma$. For a dead $\sigma$ the set $U_\sigma$ is independent, so this says that $U_\sigma$ spans $A_\sigma$. On the other hand $r(K_\sigma)=0$ if and only if the
empty set is the only basis of $K_\sigma$, which again says that $U_\sigma$ spans $A_\sigma$. This proves (a)$\Leftrightarrow$(b).

If $\{0\}$ is a fibre at $\sigma$ then $h(K_\sigma)=(1)$ by condition (i), so $r(K_\sigma)=0$, giving (c)$\Rightarrow$(b). Conversely, if $r(K_\sigma)=0$ then $\{0\}$ is a fibre at $\sigma$, and the only one, by Lemma~\ref{lem:zerofibre}. This gives (b)$\Rightarrow$(c) and the final clause.
\end{proof}

\begin{corollary}\label{cor:specialise}
Let $M$ be a matroid with a basis $B_0$ and fundamental circuits $\CC=(C_i)_{i\in R}$.
\begin{enumerate}[label=(\alph*)]
\item If $\CC$ is a cycle system, there are no dead strata, the empty fibre system is
      the only one, and $\PP^*(F)=\PP^*(\CC)$ is the set of coparking functions of
      \cite{CycleSystems}.
\item If $\CC$ is a generalized cycle system, the unique fibre system of
      Proposition~\ref{prop:gcs-trivial} gives
      \[
        \PP^*(F)=\Bigl\{a\in\N^R:\
        \begin{array}{l}
          \text{no live stratum can be fired from }a,\\[2pt]
          a|_\tau=b_\tau\ \text{for every dead }\tau\text{ that can be fired from }a
        \end{array}\Bigr\}.
      \]
\end{enumerate}
In both cases $\PP^*(F)$ is a pure multicomplex with degree sequence $h(M)$, by
Theorem~\ref{thm:main}.
\end{corollary}
\begin{proof}
(a) If $\CC$ is a cycle system then every $U_\sigma$ is dependent, so there are no dead strata and a fibre system has nothing to assign. Definition~\ref{def:relaxed} then says $a\in\PP^*(F)$ if and only if no nonempty $\tau$ can be fired from $a$, which is equivalent to $a \in \PP^*(\CC)$.

(b) By Proposition~\ref{prop:gcs-trivial} every dead stratum has the fibre $\{0\}$, so the condition $a|_\tau-b_\tau\in F(\tau)$ of Definition~\ref{def:relaxed} reads $a|_\tau=b_\tau$. Hence $a\in\PP^*(F)$ if and only if no live $\tau$ can be fired from $a$ and every dead
$\tau$ that can be fired from $a$ has $a|_\tau=b_\tau$, which is the displayed description.
\end{proof}

Part (b) is the pure multicomplex theorem for the generalized cycle systems formed by the fundamental circuits of a basis. This case had not been proved before, and whether it extends to an arbitrary collection of cycles is Question~\ref{q:cycles}.

\subsection{The truncated \texorpdfstring{$h$}{h}-polynomial}\label{sec:truncated}

By Lemma~\ref{lem:mono}(d), the trivial set $\{0\}$ satisfies condition (ii) of
Definition~\ref{def:fibre} at every dead stratum, whatever its rank. So the system $F_0$ with $F_0(\tau)=\{0\}$ at every dead $\tau$ is a
fibre system in every respect except condition (i). Lemmas~\ref{lem:C} and
\ref{lem:D} do not use (i), so they apply to $F_0$. Put
\[
  \PP^*_0(M,B_0)=\PP^*(F_0),
\]
the set displayed in Corollary~\ref{cor:specialise}(b), and
\[
  \tilde h(M,B_0)=\sum_{a\in\PP^*_0}t^{|a|},
\]
the \emph{truncated $h$-polynomial} of $(M,B_0)$: the coparking inequalities on every stratum,
with the fibre forced to be trivial wherever the recursion stops. The truncated $h$-polynomial
depends on $B_0$ and not only on $M$.

\begin{proposition}[The gap is the fibres]\label{prop:gap}
$\PP^*_0(M,B_0)$ is a multicomplex and $\PP^*_0\subseteq\PP^*(F)$ for every fibre system $F$. Summing over $\tau=\emptyset$ and all dead strata,
\[
  \tilde h(M,B_0)=\sum_{\tau}t^{|b_\tau|}E_\tau(t),\qquad
  h(M)-\tilde h(M,B_0)=\sum_{\tau\ne\emptyset}t^{|b_\tau|}\bigl(h(K_\tau)-1\bigr)E_\tau(t).
\]
The gap $h(M)-\tilde h(M,B_0)$ has nonnegative coefficients. Moreover it vanishes if and only if every dead node has rank zero. By Proposition~\ref{prop:gcs-trivial} this happens if and only if the fundamental circuits of $B_0$ form a
generalized cycle system.
\end{proposition}
\begin{proof}
By Lemma~\ref{lem:mono}(d) the trivial fibre satisfies (ii) at every dead stratum.
Lemmas~\ref{lem:C} and~\ref{lem:D} use only (ii) and closure under
division, so $\PP^*_0$ is a multicomplex and \eqref{eq:decomp} holds for $F_0$. As the degree
sequence of $F_0(\tau)=\{0\}$ is $1$, \eqref{eq:degvec} gives the first formula, and
subtracting it from Theorem~\ref{thm:B} gives the second. Since $0\in F(\tau)$ for every fibre
system $F$, $\PP^*_0\subseteq\PP^*(F)$. The gap has nonnegative coefficients because
$h_0(K_\tau)=1$, and since $E_\tau(t)\ne0$ by Corollary~\ref{cor:extpure} it vanishes if and only if every
$h(K_\tau)=1$, if and only if every dead node has rank zero, if and only if $\CC$ is a generalized cycle system by
Proposition~\ref{prop:gcs-trivial}.
\end{proof}

Purity of $\PP^*_0$ is not among the equivalent conditions. On $G_0$ with its tree $B_0$ the gap is $t^5$, yet
$\PP^*_0(M(G_0),B_0)$ is pure: its $13$ maximal elements all have degree $5$, the element $b_\sigma$ of
the rank-one stratum $\sigma=\{24,35,45\}$ lying below $b_R$ \cite{Notebook}. So the truncated
$h$-polynomial $\tilde h(M(G_0),B_0)=(1,4,10,17,20,13)$ is a pure $O$-sequence.

\subsection{Two criteria for fibres}\label{sec:criteria}

Theorem~\ref{thm:main} asks for a fibre at every dead stratum. Two criteria decide the
question in many cases without any recursion. The first is necessary for a fibre to exist,
the second sufficient, and both depend only on the coparking baselines below the stratum. Fix a dead stratum
$\sigma$. A live $\tau\subsetneq\sigma$ \emph{kills} an offset $m\in\N^\sigma$ if $\tau$ can be fired from $b_\sigma+m$,
which by Lemma~\ref{lem:mono}(c) happens if and only if
$m|_\tau\ge\exc_\tau$, and $m$ is \emph{spared} if no live proper stratum kills it. By
Definition~\ref{def:fibre}(i) and (ii), every fibre at $\sigma$ is a pure multicomplex of
spared offsets with degree sequence $h(K_\sigma)$. So if the spared offsets contain no such
multicomplex, then $\sigma$ is obstructed for every fibre system on $(M,B_0)$, whatever the fibres
below. At the dead strata that can be fired, condition (ii) can exclude further offsets through the choices made at lower strata,
which is why the criterion is not sufficient.

The sufficient criterion is Lemma~\ref{lem:baseline}, which is
equally choice-free: an unconditional fibre is a fibre for every fibre system below
$\sigma$. Its hypothesis says that every
$m\in F(\sigma)$ is spared and that $b_\sigma+m$ agrees with the baseline on every dead stratum that
can be fired from it.

In checking the hypothesis one needs only the \emph{thresholds} of the proper strata: for $\tau\subsetneq\sigma$
put $q_\tau=\prod_{j\in\tau}x_j^{\exc_\tau(j)}$. Then $\tau$ can be fired from $b_\sigma+m$ if and only
if $q_\tau$ divides $x^m$, and $\exc_\tau(j)=|\Delta_\tau\cap C_j|$ counts the shared elements
whose owner set meets $\tau$ exactly in $j$ (Lemma~\ref{lem:mono}). Only thresholds of degree at
most $r(K_\sigma)$ can matter, since every offset in a fibre has degree at most $r(K_\sigma)$ by Definition~\ref{def:fibre}(i).

\section{Combinatorial criteria on radius-two graphs}\label{sec:radius2}

The model of Section~\ref{sec:model} reduces Stanley's conjecture for $(G,B_0)$ to the existence
of fibres at the dead strata. In this section we specialise to graphs of radius two with the
breadth-first tree from a centre, and translate the deadness of a stratum and the rank of a dead node into
statements one can read off the graph. Sections~\ref{sec:roots} to~\ref{sec:pathnf} hold for every radius-two root, while Sections~\ref{sec:killers} and~\ref{sec:htwo} specialise to roots with three and with two fertile hubs. Together they are the tools of Section~\ref{sec:cases}, which does not need all of their generality. We have kept it because the original aim of the project was the conjecture for all radius-two graphs with the breadth-first tree (Section~\ref{sec:outlook}), and the general results of this section are the foundation for further work in that direction.

\subsection{Roots}\label{sec:roots}

\begin{definition}\label{def:root}
A \emph{radius-two root} is a pair $(G,B_0)$ where $G$ is a connected simple graph of radius at
most two with a distinguished vertex $o$ of eccentricity at most two, the \emph{centre}, and $B_0$ is a
breadth-first spanning tree of $G$ rooted at $o$. The neighbours of $o$ are the \emph{hubs} and
the vertices at distance two the \emph{outer vertices} or \emph{children}. The tree edges are the
\emph{spines} $s_h=oh$, one per hub, and the \emph{legs} $\ell_c=p(c)c$, one per child, where
$p(c)$ is the tree parent of $c$. A hub is \emph{fertile} if it has a child in $B_0$ and
\emph{barren} otherwise. We write $H$ for the number of fertile hubs, and the \emph{branch} of a hub
is $B_h=\{h\}\cup\{c:p(c)=h\}$.
\end{definition}

Every edge of $G$ joins two vertices whose distances from $o$ differ by at most one, so a red
edge joins two hubs, a hub and a child, or two children, and no red edge is incident with $o$.
The tree path of a red edge has at most four edges, so $|C_i|\le5$ for every $i\in R$.

\subsection{Multiplicities and dead strata}

For a stratum $\sigma$ write $\deg_\sigma(v)$ for the number of edges of $\sigma$ at a vertex
$v$ and $\out_\sigma(h)$ for the number of edges of $\sigma$ with exactly one end in the branch
$B_h$.

\begin{lemma}[Multiplicity rules]\label{lem:mult}
For every child $c$ and every hub $h$,
$m_\sigma(\ell_c)=\deg_\sigma(c)$ and $m_\sigma(s_h)=\out_\sigma(h)$ (Figure~\ref{fig:mult}).
\end{lemma}
\begin{proof}
Given a red edge $i$, the circuit $C_i$ contains $\ell_c=p(c)c$ if and only if $i$ has $c$ as an endpoint, since $c$ is a leaf of $B_0$. Similarly, $C_i$ contains $s_h=oh$ if and only if one end of $i$ lies in $B_h$ and the other does not, since $B_h$ is the set of vertices whose tree path to $o$ passes through $oh$. Counting over $i\in\sigma$ gives both formulas.
\end{proof}

Lemma~\ref{lem:mult} gives the following description of $U_\sigma$ and of the ground set of $K_\sigma$.

\begin{corollary}\label{cor:split}
$\ell_c\in U_\sigma$ if and only if $\deg_\sigma(c)=1$, $\ell_c\in A_\sigma\setminus U_\sigma$ if and only if
$\deg_\sigma(c)\ge2$, and $\ell_c\notin A_\sigma$ if and only if $\deg_\sigma(c)=0$, and likewise for $s_h$ with $\out_\sigma(h)$ in place of $\deg_\sigma(c)$. Hence
\[
  U_\sigma=\sigma\ \cup\ \{\ell_c:\deg_\sigma(c)=1\}\ \cup\ \{s_h:\out_\sigma(h)=1\},
\]
and with
\begin{align*}
  I(\sigma)&=\#\{c\text{ child}:\deg_\sigma(c)\ge2\}+\#\{h\text{ barren}:\out_\sigma(h)\ge2\},\\
  \phi(\sigma)&=\#\{h\text{ fertile}:\out_\sigma(h)\ge2\},
\end{align*}
the ground set of $K_\sigma$ has $I(\sigma)+\phi(\sigma)$ elements.
\end{corollary}

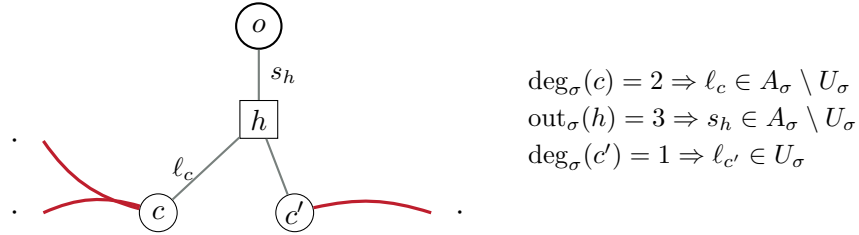
\begin{figure}[h]\centering
\begin{tikzpicture}[scale=0.95]
\node[ctr] (o) at (0,2.2) {$o$};
\node[hub] (h) at (0,0.9) {$h$};
\node[ch] (c1) at (-1.4,-0.4) {$c$}; \node[ch] (c2) at (0.5,-0.4) {$c^{\prime}$};
\draw[tr] (o)--(h) node[midway,right,black,lab]{$s_h$};
\draw[tr] (h)--(c1) node[midway,left,black,lab]{$\ell_c$};
\draw[tr] (h)--(c2);
\draw[rd] (c1) to[bend right=20] (-3.0,-0.4); \draw[rd] (c1) to[bend left=20] (-3.0,0.6);
\node at (-3.4,-0.4) {$\cdot$}; \node at (-3.4,0.6) {$\cdot$};
\draw[rd] (c2) to[bend left=15] (2.4,-0.4); \node at (2.8,-0.4) {$\cdot$};
\node[anchor=west,align=left,lab] at (3.6,0.9)
 {$\deg_\sigma(c)=2\Rightarrow \ell_c\in A_\sigma\setminus U_\sigma$\\[2pt]
  $\mathrm{out}_\sigma(h)=3\Rightarrow s_h\in A_\sigma\setminus U_\sigma$\\[2pt]
  $\deg_\sigma(c^{\prime})=1\Rightarrow \ell_{c^{\prime}}\in U_\sigma$};
\end{tikzpicture}
\caption{Lemma~\ref{lem:mult}: a leg counts the red edges at its child, and a spine counts the red
edges leaving its branch. Here three red edges leave $B_h$, two at $c$ and one at $c'$.}
\label{fig:mult}
\end{figure}

The following construction is motivated by the \emph{blueprint} of \cite{Triconed}.

\begin{definition}
Let $\sigma$ be a forest. The \emph{attachment graph} $A(\sigma)$ is the multigraph whose
vertices are the connected components of $\sigma$ together with the vertices of $G$ not covered
by $\sigma$, and whose edges are the elements of $U_\sigma\setminus\sigma$, namely the legs $\ell_c$
with $\deg_\sigma(c)=1$ and the spines $s_h$ with $\out_\sigma(h)=1$, each drawn between the
images of its endpoints.
\end{definition}

\begin{theorem}[Dead strata]\label{thm:indep}
A stratum $\sigma$ is dead if and only if $\sigma$ is a forest and $A(\sigma)$ is a forest.
\end{theorem}
\begin{proof}
If $\sigma$ is not a forest then neither is $U_\sigma\supseteq\sigma$. So assume $\sigma$ is a forest. Contracting the edges of $\sigma$ in the graph $(V,U_\sigma)$ identifies each component of $\sigma$ to a point and leaves the edges $U_\sigma\setminus\sigma$, which form the attachment graph $A(\sigma)$ by Corollary~\ref{cor:split}. Contracting a forest of edges preserves being a forest and, since every cycle of $U_\sigma$ uses an edge outside the forest $\sigma$, also preserves not being one. Hence $U_\sigma$ is a forest if and only if $A(\sigma)$ is.
\end{proof}

\begin{example}[The criterion on the running example]\label{ex:attach}
In $G_0$ take $\sigma=\{24,35,45\}$, the path $2\,4\,5\,3$. Its two leaves are the fertile hub $2$ and the
barren hub $3$, so no leg joins $U_\sigma$. The branches are $B_1=\{1,4\}$, $B_2=\{2,5\}$,
$B_3=\{3\}$, with $\out_\sigma(1)=2$ (edges $24,45$), $\out_\sigma(2)=3$ and $\out_\sigma(3)=1$,
so the only attachment is the spine $s_3=03$. The attachment graph (Figure~\ref{fig:attach},
left) is the single edge from the component $\{2,4,5,3\}$ to $0$, a forest: $\sigma$ is
dead, as found by hand in Example~\ref{ex:running}. Now take $\tau=\{24,45\}$, the path
$2\,4\,5$. The leaf $5$ is a child with $\deg_\tau(5)=1$, so its leg $\ell_5=25$ attaches, and since
$\out_\tau(1)=\out_\tau(2)=2$ no spine does. The leg $\ell_5$ joins $5$ to its
parent $2$, and both lie in the component. So $\ell_5$ is a loop of the attachment graph
(Figure~\ref{fig:attach}, right) and $\tau$ is live. Indeed $U_\tau=\{24,45,25\}$ is the
triangle $2\,4\,5$.
\end{example}

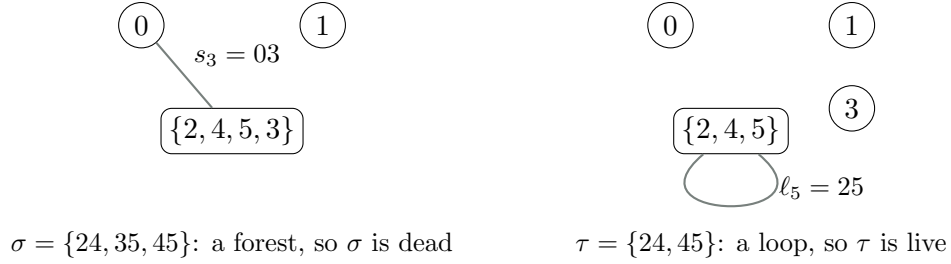
\begin{figure}[h]\centering
\begin{tikzpicture}[scale=1,every node/.style={font=\small},
  nd/.style={draw,circle,minimum size=6mm,inner sep=0pt,fill=white},
  cp/.style={draw,rounded corners,minimum height=6mm,inner sep=3pt,fill=white}]
\begin{scope}
\node[nd] (o) at (0,1.4) {$0$}; \node[nd] (h) at (2.4,1.4) {$1$};
\node[cp] (K) at (1.2,0) {$\{2,4,5,3\}$};
\draw[tr] (o)--(K) node[midway,above right=-1pt,black,lab]{$s_3=03$};
\node[lab] at (1.2,-1.5) {$\sigma=\{24,35,45\}$: a forest, so $\sigma$ is dead};
\end{scope}
\begin{scope}[xshift=7cm]
\node[nd] (o) at (0,1.4) {$0$}; \node[nd] (h1) at (2.4,1.4) {$1$}; \node[nd] (h3) at (2.4,0.3) {$3$};
\node[cp] (K) at (0.8,0) {$\{2,4,5\}$};
\draw[tr] (K) to[out=-40,in=-140,looseness=5] (K); \node[lab] at (2.0,-0.75) {$\ell_5=25$};
\node[lab] at (1.2,-1.5) {$\tau=\{24,45\}$: a loop, so $\tau$ is live};
\end{scope}
\end{tikzpicture}
\caption{Example~\ref{ex:attach}: attachment graphs in $G_0$. Deadness is decided after
contracting each red component to a point.}\label{fig:attach}
\end{figure}

\subsection{Rank, path normal form and the shape of a dead node}\label{sec:pathnf}

\begin{theorem}[Rank formula and bound]\label{thm:rank}
For a dead stratum $\sigma$ with $c(\sigma)$ components,
\[
  r(K_\sigma)=I(\sigma)+\phi(\sigma)-|\sigma|\qquad\text{and}\qquad
  r(K_\sigma)\le \phi(\sigma)-c(\sigma).
\]
\end{theorem}
\begin{proof}
Recall from Lemma~\ref{lem:rank} that $r(K_\sigma)=|A_\sigma|-|U_\sigma|-|\sigma|$, and from Corollary~\ref{cor:split} that $|A_\sigma\setminus U_\sigma|=I(\sigma)+\phi(\sigma)$. This proves the first claim.

For the second claim, note that for a barren hub $h$ the branch is $B_h=\{h\}$, so $\out_\sigma(h)=\deg_\sigma(h)$, and hence $I(\sigma)$ is at most the number of vertices of degree at least two in the forest $\sigma$. That forest has $|\sigma|+c(\sigma)$ vertices and at least $2c(\sigma)$ leaves, so at most $|\sigma|-c(\sigma)$ of its vertices have degree at least two.
Thus $I(\sigma)\le|\sigma|-c(\sigma)$, and the bound follows from the first claim.
\end{proof}

Theorem~\ref{thm:rank} has the following consequences.

\begin{corollary}\label{cor:rankH}
Every dead node of a root with $H\le2$ fertile hubs has rank at most one. A dead node of rank
two has $\phi(\sigma)\ge c(\sigma)+2\ge3$, so a rank-two stratum in a root with $H=3$ is connected. A connected stratum has $r(K_\sigma)\le H-1$.
\end{corollary}

\begin{proposition}[Path normal form]\label{prop:path}
Let $\sigma$ be a dead stratum with $r(K_\sigma)=d\in\{1,2\}$ and $\phi(\sigma)\le d+1$ (for
instance, any dead stratum of rank $d$ in a root with $H=d+1$). Then $\phi(\sigma)=d+1$ and
$\sigma$ is a path with no fertile hub as an internal vertex. Exactly $d+1$ fertile spines
lie in $A_\sigma\setminus U_\sigma$. The parent edge of every internal vertex $v$ of the
path lies in $A_\sigma\setminus U_\sigma$ with owner set exactly the two path edges at $v$.
\end{proposition}
\begin{proof}
By Theorem~\ref{thm:rank}, $d=r(K_\sigma)\le \phi(\sigma)-c(\sigma)\le d+1-c(\sigma)$, so $c(\sigma)=1$ and $\phi(\sigma)=d+1$. The first formula of Theorem~\ref{thm:rank} then gives $I(\sigma)=d+|\sigma|-(d+1)=|\sigma|-1$. On the other hand, a barren hub $h$ has $B_h=\{h\}$ and so $\out_\sigma(h)=\deg_\sigma(h)$, and the tree $\sigma$ has $|\sigma|+1$ vertices and at least two leaves, so
\[
  I(\sigma)\ \le\ \#\{v:\ \deg_\sigma(v)\ge2,\ v\text{ a child or a barren hub}\}
  \ \le\ \#\{v:\ \deg_\sigma(v)\ge2\}\ \le\ |\sigma|-1 .
\]
Hence all three inequalities are equalities. The last says the tree $\sigma$ has exactly two leaves, so it is a path. The middle says no fertile hub has degree at least two in $\sigma$, so no fertile hub is an internal vertex. The first says every internal vertex $v$ is counted in $I(\sigma)$. Its parent edge (the leg $\ell_v$ if $v$ is a child, the spine $s_v$ if $v$ is a
barren hub) lies in $A_\sigma\setminus U_\sigma$, and by Lemma~\ref{lem:mult} its owner set is the set of $\sigma$-edges at $v$, namely the two path edges at $v$. Finally $\phi(\sigma)=d+1$ says exactly $d+1$ fertile spines lie in $A_\sigma\setminus U_\sigma$.
\end{proof}

A stratum as in Proposition~\ref{prop:path} can therefore be recorded as a word, one letter per path vertex: a letter for each fertile hub, used for the children of that hub, and a letter $\word A$ for the barren hubs. The fertile hubs themselves may appear only at the two ends. The two endpoint \emph{anchors} of the path are the attachments of the two leaves: the parent for a leaf child, $o$ for a leaf barren hub, and the hub itself for a leaf fertile hub. Suppose from now on that the
root has exactly $d+1$ fertile hubs, so that $\out_\sigma(h)\ge2$ for every fertile hub $h$
and no fertile spine lies in $U_\sigma$. The same notion applies to a subpath $\tau$ of $\sigma$, whose anchors are the attachments of its two ends.

\begin{lemma}\label{lem:excess}
Let $\sigma$ be as in Proposition~\ref{prop:path}, in a root with exactly $d+1$ fertile hubs, and let $\tau\subseteq\sigma$.
\begin{enumerate}[label=(\alph*)]
\item Call a vertex an \emph{internal leaf} of $\tau$ if it is internal in $\sigma$ and has degree one in $\tau$. Then $\Delta_\tau$ consists of the parent edges of the internal leaves of $\tau$ and the spines $s_h$ of the fertile hubs $h$ with $\out_\tau(h)=1$. For an edge $e$ of $\tau$,
\[
\begin{split}
  \exc_\tau(e)
    &= \#\{v\text{ an end of }e:\ v\text{ an internal leaf of }\tau\} \\
    &\quad + \#\{h\text{ fertile}:\ \out_\tau(h)=1,\
       e\text{ the edge of }\tau\text{ leaving }B_h\}.
\end{split}
\]
\item If $\tau$ is a subpath with $\out_\tau(h)\ne1$ for every fertile hub $h$, for instance $\tau=\sigma$, then $U_\tau$ consists of $\tau$ and the attachments of its two ends, and $\tau$ is dead if and only if its two anchors differ.
\end{enumerate}
\end{lemma}
\begin{proof}
(a) Recall that $\exc_\tau(e)=b_\tau(e)-b_\sigma|_\tau(e)=|C_{e}\cap U_\tau|-|C_{e}\cap U_\sigma|$,
which by Lemma~\ref{lem:mono}(c) equals $|\Delta_\tau\cap C_{e}|$, the number of edges of $C_{e}$ that are uniquely owned in $\tau$ but shared in $\sigma$.

Since $\Delta_\tau\subseteq A_\sigma\setminus U_\sigma$, we first describe the latter set. By Corollary~\ref{cor:split} it consists of the legs $\ell_c$ with $\deg_\sigma(c)\ge2$ and the spines $s_h$ with $\out_\sigma(h)\ge2$. As $\sigma$ is a path with no fertile hub internal (Proposition~\ref{prop:path}), and $\out_\sigma(h)=\deg_\sigma(h)$ for a barren hub $h$ because $B_h=\{h\}$, these are the parent edges of the internal vertices of $\sigma$
together with the $d+1$ fertile spines.

An element $y$ of $A_\sigma\setminus U_\sigma$ lies in $\Delta_\tau$ if and only if $m_\tau(y)=1$. By Lemma~\ref{lem:mult}, the parent edge of an internal vertex $v$ has $m_\tau=\deg_\tau(v)$ (for a barren hub $v$ the parent edge is the spine $s_v$ and $\out_\tau(v)=\deg_\tau(v)$), and the spine $s_h$ of a fertile hub has $m_\tau=\out_\tau(h)$. So $\Delta_\tau$ consists of the parent edges of the internal leaves of $\tau$ and the fertile spines with $\out_\tau(h)=1$, as claimed.

It remains to count, for an edge $e$ of $\tau$, the elements of $\Delta_\tau$ that lie in $C_e$. The parent edge of an internal leaf $v$ lies in $C_e$ if and only if $e$ is the unique edge of $\tau$ at $v$, that is, $v$ is an end of $e$. The spine $s_h$ with $\out_\tau(h)=1$ lies in $C_e$ if and only if $e$ is the edge of $\tau$ leaving $B_h$. Both facts come from the proof of Lemma~\ref{lem:mult}, and they give the two terms.

(b) No fertile spine lies in $U_\tau$, since $\out_\tau(h)\ne1$, and a vertex internal to the path $\tau$ has degree two in $\tau$, so by Corollary~\ref{cor:split} the elements of $U_\tau\setminus\tau$ are the attachments at the ends of $\tau$ that are not fertile hubs. This is the first claim. For the second, recall that the attachment graph $A(\tau)$ has as vertices the single component $\tau$ and the vertices of $G$ not covered by $\tau$, and as edges the elements of $U_\tau\setminus\tau$. Here these edges are the leg to the parent hub for an end that is a child and the spine to $o$ for an end that is a barren hub, each joining the component to the anchor of that end. An end that is a fertile hub contributes no edge and is anchored at itself, a vertex of the component. So $A(\tau)$ has at most two edges, and it has a cycle if and only if the two anchors coincide, a loop when one anchor is the other end and a pair of parallel edges otherwise. By Theorem~\ref{thm:indep}, $\tau$ is dead if and only if its two anchors differ.
\end{proof}

So $U_\sigma$ is the path together with the attachments of its two leaves, $\sigma$ is dead if and only if its two anchors differ, and the dead node is read off the anchors.

\begin{proposition}[Dead nodes of a path]\label{prop:startri}
Let $\sigma$ be as in Proposition~\ref{prop:path}, in a root with exactly $d+1$ fertile hubs, with
distinct anchors. Contracting $U_\sigma$
collapses the path and its endpoint attachments to a single vertex $w$, identified with the two
anchors, and the elements of $K_\sigma$ are the $d+1$ fertile spines and the parent edges of the
internal vertices, the latter each joining the parent to $w$. Write $n_h$ for the number of internal
vertices of $\sigma$ that are children of the fertile hub $h$, and $n_{\word A}$ for the number
of internal vertices of $\sigma$ that are barren hubs.
\begin{enumerate}[label=(\alph*)]
\item If $d=1$ with fertile hubs $1,2$ and anchors $o$ and $1$, then $K_\sigma$ is a parallel
      class of size $1+n_2$ plus loops, and $h(K_\sigma)=(1,n_2)$.
\item If $d=1$ with anchors $1$ and $2$, then $K_\sigma$ is a parallel class of size
      $2+n_{\word A}$ plus loops, and $h(K_\sigma)=(1,1+n_{\word A})$.
\item If $d=2$ with fertile hubs $1,2,3$ and anchors $o$ and $1$, the simplification of
      $K_\sigma$ is a path of two parallel classes, of sizes $p=1+n_2$ and $q=1+n_3$, and
      $h(K_\sigma)=(1,p+q-2,(p-1)(q-1))$.
\item If $d=2$ with anchors $1$ and $2$ and third hub $3$, the simplification of $K_\sigma$ is a
      triangle with parallel classes of sizes $2+n_{\word A}$, $1$ and $n_3$. Writing the three
      sizes as $(1,p,q)$, $h(K_\sigma)=(1,p+q-1,pq)$ (a \emph{triangle node}).
\end{enumerate}
\end{proposition}
\begin{proof}
By Proposition~\ref{prop:path}, $U_\sigma$ consists of the path $\sigma$ and its endpoint
attachments, which gives the description of the contraction. After the contraction, a spine $s_h$ joins $o$
to $h$ and the parent edge of an internal vertex joins its parent, a hub or $o$, to $w$. An
element is a loop exactly when both of its ends are merged into $w$.

If one anchor is $o$ and the other is the fertile hub $1$, then $w$ is merged with $o$ and $1$.
The spine $s_1$, the legs of the internal children of $1$ and the spines of the internal barren
hubs are loops. For each other fertile hub $h$, its spine together with the legs of its internal children forms a parallel class $w$--$h$ of size $1+n_h$. For $d=1$ this is one class, giving
(a), and for $d=2$ two classes on the path $2\,w\,3$, giving (c).

If the anchors are the fertile hubs $1$ and $2$, then $w$ is merged with $1$ and $2$. The legs of
the internal children of $1$ and $2$ are loops, while $s_1$, $s_2$ and the spines of the internal
barren hubs join $o$ to $w$, a parallel class of size $2+n_{\word A}$. For $d=1$ this is all of
$K_\sigma$, giving (b). For $d=2$ the third hub $3$ contributes the spine $s_3$ from $o$ to $3$
and the legs of its internal children from $3$ to $w$, which form a class of size $n_3$. So the
simplification is the triangle $o\,w\,3$ with classes $2+n_{\word A}$, $1$, $n_3$, giving (d).

Since $K_\sigma$ has no coloops by Lemma~\ref{lem:rank}, a class that is a bridge of the simplification has size at least two, so $n_2\ge1$ in (a) and $n_2,n_3\ge1$ in (c). In (d) the class of size $n_3$ is nonempty, since otherwise $s_3$ would be a coloop, so $n_3\ge1$. The $h$-vectors are those of the small graphs of Figures~\ref{fig:rankone} and~\ref{fig:startri}: a parallel class of size $n$ plus loops has $h=(1,n-1)$, and (c) and (d) follow from \eqref{eq:tutte}.
\end{proof}

\begin{figure}[h]\centering
\begin{tikzpicture}[scale=0.9,every node/.style={font=\small},
  nd/.style={draw,circle,minimum size=6mm,inner sep=0pt,fill=white}]
\begin{scope}
\node at (2.1,2.3) {\bfseries (a) one anchor is $o$};
\draw[rd] (0,1.4)--(4.2,1.4);
\foreach \x in {0,1.05,2.1,3.15,4.2} \fill[redge] (\x,1.4) circle (2pt);
\node[lab] at (0,1.75) {anchor $o$}; \node[lab] at (4.2,1.75) {anchor $1$};
\node[nd] (S) at (1.1,0) {$w$}; \node[nd] (j) at (3.1,0) {$2$};
\draw (S) to[bend left=30] (j); \draw (S) -- (j); \draw (S) to[bend right=30] (j);
\draw (S) to[out=150,in=210,looseness=8] (S);
\node at (2.1,0.65) {$1+n_2$};
\node[lab] at (2.1,-0.8) {$h=(1,n_2)$};
\end{scope}
\begin{scope}[xshift=8cm]
\node at (2.1,2.3) {\bfseries (b) two fertile anchors};
\draw[rd] (0,1.4)--(4.2,1.4);
\foreach \x in {0,1.05,2.1,3.15,4.2} \fill[redge] (\x,1.4) circle (2pt);
\node[lab] at (0,1.75) {anchor $1$}; \node[lab] at (4.2,1.75) {anchor $2$};
\node[nd] (o) at (1.1,0) {$o$}; \node[nd] (S) at (3.1,0) {$w$};
\draw (o) to[bend left=30] (S); \draw (o) -- (S); \draw (o) to[bend right=30] (S);
\draw (S) to[out=30,in=-30,looseness=8] (S);
\node at (2.1,0.65) {$2+n_{\mathsf A}$};
\node[lab] at (2.1,-0.8) {$h=(1,1+n_{\mathsf A})$};
\end{scope}
\end{tikzpicture}
\caption{Proposition~\ref{prop:startri}(a) and (b): the rank-one dead nodes at $\phi\le2$. The
anchors of the path decide the dead node, a single parallel class plus loops.}\label{fig:rankone}
\end{figure}
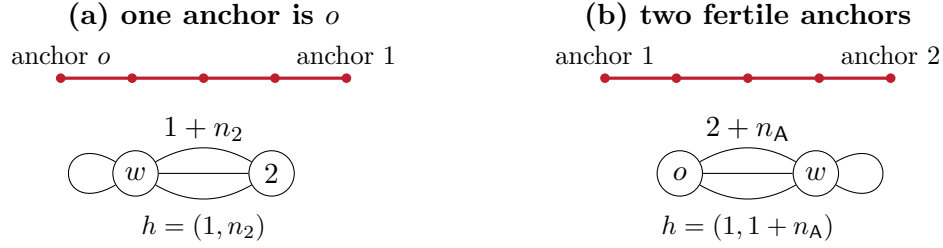

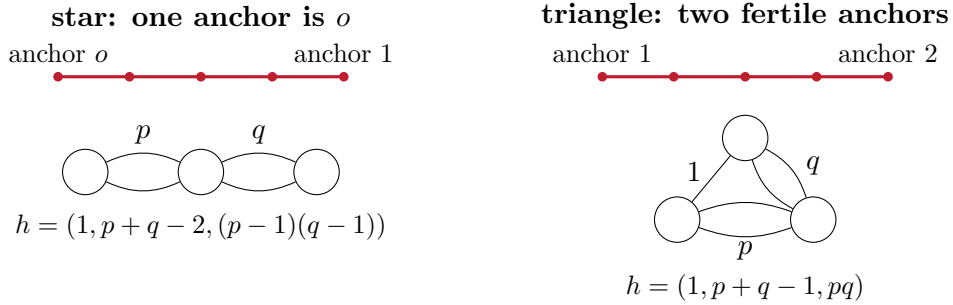
\begin{figure}[h]\centering
\begin{tikzpicture}[scale=0.9,every node/.style={font=\small},
  nd/.style={draw,circle,minimum size=6mm,inner sep=0pt,fill=white}]
\begin{scope}
\node at (2.1,2.3) {\bfseries star: one anchor is $o$};
\draw[rd] (0,1.4)--(4.2,1.4);
\foreach \x in {0,1.05,2.1,3.15,4.2} \fill[redge] (\x,1.4) circle (2pt);
\node[lab] at (0,1.75) {anchor $o$}; \node[lab] at (4.2,1.75) {anchor $1$};
\node[nd] (a) at (0.4,0) {}; \node[nd] (b) at (2.1,0) {}; \node[nd] (c) at (3.8,0) {};
\draw (a) to[bend left=25] (b); \draw (a) to[bend right=25] (b);
\draw (b) to[bend left=25] (c); \draw (b) to[bend right=25] (c);
\node at (1.25,0.55) {$p$}; \node at (2.95,0.55) {$q$};
\node[lab] at (2.1,-0.8) {$h=(1,p+q-2,(p-1)(q-1))$};
\end{scope}
\begin{scope}[xshift=8cm]
\node at (2.1,2.3) {\bfseries triangle: two fertile anchors};
\draw[rd] (0,1.4)--(4.2,1.4);
\foreach \x in {0,1.05,2.1,3.15,4.2} \fill[redge] (\x,1.4) circle (2pt);
\node[lab] at (0,1.75) {anchor $1$}; \node[lab] at (4.2,1.75) {anchor $2$};
\node[nd] (x) at (2.1,0.5) {}; \node[nd] (y) at (1.1,-0.7) {}; \node[nd] (z) at (3.1,-0.7) {};
\draw (x)--(y); \draw (x) to[bend left=22] (z); \draw (x) to[bend right=22] (z);
\draw (y) to[bend left=20] (z); \draw (y) to[bend right=20] (z);
\node at (1.35,0.0) {$1$}; \node at (3.1,0.1) {$q$}; \node at (2.1,-1.15) {$p$};
\node[lab] at (2.1,-1.7) {$h=(1,p+q-1,pq)$};
\end{scope}
\end{tikzpicture}
\caption{Proposition~\ref{prop:startri}(c) and (d): the rank-two dead nodes at $\phi\le3$. As in
Figure~\ref{fig:rankone}, the anchors decide the dead node. On the Petersen root all twelve
rank-two strata are triangle nodes.}\label{fig:startri}
\end{figure}

\subsection{Killers on a path}\label{sec:killers}

The following parity fact is used repeatedly for subpaths of a path stratum.

\begin{lemma}[Parity]\label{lem:parity}
Let $\tau$ be a path in $G$ and $h$ a hub. Then $\out_\tau(h)$ is odd if and only if exactly
one of the two ends of $\tau$ lies in $B_h$.
\end{lemma}
\begin{proof}
Walking along $\tau$, the edges with exactly one end in $B_h$ are the ones at which the walk
crosses the boundary of $B_h$, and the walk ends on the side it started from if and only if it
crosses an even number of times.
\end{proof}

Fix a root with three fertile hubs and a dead stratum $\sigma$ with $r(K_\sigma)=2$, which is a path by
Proposition~\ref{prop:path}. Let $\tau\subsetneq\sigma$ be a nonempty proper subset of its
edges, a disjoint union of subpaths which we call the components of $\tau$. By
Lemma~\ref{lem:excess}(a), the set $\Delta_\tau$ consists of the parent edges
of the internal leaves of $\tau$ and the spines $s_h$ of the fertile hubs with $\out_\tau(h)=1$.
The threshold $q_\tau$ has degree $|\Delta_\tau|$, so only a $\tau$ with at most two internal
leaves can kill an offset of degree at most two. Such a $\tau$ is live if and only if its components, closed by
the attachments of their endpoints and by the spines in $\Delta_\tau$, contain a cycle
(Theorem~\ref{thm:indep}). The two configurations met in Sections~\ref{sec:cases}
and~\ref{sec:outlook} are the following. A \emph{return} is a single component whose two
ends have the same anchor. When the component is a prefix or a suffix of the path it kills the
linear $x_i$, where $e_i$ is the edge of $\tau$ at its internal leaf, and otherwise it kills the
product $x_ix_j$ of the edges at its two internal leaves ($x_i^2$ when the component is a single
edge). A \emph{double bridge} is a prefix and a suffix joining the same pair of anchors (two
components of any other kind have at least three internal leaves). It kills the product
$x_ix_j$ of the edges at its two internal leaves and neither square. In both configurations
$\out_\tau(h)$ is even for every fertile hub $h$ by Lemma~\ref{lem:parity}, applied to each component,
so no fertile spine lies in $\Delta_\tau$ and the thresholds are as stated. For such a stratum, with
$h(K_\sigma)=(1,h_1,h_2)$, the necessary condition of Section~\ref{sec:criteria} reads: $\sigma$
admits a fibre for some fibre system only if some set $L$ of $h_1$ spared coordinates carries
$h_2$ spared quadratics supported on $L$ and covering every element of $L$.

\subsection{Rank-one nodes at \texorpdfstring{$H\le2$}{H at most 2}: the spared linears}\label{sec:htwo}

In this subsection the root has exactly two fertile hubs, which we call $1$ and $2$, and
$\sigma$ is a dead stratum with $r(K_\sigma)=1$. By Proposition~\ref{prop:path},
$\sigma$ is a path $v_0v_1\cdots v_m$ with edges $e_i=v_{i-1}v_i$, no fertile hub is internal,
and $\out_\sigma(1)\ge2$, $\out_\sigma(2)\ge2$. Colour a vertex $v$ of the path by
$\kappa(v)=h$ if $v$ is a child of the fertile hub $h$ or $v=h$, and by $\kappa(v)=\word A$ if $v$ is a
barren hub, and call $\kappa(v_0)\kappa(v_1)\cdots \kappa(v_m)$ the \emph{colour word} of $\sigma$. Since
$B_h$ is the set of vertices coloured $h$, for $\tau\subseteq\sigma$ the number $\out_\tau(h)$
is the number of edges of $\tau$ whose two ends are coloured differently, one of them $h$. The
anchor of an endpoint is $h$ if it is coloured $h$ and $o$ if it is coloured $\word A$, and
$\sigma$ is dead if and only if its two anchors differ (Lemma~\ref{lem:excess}(b)). For
$1\le j\le m$ write $P_j=\{e_1,\dots,e_j\}$ and $S_j=\{e_j,\dots,e_m\}$ for the prefix and the
suffix at $j$. The anchors of the path $P_j$ are those of $v_0$ and $v_j$, and the anchors of
$S_j$ are those of $v_{j-1}$ and $v_m$.

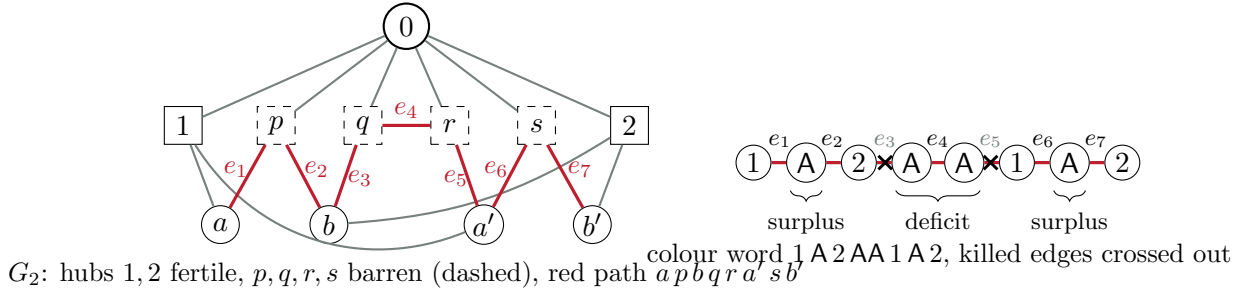
\begin{figure}[h]\centering
\begin{tikzpicture}[scale=0.82]
\begin{scope}
\node[ctr] (o) at (3.0,3.4) {$0$};
\node[hub] (h1) at (-0.6,1.8) {$1$}; \node[hub] (h2) at (6.6,1.8) {$2$};
\node[hub,dashed] (p) at (0.9,1.8) {$p$}; \node[hub,dashed] (q) at (2.3,1.8) {$q$};
\node[hub,dashed] (r) at (3.7,1.8) {$r$}; \node[hub,dashed] (s) at (5.1,1.8) {$s$};
\node[ch] (a) at (0.0,0.2) {$a$}; \node[ch] (b) at (1.75,0.2) {$b$};
\node[ch] (a2) at (4.25,0.2) {$a'$}; \node[ch] (b2) at (6.0,0.2) {$b'$};
\foreach \x in {h1,h2,p,q,r,s} \draw[tr] (o)--(\x);
\draw[tr] (h1)--(a); \draw[tr] (h1) to[bend right=40] (a2); \draw[tr] (h2)--(b2); \draw[tr] (h2) to[bend left=14] (b);
\draw[rd] (a)--(p); \draw[rd] (p)--(b); \draw[rd] (b)--(q); \draw[rd] (q)--(r); \draw[rd] (r)--(a2); \draw[rd] (a2)--(s); \draw[rd] (s)--(b2);
\node[rlab] at (0.25,1.1) {$e_1$}; \node[rlab] at (1.55,1.1) {$e_2$}; \node[rlab] at (2.25,0.95) {$e_3$};
\node[rlab] at (3.0,2.05) {$e_4$}; \node[rlab] at (3.8,0.95) {$e_5$}; \node[rlab] at (4.45,1.1) {$e_6$}; \node[rlab] at (5.8,1.1) {$e_7$};
\node[lab] at (3.0,-0.6) {$G_2$: hubs $1,2$ fertile, $p,q,r,s$ barren (dashed), red path $a\,p\,b\,q\,r\,a'\,s\,b'$};
\end{scope}
\begin{scope}[xshift=8.6cm,yshift=1.2cm]
\foreach \i/\c in {0/1,1/{\mathsf A},2/2,3/{\mathsf A},4/{\mathsf A},5/1,6/{\mathsf A},7/2}
  \node[draw,circle,inner sep=1.5pt,font=\small,fill=white] (w\i) at (\i*0.85,0) {$\c$};
\foreach \i/\j/\lab/\st in {0/1/e_1/spared,1/2/e_2/spared,2/3/e_3/killed,3/4/e_4/spared,4/5/e_5/killed,5/6/e_6/spared,6/7/e_7/spared}{
  \draw[rd] (w\i)--(w\j);}
\foreach \i/\lab in {0/e_1,1/e_2,3/e_4,5/e_6,6/e_7} \node[font=\scriptsize] at (\i*0.85+0.425,0.35) {$\lab$};
\foreach \i/\lab in {2/e_3,4/e_5} \node[font=\scriptsize,grey] at (\i*0.85+0.425,0.35) {$\lab$};
\foreach \i in {2,4} \draw[black,line width=1.1pt] (\i*0.85+0.425-0.11,-0.11)--(\i*0.85+0.425+0.11,0.11) (\i*0.85+0.425-0.11,0.11)--(\i*0.85+0.425+0.11,-0.11);
\draw[decorate,decoration={brace,mirror,amplitude=4pt}] (0.85*0.7,-0.45)--(0.85*1.3,-0.45) node[midway,below=4pt,font=\scriptsize] {surplus};
\draw[decorate,decoration={brace,mirror,amplitude=4pt}] (0.85*2.7,-0.45)--(0.85*4.3,-0.45) node[midway,below=4pt,font=\scriptsize] {deficit};
\draw[decorate,decoration={brace,mirror,amplitude=4pt}] (0.85*5.7,-0.45)--(0.85*6.3,-0.45) node[midway,below=4pt,font=\scriptsize] {surplus};
\node[lab] at (3.0,-1.5) {colour word $1\,\mathsf A\,2\,\mathsf A\mathsf A\,1\,\mathsf A\,2$, killed edges crossed out};
\end{scope}
\end{tikzpicture}
\caption{The root $G_2$ used in the examples of this subsection, and the colour word of its full
red path $\sigma=\{e_1,\dots,e_7\}$, with the killed edges crossed out and the blocks of barren
hubs marked.}\label{fig:htwo}
\end{figure}

\begin{example}\label{ex:htwo-setup}
Figure~\ref{fig:htwo} shows a root $G_2$ with centre $0$, fertile hubs $1$ (children $a,a'$)
and $2$ (children $b,b'$), barren hubs $p,q,r,s$, and the red path
$\sigma=\{e_1,\dots,e_7\}$ through $a\,p\,b\,q\,r\,a'\,s\,b'$. Its colour word is
$1\,\word A\,2\,\word A\word A\,1\,\word A\,2$ and its anchors are $1$ and $2$. As
$\out_\sigma(1)=\out_\sigma(2)=3$, the stratum $\sigma$ is a dead rank-one stratum of type (b) in
Proposition~\ref{prop:startri}: $K_\sigma$ is a parallel class of size $2+4$ plus loops and
$h(K_\sigma)=(1,5)$, with $b_\sigma=(2,1,1,1,1,1,2)$.
\end{example}

\begin{lemma}\label{lem:htwo-dom}
Let $1\le j\le m$ and let $\tau\subsetneq\sigma$ be a nonempty proper stratum. Then
$\tau$ can be fired from $b_\sigma+\mathbf e_j$ if and only if $\tau$ is the prefix $P_j$ with $j<m$ or
the suffix $S_j$ with $j>1$, and $\out_\tau(1)\ne1\ne\out_\tau(2)$. In that case
$b_\tau=b_\sigma|_\tau+\mathbf e_j$, so the offset of $b_\sigma+\mathbf e_j$ at $\tau$ is $0$.
\end{lemma}
\begin{proof}
By definition of the excess, $\tau$ can be fired from $b_\sigma+\mathbf e_j$ if and only if
$\exc_\tau\le\mathbf e_j|_\tau$. Since $\sum_{e_i\in\tau}\exc_\tau(e_i)=|\Delta_\tau|$ by
Lemma~\ref{lem:mono}(b,c), this says that $|\Delta_\tau|\le1$. If $\Delta_\tau$ has an element then that element lies in $C_{e_j}$ with $e_j\in\tau$. Summing
Lemma~\ref{lem:excess}(a) over $e_i\in\tau$ gives
\[
  |\Delta_\tau|=\#\{v\text{ internal in }\sigma:\ \deg_\tau(v)=1\}
  +\#\{h\in\{1,2\}:\ \out_\tau(h)=1\},
\]
since an internal vertex with $\deg_\tau(v)=1$ is an end of exactly one edge of $\tau$ and a
spine with $\out_\tau(h)=1$ is left by exactly one edge of $\tau$.

The edge set $\tau$ is a disjoint union of paths inside $\sigma$, and the vertices with
$\deg_\tau(v)=1$ are the endpoints of these paths. Since only $v_0$ and $v_m$ are not internal in
$\sigma$, two or more paths would have at least two endpoints internal in $\sigma$. So
$|\Delta_\tau|\le1$ forces $\tau$ to be a single path of which at least one endpoint is an endpoint of $\sigma$. Since $\tau\ne\sigma$, the other endpoint is internal in $\sigma$, so the first term equals $1$ and the second term must be $0$, that is $\out_\tau(h)\ne1$ for both fertile hubs. Thus $\tau=P_i$ with $i<m$ or $\tau=S_i$ with $i>1$. The unique element of $\Delta_\tau$ is the parent edge of the internal
endpoint, whose only edge of $\tau$ is $e_i$, so it lies in $C_{e_j}$ if and only if $i=j$.

Conversely, for $\tau=P_j$ ($j<m$) or $S_j$ ($j>1$) with $\out_\tau(1)\ne1\ne\out_\tau(2)$ the
same count gives $\exc_\tau=\mathbf e_j$, so $b_\tau=b_\sigma|_\tau+\mathbf e_j$ and the
offset of $b_\sigma+\mathbf e_j$ at $\tau$ is $0$.
\end{proof}

\begin{example}\label{ex:htwo-dom}
In $G_2$ exactly one proper stratum can be fired from $b_\sigma+\mathbf e_2$, the suffix
$S_2=\{e_2,\dots,e_7\}$, which has $\out_{S_2}(1)=2$ and $\out_{S_2}(2)=3$. The
prefix $P_2=\{e_1,e_2\}$ cannot be fired from it. Indeed $\out_{P_2}(1)=1$ by the edge $e_1$ and $\out_{P_2}(2)=1$ by the
edge $e_2$, so $\exc_{P_2}=\mathbf e_1+2\mathbf e_2$. No proper stratum at all can be fired from $b_\sigma+\mathbf e_1$:
$\out_{P_1}(1)=1$, and $S_1=\sigma$ is not proper. Likewise none can be fired from $b_\sigma+\mathbf e_4$,
since $\out_{P_4}(1)=1$ by the edge $e_1$ and $\out_{S_4}(2)=1$ by the edge $e_7$.
\end{example}

Lemma~\ref{lem:htwo-dom} leaves only prefixes and suffixes to check. Call the index $j$ \emph{spared} if the offset $\mathbf e_j$ is spared, that is, no live $\tau\subsetneq\sigma$ can be fired from $b_\sigma+\mathbf e_j$, and \emph{killed} otherwise. We say that the linear $x_j$ is spared if the index $j$ is.

\begin{lemma}[Spared linears]\label{lem:htwo-cert}
The index $j$ is killed if and only if at least one of the following holds:
\begin{enumerate}[label=(\alph*)]
\item $j<m$ and the anchors of $v_0$ and $v_j$ coincide,
\item $j>1$ and the anchors of $v_{j-1}$ and $v_m$ coincide.
\end{enumerate}
Consequently, if $L$ is a set of spared indices with $|L|=h_1(K_\sigma)$, then $\{0\}\cup\{\mathbf e_j:j\in L\}$ is an unconditional fibre at $\sigma$. This is the multicomplex spanned by the variables $x_j$ with $j\in L$.
\end{lemma}
\begin{proof}
From Lemma~\ref{lem:htwo-dom}, the only $\tau \subsetneq \sigma$ that can be fired from $b_\sigma+\mathbf e_j$ are $P_j$ and $S_j$ with $\out_\tau(1)\ne1\ne\out_\tau(2)$, and $j$ is killed if and only if one of them is live. We show that for $\tau=P_j$ with $j<m$ or $\tau=S_j$ with $j>1$, the condition $\out_\tau(1)\ne1\ne\out_\tau(2)$ together with liveness is equivalent to the two anchors of $\tau$ coinciding. When the two anchors coincide, the two ends of $\tau$ lie in the same fertile branch or in none, so $\out_\tau(1)$ and $\out_\tau(2)$ are even by Lemma~\ref{lem:parity} and the condition on $\out$ holds. Under the condition on $\out$, Lemma~\ref{lem:excess}(b) says that such a $\tau$ is live if and only if its two anchors coincide.

For the second claim, $\{0\}\cup\{\mathbf e_j:j\in L\}$ is a pure multicomplex with degree sequence $(1,|L|)=h(K_\sigma)$. Every proper $\tau$ that can be fired from $b_\sigma+0$ is dead and has offset $0$ by Lemma~\ref{lem:mono}(d). Every proper $\tau$ that can be fired from $b_\sigma+\mathbf e_j$ has offset $0$ by Lemma~\ref{lem:htwo-dom}, and is dead because $j$ is spared. So Lemma~\ref{lem:baseline} applies.
\end{proof}

\begin{example}\label{ex:htwo-cert}
In $G_2$ the index $3$ is killed: $S_3$ goes from $b$ to $b'$, both children of $2$, so its
anchors coincide, and $\out_{S_3}(1)=\out_{S_3}(2)=2$. The index $5$ is killed by $P_5$, which goes
from $a$ to $a'$, both children of $1$. The index $2$ is spared although $S_2$ can be fired from
$b_\sigma+\mathbf e_2$: that path has anchors $o$ (at the barren hub $p$) and $2$, so it is dead,
and the offset there is $0$. The spared set is $\{1,2,4,6,7\}$.
\end{example}

\begin{lemma}[Counting spared linears]\label{lem:htwo-count}
At least $h_1(K_\sigma)$ indices are spared.
\end{lemma}
\begin{proof}
We treat the cases (a) and (b) of Proposition~\ref{prop:startri} in turn.

Suppose first that one of the two anchors is $o$. We read the path starting from the barren hub end, so $\kappa(v_0)=\word A$ and $\kappa(v_m)=h$, where the other endpoint $v_m$ is a child
of the hub $h$ or the hub $h$ itself. Without loss of generality $h=2$, so $1$ is the other hub. In this case an edge $e_k=v_{k-1}v_k$ is killed if and only if $\kappa(v_k)=\word A$
or $\kappa(v_{k-1})=2$, so the spared edges are exactly those coloured $11$, $\word A1$, $12$, $\word A2$.
Indeed, by Lemma~\ref{lem:htwo-cert} the anchors of $P_k$ coincide if and only if $\kappa(v_k)=\word A$ and those of $S_k$ coincide if and only if $\kappa(v_{k-1})=2$.

Call a maximal set of consecutive vertices coloured $1$ a $1$-block, and write a $1$-block of
$k$ vertices together with its two neighbours on the path as $x\,1^k\,y$, where $x,y\in\{\word A,2\}$
since the two ends of the path are coloured $\word A$ and $2$. Its $k-1$ internal edges are
coloured $11$ and spared, its entering edge is spared if and only if $x=\word A$ and its leaving edge if and only if
$y=2$. So a block $2\,1^k\,\word A$ (a \emph{deficit block}) accounts for $k-1$ spared edges,
a block $\word A\,1^k\,2$ (a \emph{surplus block}) for $k+1$, and a block $\word A\,1^k\,\word A$
or $2\,1^k\,2$ for $k$, and the spared edges touching no $1$-block are those coloured
$\word A2$. Deleting the vertices coloured $1$ leaves a word in $\word A,2$ from $\word A$ to
$2$, which has one more factor $\word A2$ than factors $2\word A$, and its factors $\word A2$
are the surplus blocks and the spared $\word A2$ edges while its factors $2\word A$ are the
deficit blocks and the killed $2\word A$ edges. Hence the number of spared edges is at least
$\bigl(\sum_{\text{blocks}}k\bigr)+1$, and $\sum_{\text{blocks}}k=n_1$, because the vertices
coloured $1$ are exactly the internal children of the hub $1$, while $n_1=h_1(K_\sigma)$ by
Proposition~\ref{prop:startri}(a) with the two hubs interchanged. The count can also be realised by
matching the vertices coloured $1$ to distinct spared edges. Match each vertex of a block to
the edge leaving it, except the last vertex of a block $x\,1^k\,\word A$, whose leaving edge is
killed. That vertex is matched to the entering edge of the block if $x=\word A$, and to the spared
edge of the next factor $\word A2$ if $x=2$. The first factor $\word A2$ is then left over
(Figure~\ref{fig:htwocount}, bottom).

Suppose now that the anchors are $1$ and $2$. We read the path starting from the endpoint coloured $1$, so $\kappa(v_0)=1$ and $\kappa(v_m)=2$. In this case an edge $e_k=v_{k-1}v_k$ is killed if and only if $\kappa(v_k)=1$ or $\kappa(v_{k-1})=2$, so the spared edges are exactly those that do not start with $2$ and do not end with $1$.

Consider now the $\word A$-blocks of the path, the maximal sets of consecutive vertices coloured $\word A$, written $x\,\word A^k\,y$ with $x,y\in\{1,2\}$. This time a deficit block is $2\,\word A^k\,1$ and a surplus block is $1\,\word A^k\,2$. Deleting the vertices coloured $\word A$ leaves a word in $1,2$ that starts with $1$ and ends with $2$, and the rest of the argument goes exactly as in the previous case, with the factors $12$ and $21$ in place of $\word A2$ and $2\word A$. The number of spared edges is at least $\bigl(\sum_{\text{blocks}}k\bigr)+1=n_{\word A}+1=h_1(K_\sigma)$ by Proposition~\ref{prop:startri}, since the vertices coloured $\word A$ are the internal barren hubs. This time the left-over factor $12$ is needed, for the additional unit in $h_1(K_\sigma)$ (Figure~\ref{fig:htwocount}, top).
\end{proof}

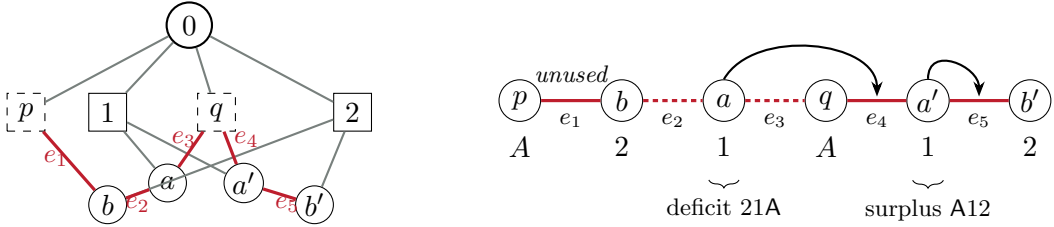
\begin{figure}[h]\centering
\begin{tikzpicture}[x=1cm,y=1cm,
  wv/.style={circle,draw,inner sep=1.4pt,minimum size=5.5mm,font=\small},
  cw/.style={font=\small\sffamily},
  el/.style={font=\scriptsize},
  killed/.style={rd,dash pattern=on 2pt off 1.6pt},
  spared/.style={rd},
  inj/.style={-{Stealth[length=2mm]},thick,shorten >=1.5pt},
  extra/.style={font=\scriptsize\itshape}]
\begin{scope}[xshift=-2.2cm]
  \node[el,anchor=west] at (-0.4,1.35) {type (b): $G_2$, anchors $1$ and $2$, counted letter $\word A$, $h_1(K_\sigma)=1+n_{\word A}=5$};
  \foreach \i/\n/\c in {0/a/1,1/p/A,2/b/2,3/q/A,4/r/A,5/a'/1,6/s/A,7/b'/2}{
    \node[wv] (v\i) at (1.35*\i,0) {$\n$};
    \node[cw,below=1.5pt of v\i] {$\c$};
  }
  \foreach \i/\j/\s in {0/1/spared,1/2/spared,2/3/killed,3/4/spared,4/5/killed,5/6/spared,6/7/spared}{
    \draw[\s] (v\i) -- (v\j);
    \pgfmathtruncatemacro{\k}{\j}
    \coordinate (m\k) at ($(v\i)!0.5!(v\j)$);
    \node[el,below=1pt] at (m\k) {$e_{\k}$};
  }
  \draw[inj] (v1.north) to[out=90,in=90,looseness=1.6] (m2);
  \draw[inj] (v3.north) to[out=90,in=90,looseness=1.6] (m4);
  \draw[inj] (v4.north) to[out=60,in=90,looseness=1.0] (m6);
  \draw[inj] (v6.north) to[out=90,in=90,looseness=1.6] (m7);
  \node[extra,above=3pt] at (m1) {$+1$};
  \draw[decorate,decoration={brace,mirror,amplitude=3pt}] ($(v1.south west)+(0,-0.85)$) -- ($(v1.south east)+(0,-0.85)$)
     node[midway,below=3pt,el] {surplus $1\word A2$};
  \draw[decorate,decoration={brace,mirror,amplitude=3pt}] ($(v3.south west)+(0,-0.85)$) -- ($(v4.south east)+(0,-0.85)$)
     node[midway,below=3pt,el] {deficit $2\word A\word A1$};
  \draw[decorate,decoration={brace,mirror,amplitude=3pt}] ($(v6.south west)+(0,-0.85)$) -- ($(v6.south east)+(0,-0.85)$)
     node[midway,below=3pt,el] {surplus $1\word A2$};
\end{scope}
\begin{scope}[yshift=-5.4cm]
  \node[el,anchor=west] at (-7.0,2.1) {type (a): $G_3$, anchors $o$ and $2$, counted letter $1$, $h_1(K_\sigma)=n_1=2$};
  \begin{scope}[xshift=-6.3cm,yshift=-1.45cm,scale=0.72]
    \node[ctr] (o) at (2.4,3.4) {$0$};
    \node[wh] (hp) at (-0.6,1.8) {$p$}; \node[hub] (h1) at (0.9,1.8) {$1$};
    \node[wh] (hq) at (2.9,1.8) {$q$}; \node[hub] (h2) at (5.4,1.8) {$2$};
    \node[ch] (cb) at (0.9,0.1) {$b$}; \node[ch] (ca) at (2.0,0.5) {$a$};
    \node[ch] (cA) at (3.4,0.5) {$a'$}; \node[ch] (cB) at (4.7,0.1) {$b'$};
    \foreach \x in {h1,h2,hp,hq} \draw[tr] (o)--(\x);
    \draw[tr] (h1)--(ca); \draw[tr] (h1)--(cA);
    \draw[tr] (h2)--(cB); \draw[tr] (h2)--(cb);
    \draw[rd] (hp)--(cb); \draw[rd] (cb)--(ca); \draw[rd] (ca)--(hq); \draw[rd] (hq)--(cA); \draw[rd] (cA)--(cB);
    \node[rlab] at (-0.05,0.95) {$e_1$}; \node[rlab] at (1.45,0.1) {$e_2$}; \node[rlab] at (2.3,1.3) {$e_3$};
    \node[rlab] at (3.45,1.3) {$e_4$}; \node[rlab] at (4.2,0.1) {$e_5$};
  \end{scope}
  \foreach \i/\n/\c in {0/p/A,1/b/2,2/a/1,3/q/A,4/a'/1,5/b'/2}{
    \node[wv] (w\i) at (1.35*\i-0.2,0) {$\n$};
    \node[cw,below=1.5pt of w\i] {$\c$};
  }
  \foreach \i/\j/\s in {0/1/spared,1/2/killed,2/3/killed,3/4/spared,4/5/spared}{
    \draw[\s] (w\i) -- (w\j);
    \pgfmathtruncatemacro{\k}{\j}
    \coordinate (n\k) at ($(w\i)!0.5!(w\j)$);
    \node[el,below=1pt] at (n\k) {$e_{\k}$};
  }
  \draw[inj] (w2.north) to[out=60,in=90,looseness=1.0] (n4);
  \draw[inj] (w4.north) to[out=90,in=90,looseness=1.6] (n5);
  \node[extra,above=3pt] at (n1) {unused};
  \draw[decorate,decoration={brace,mirror,amplitude=3pt}] ($(w2.south west)+(0,-0.85)$) -- ($(w2.south east)+(0,-0.85)$)
     node[midway,below=3pt,el] {deficit $21\word A$};
  \draw[decorate,decoration={brace,mirror,amplitude=3pt}] ($(w4.south west)+(0,-0.85)$) -- ($(w4.south east)+(0,-0.85)$)
     node[midway,below=3pt,el] {surplus $\word A12$};
\end{scope}
\end{tikzpicture}
\caption{Lemma~\ref{lem:htwo-count} on the two types of Proposition~\ref{prop:startri}: the type~(b) root $G_2$ of Figure~\ref{fig:htwo} (top) and the type~(a) root $G_3$ of Example~\ref{ex:htwo-count-a} (bottom).
Killed edges are dashed, and the arrows match the vertices of the count to distinct spared edges. The left-over spared edge accounts for the additional unit in $h_1(K_\sigma)=1+n_{\word A}$ in $G_2$ and is not needed in $G_3$, where $h_1(K_\sigma)=n_1$.}\label{fig:htwocount}
\end{figure}

\begin{example}[Type (a)]\label{ex:htwo-count-a}
Let $G_3$ be the root with centre $0$, fertile hubs $1$ (children $a,a'$) and $2$ (children
$b,b'$), barren hubs $p,q$, and red path $\sigma=\{e_1,\dots,e_5\}$ through $p\,b\,a\,q\,a'\,b'$
(Figure~\ref{fig:htwocount}, bottom). Its colour word is $\word A\,2\,1\,\word A\,1\,2$ and its
anchors are $o$ and $2$. As $\out_\sigma(1)=4$ and $\out_\sigma(2)=3$, the stratum $\sigma$ is a
dead rank-one stratum of type~(a) with $h(K_\sigma)=(1,n_1)=(1,2)$ and $b_\sigma=(2,1,1,1,2)$.
The index $2$ is killed by the suffix $S_2$ from $b$ to $b'$, with anchors $2$ and $2$, and the
index $3$ by the prefix $P_3$ from $p$ to $q$, with anchors $o$ and $o$. The spared indices are
$1,4,5$, three of them for $h_1(K_\sigma)=2$, and each of $\{0,\mathbf e_1,\mathbf e_4\}$,
$\{0,\mathbf e_1,\mathbf e_5\}$, $\{0,\mathbf e_4,\mathbf e_5\}$ is a fibre by
Lemma~\ref{lem:htwo-cert}.
\end{example}

\begin{example}[Type (b)]\label{ex:htwo-count}
In $G_2$ the $\word A$-blocks of the colour word $1\,\word A\,2\,\word A\word A\,1\,\word A\,2$
are $\{p\}$ and $\{s\}$, both surplus blocks, and the deficit block $\{q,r\}$. So the count of
the proof is $\bigl(\sum_{\text{blocks}}k\bigr)+1=4+1=5=h_1(K_\sigma)$
(Figure~\ref{fig:htwocount}, top). The spared indices are exactly $1,2,4,6,7$
(Example~\ref{ex:htwo-cert}), so
$F(\sigma)=\{0,\mathbf e_1,\mathbf e_2,\mathbf e_4,\mathbf e_6,\mathbf e_7\}$ is the only fibre
at this stratum of the form in Lemma~\ref{lem:htwo-cert}.
\end{example}

The spared set is read off the colour word alone. In particular, red edges between two
children of the same hub, red edges at a hub, and paths with a fertile hub as an endpoint need
no separate treatment.

\section{Applications}\label{sec:cases}

We now verify that a fibre system exists in each of the classes announced in the introduction. The
numerical statements of this section and of Section~\ref{sec:outlook} are checked in the
companion notebook \cite{Notebook}, which enumerates the relaxed coparking functions directly
from Definitions~\ref{def:fibre} and~\ref{def:relaxed}.

\subsection{Coned, biconed and triconed graphs}

A set of vertices \emph{dominates} a graph if every vertex outside it is adjacent to one of
its members. A graph is \emph{coned} \cite{Kook} if it is dominated by a vertex $0$, \emph{biconed}
\cite{Biconed} if it is dominated by an edge $01$ with no parallel edges at $0$ or $1$, and
\emph{triconed} \cite{Triconed} if it is dominated by a path $1\,0\,2$ with no parallel edges
at $0$, $1$ or $2$. We assume there are no loops or coloops, which can be removed without changing $h$ or leaving
the class \cite[Lemma 3.3]{Triconed}. The canonical tree of \cite{Triconed} attaches every
vertex to $0$ if possible, else to $1$, else to $2$. With its canonical tree, a triconed graph is
therefore a radius-two root with fertile hubs among $1,2$, hence $H\le2$. A biconed graph with its
canonical tree is one with $H\le1$, and a coned graph one with $H=0$, the tree being the star at
$0$. Section~\ref{sec:radius2} assumes a simple graph, but the parallel edges between vertices
other than $0,1,2$ that \cite{Triconed} allows cause no difficulty: a stratum containing two
of them is not a forest, so it is live. Now let a stratum contain at most one edge of each
parallel class, and keep one edge of each parallel class, the edge of the stratum where there is one. In the resulting simple graph the stratum and its subsets have the same unique unions, baselines and dead nodes as before. Every parallel edge is red, since the tree edges all meet $0$, $1$ or $2$, so the tree is unchanged and the
results of Section~\ref{sec:radius2} apply to it.

\begin{theorem}[Rank one at $H\le2$]\label{thm:rankone}
In a radius-two root with at most two fertile hubs, every dead node has rank at most one and
every dead stratum admits an unconditional fibre.
\end{theorem}
\begin{proof}
The rank bound is Corollary~\ref{cor:rankH}. A dead stratum of rank zero admits the fibre $\{0\}$
by Lemma~\ref{lem:zerofibre}. A dead stratum of rank one has $\phi(\sigma)\ge c(\sigma)+1\ge2$ by Theorem~\ref{thm:rank}. So the root has exactly two fertile hubs and $\phi(\sigma)=2$, and Section~\ref{sec:htwo} applies. Lemma~\ref{lem:htwo-count} provides
$h_1(K_\sigma)$ spared indices, and by Lemma~\ref{lem:htwo-cert} the corresponding linears
together with $0$ form an unconditional fibre.
\end{proof}

\begin{theorem}\label{thm:triconed}
Let $G$ be a triconed graph with its canonical tree $B_0$. Then every choice of fibres along
inclusion meets no obstructed stratum, and for every fibre system $F$ on $(M(G),B_0)$ the set
$\PP^*(F)$ is a pure multicomplex with degree sequence $h(M(G))$. In particular Stanley's conjecture
holds for triconed, biconed and coned graphs. For a coned graph
every stratum is live, so the fundamental triangles form a cycle system and the relaxed
coparking functions are the coparking functions of \cite{CycleSystems}.
\end{theorem}
\begin{proof}
By the discussion above, a triconed, biconed or coned graph with its canonical tree is a
radius-two root with $H\le2$. By Theorem~\ref{thm:rankone} every dead stratum admits an unconditional fibre, so every choice of fibres along inclusion is a fibre
system with no obstructed stratum, and Theorem~\ref{thm:main} gives the pure multicomplex with
degree sequence $h(M(G))$.

For a coned graph every vertex other than $0$ is a barren hub, so by Lemma~\ref{lem:mult} the
spine $0h$ lies in $U_\sigma$ exactly when $\deg_\sigma(h)=1$. If $\sigma$ contains a cycle it
is live. Otherwise two leaves $h,h'$ of one component of the forest $\sigma$ have their spines
in $U_\sigma$, and the path from $h$ to $h'$ in $\sigma$ closed by $0h$ and $0h'$ is a cycle in
$U_\sigma$. So every stratum is live, and the rest is Corollary~\ref{cor:specialise}(a).
\end{proof}

This recovers the theorem of David, Lai, Oh and Wu \cite{Triconed} by a different proof. That
paper constructs an explicit bijection from spanning trees to weighted forests of the reduced
graph. Here the triconed case is a corollary of a general theorem, and the only input specific
to the class is the rank bound of Corollary~\ref{cor:rankH} together with the colour-word
analysis of Section~\ref{sec:htwo}. For biconed graphs the model has
only rank-zero dead nodes (Example~\ref{ex:biconed-gcs}), and Corollary~\ref{cor:specialise}(b)
gives the multicomplex directly: it is the set of vectors that satisfy the coparking condition on
live strata and, whenever a dead stratum can be fired from them,
meet its coparking baseline exactly. For coned graphs the cycle system of the theorem is that of
\cite[Section 3.1, Example (3)]{CycleSystems}, and Theorem~\ref{thm:main} reproves Kook's
theorem \cite{Kook} through it.

\subsection{The Wagner graph}\label{sec:wagner}

The Wagner graph $V_8$ is the $8$-cycle $0\,1\,2\,3\,4\,5\,6\,7$ with the four long diagonals
$v\,(v+4)$ (Figure~\ref{fig:wagner}). It is triconed, dominated by the path $1\,0\,7$.

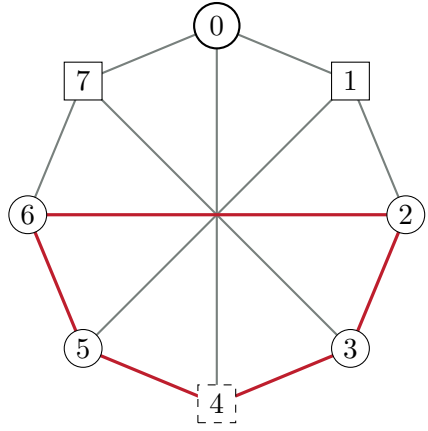
\begin{figure}[h]\centering
\begin{tikzpicture}[scale=1.25]
\foreach \i in {0,...,7} \coordinate (v\i) at ({90-45*\i}:2);
\draw[tr] (v0)--(v1); \draw[tr] (v0)--(v4);
\draw[tr] (v0)--(v7); \draw[tr] (v1)--(v2);
\draw[tr] (v1)--(v5); \draw[tr] (v7)--(v6);
\draw[tr] (v7)--(v3);
\draw[rd] (v2)--(v3); \draw[rd] (v3)--(v4); \draw[rd] (v4)--(v5); \draw[rd] (v5)--(v6); \draw[rd] (v2)--(v6);
\node[ctr] at (v0) {$0$}; \node[hub] at (v1) {$1$}; \node[hub] at (v7) {$7$}; \node[wh] at (v4) {$4$};
\node[ch] at (v2) {$2$}; \node[ch] at (v3) {$3$}; \node[ch] at (v5) {$5$}; \node[ch] at (v6) {$6$};
\end{tikzpicture}
\caption{The Wagner graph $V_8$ with the canonical tree of the dominating path $1\,0\,7$,
drawn as in Figure~\ref{fig:running}: tree edges grey, red edges red. Hub $4$ (dashed) is
barren, so $H=2$.}\label{fig:wagner}
\end{figure}

\begin{proposition}\label{prop:wagner}
$V_8$ admits no generalized cycle system consisting of circuits (Proposition~\ref{prop:nogcs}),
but it carries a fibre system with the canonical tree of the path $1\,0\,7$, so
Stanley's conjecture holds for $V_8$ with an explicit pure multicomplex of relaxed coparking
functions.
\end{proposition}
\begin{proof}
The first claim is Perkinson's computation, Proposition~\ref{prop:nogcs}. For the second, the vertices $2$
and $5$ are adjacent to $1$, the vertices $3$ and $6$ to $7$, and $4$ to $0$. So $V_8$ is
triconed with dominating path $1\,0\,7$ and Theorem~\ref{thm:triconed} applies.
\end{proof}

\subsection{The Petersen graph}\label{sec:petersen}

Label the Petersen graph as in Figure~\ref{fig:petersen}, with outer cycle $0\,1\,2\,3\,4$,
inner vertices $5,\dots,9$ with $v$ adjacent to $v+5$, and inner edges $57,79,96,68,85$. It has
diameter two and every vertex has exactly one common neighbour with every non-adjacent vertex, so from any centre the
breadth-first tree is unique. By vertex-transitivity we take the centre $0$. The hubs are
$1,4,5$, each with two children: $1$ has $2$ and $6$, $4$ has $3$ and $9$, and $5$ has $7$ and $8$. The tree is
$B_0=\{01,04,05,12,16,43,49,57,58\}$ and the six red edges form the hexagon
\[
  2-3-8-6-9-7-2 ,
\]
whose vertices are coloured by their parents $\word A,\word B,\word C,\word A,\word B,\word C$ with $\word A=1$, $\word B=4$, $\word C=5$: every hub
occurs exactly twice and no two consecutive vertices share a hub. Since the root has no barren hubs, we depart here from the convention of Section~\ref{sec:htwo} and write the three fertile hubs as $\word A,\word B,\word C$. Since all six red edges join children, no stratum has a hub as a path vertex.

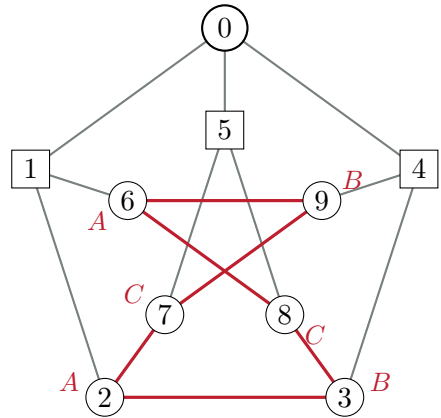
\begin{figure}[h]\centering
\begin{tikzpicture}[scale=1.35]
\foreach \i in {0,...,4} { \coordinate (O\i) at ({90+72*\i}:2); \coordinate (I\i) at ({90+72*\i}:1); }
\foreach \i/\j in {0/1,1/2,2/3,3/4,4/0} \draw[grey!60] (O\i)--(O\j);
\foreach \i in {0,...,4} \draw[grey!60] (O\i)--(I\i);
\foreach \i/\j in {0/2,2/4,4/1,1/3,3/0} \draw[grey!60] (I\i)--(I\j);
\draw[tr] (O0)--(O1); \draw[tr] (O0)--(O4); \draw[tr] (O0)--(I0);
\draw[tr] (O1)--(O2); \draw[tr] (O1)--(I1); \draw[tr] (O4)--(O3); \draw[tr] (O4)--(I4);
\draw[tr] (I0)--(I2); \draw[tr] (I0)--(I3);
\draw[rd] (O2)--(O3); \draw[rd] (O3)--(I3); \draw[rd] (I3)--(I1); \draw[rd] (I1)--(I4); \draw[rd] (I4)--(I2); \draw[rd] (I2)--(O2);
\node[ctr] at (O0) {$0$};
\node[hub] at (O1) {$1$}; \node[hub] at (O4) {$4$}; \node[hub] at (I0) {$5$};
\node[ch] at (O2) {$2$}; \node[ch] at (O3) {$3$}; \node[ch] at (I1) {$6$}; \node[ch] at (I2) {$7$}; \node[ch] at (I3) {$8$}; \node[ch] at (I4) {$9$};
\node[lab,redge] at ($(O2)+(-0.35,0.15)$) {$A$}; \node[lab,redge] at ($(O3)+(0.35,0.15)$) {$B$};
\node[lab,redge] at ($(I3)+(0.3,-0.2)$) {$C$}; \node[lab,redge] at ($(I1)+(-0.3,-0.2)$) {$A$};
\node[lab,redge] at ($(I4)+(0.3,0.2)$) {$B$}; \node[lab,redge] at ($(I2)+(-0.3,0.2)$) {$C$};
\end{tikzpicture}
\caption{The Petersen root from the centre $0$. The tree $B_0$ is the set of non-red edges, and the six red edges
form the red hexagon $2\,3\,8\,6\,9\,7$. Each outer vertex is coloured by its parent
($\word A=1$, $\word B=4$, $\word C=5$).}\label{fig:petersen}
\end{figure}

\begin{lemma}[Strata of the Petersen root]\label{lem:petersen-strata}
Of the $63$ nonempty sets of red edges, exactly $24$ are dead:
\begin{center}
\begin{tabular}{@{}lccc@{}}
\toprule
selected red edges & shape in the hexagon & number & $r(K_\sigma)$ \\
\midrule
three & a path of two edges and a disjoint edge & $12$ & $0$ \\
four  & a path of four edges & $6$ & $2$ \\
five  & a path of five edges & $6$ & $2$ \\
\bottomrule
\end{tabular}
\end{center}
\end{lemma}
\begin{proof}
The hexagon itself is live, and every other stratum $\sigma$ is a forest whose components
$P_1,\dots,P_c$ are paths in the hexagon. Every outer vertex is a child, so by
Corollary~\ref{cor:split} the edges of $A(\sigma)$ are the two legs from the endpoints of each
$P_j$ to their parent hubs and the spines $s_h$ of the hubs with $\out_\sigma(h)=1$. Each $P_j$
has degree two in $A(\sigma)$. Contracting one of the two edges at each $P_j$ does not change whether the graph is a forest, so $A(\sigma)$ is a forest if and only if
the multigraph $A'(\sigma)$ on the vertices $o,\word A,\word B,\word C$ is a forest, where $A'(\sigma)$ has an edge
joining the colours of the two endpoints of each $P_j$ and an edge $oh$ for each hub with
$\out_\sigma(h)=1$. By Theorem~\ref{thm:indep} this decides whether $\sigma$ is dead. Around the hexagon the colours read $\word{ABCABC}$, so opposite
vertices share a colour and vertices at distance one or two do not.

When $|\sigma|\le2$, $\sigma$ is live, as is every stratum with at most two elements in a
binary matroid (the remark following Lemma~\ref{lem:mono}). When $|\sigma|=3$, a path of three
edges has opposite endpoints and gives a loop of $A'(\sigma)$, and three disjoint edges give the
triangle $\word{ABC}$. A path of two edges with a disjoint edge (twelve strata) gives two edges
sharing one hub together with one spine, the spine of the one hub with $\out_\sigma(h)=1$, and
these three edges form a path. So these twelve strata are dead, and $r(K_\sigma)=I(\sigma)+\phi(\sigma)-|\sigma|=1+2-3=0$
by Theorem~\ref{thm:rank}. When $|\sigma|=4$, two paths of two edges (three strata) have the
same pair of endpoint colours and give two parallel edges, and a path of three edges with a
disjoint edge (six strata) gives a loop. A path of four edges (six strata) has endpoints at
distance two, hence of different colours, and $\out_\sigma(h)\ge2$ for every hub, so
$A'(\sigma)$ is a single edge. These six strata are dead with $r(K_\sigma)=3+3-4=2$. When
$|\sigma|=5$, a path of five edges (six strata) has adjacent endpoints, again of different
colours and with $\out_\sigma(h)\ge2$ for every hub, so it is dead with
$r(K_\sigma)=4+3-5=2$.
\end{proof}

The twelve rank-zero strata have the fibre $\{0\}$ by Lemma~\ref{lem:zerofibre}. The
twelve rank-two strata are, up to the symmetries of the hexagon and renaming of the hubs, the paths
$\word{ABCAB}$ (four edges) and $\word{ABCABC}$ (five edges). Both have their anchors at
two fertile hubs, so by Proposition~\ref{prop:startri} their dead nodes are triangle nodes, with
parallel classes $(2,1,1)$ and $h(K_\sigma)=(1,2,2)$ for $\word{ABCAB}$, and parallel classes $(2,1,2)$ and $h(K_\sigma)=(1,3,4)$ for $\word{ABCABC}$. Number the red edges of the path
$e_1,\dots,e_k$ in path order, with vertices $v_0,\dots,v_k$ and $e_i=v_{i-1}v_i$ as in Section~\ref{sec:htwo}, so that the variable $x_i$ in the tables below belongs to $e_i$.

\begin{lemma}[Unconditional fibres of the Petersen root]\label{lem:petersen-cert}
The following are unconditional fibres:
\begin{center}
\begin{tabular}{@{}llll@{}}
\toprule
path & coparking baseline $b_\sigma$ & $h(K_\sigma)$ & fibre $F(\sigma)$ \\
\midrule
$\word{ABCAB}$  & $(2,1,1,2)$   & $(1,2,2)$ & $\{1,\ x_1,\ x_4,\ x_1^2,\ x_4^2\}$ \\
$\word{ABCABC}$ & $(2,1,1,1,2)$ & $(1,3,4)$ & $\{1,\ x_1,\ x_2,\ x_5,\ x_1^2,\ x_1x_2,\ x_1x_5,\ x_5^2\}$ \\
\bottomrule
\end{tabular}
\end{center}
\end{lemma}
\begin{proof}
By Section~\ref{sec:criteria}, only proper strata $\tau$ with $|\Delta_\tau|\le2$ can kill an
offset of degree at most two, and by Section~\ref{sec:killers} such a $\tau$ has at most two
internal leaves, so it is a prefix, a suffix, a subpath avoiding both ends, or the union of a
prefix and a suffix. Those with $\out_\tau(h)=1$ for some hub
$h$ have $|\Delta_\tau|\ge3$ (a direct count on the list) and are discarded. The rest are listed below, with the threshold $q_\tau$ read
off the colours and whether $\tau$ is dead or live from Theorem~\ref{thm:indep}:
\begin{center}
\begin{tabular}{@{}lll@{}}
\toprule
path & proper strata $\tau$ & $q_\tau$, dead or live \\
\midrule
$\word{ABCAB}$ & $\{e_1,e_2,e_3\}$ & $x_3$, live (return to $\word A$) \\
               & $\{e_2,e_3,e_4\}$ & $x_2$, live (return to $\word B$) \\
               & $\{e_1,e_4\}$   & $x_1x_4$, live (double bridge $\word A$--$\word B$) \\
\midrule
$\word{ABCABC}$ & $\{e_1,e_2,e_3\}$, $\{e_3,e_4,e_5\}$ & $x_3$, live \\
                & $\{e_2,e_3,e_4\}$, $\{e_1,e_2,e_4,e_5\}$ & $x_2x_4$, live \\
                & $\{e_1,e_2,e_3,e_5\}$ & $x_3x_5$, live \\
                & $\{e_1,e_3,e_4,e_5\}$ & $x_1x_3$, live \\
                & $\{e_1,e_2,e_3,e_4\}$ & $x_4$, dead \\
                & $\{e_2,e_3,e_4,e_5\}$ & $x_2$, dead \\
\bottomrule
\end{tabular}
\end{center}
A proper $\tau$ can be fired from $b_\sigma+m$ if and only if $q_\tau$ divides $x^m$. No monomial of either
multicomplex is divisible by a live threshold of the table. The dead threshold $x_4$ divides no
monomial of the second multicomplex, and $x_2$ divides exactly $x_2$ and $x_1x_2$, whose
restriction to $\tau=\{e_2,e_3,e_4,e_5\}$ is $x_2=q_\tau$, so the offset at $\tau$ is $0$. Both
multicomplexes are closed under division and pure, with degree sequences $(1,2,2)$ and
$(1,3,4)=h(K_\sigma)$, so Lemma~\ref{lem:baseline} applies.
\end{proof}

\begin{theorem}\label{thm:petersen}
The Petersen graph admits no generalized cycle system consisting of circuits
(Proposition~\ref{prop:nogcs}), but it carries a fibre system with its unique breadth-first tree from any vertex: the fibres $\{0\}$ at the twelve rank-zero strata and the unconditional fibres of Lemma~\ref{lem:petersen-cert} at the twelve rank-two strata form a fibre system. Hence Stanley's conjecture holds for the Petersen graph, with an explicit pure multicomplex of relaxed coparking functions.
\end{theorem}
\begin{proof}
The first claim is Proposition~\ref{prop:nogcs}. By
Lemma~\ref{lem:petersen-strata} the dead strata are twelve of rank zero, which have the fibre
$\{0\}$ by Lemma~\ref{lem:zerofibre}, and twelve paths of four or five edges, each of which
reads $\word{ABCAB}$ or $\word{ABCABC}$ once a direction is chosen and the hubs are named in the
order in which they occur along the path. Lemma~\ref{lem:petersen-cert} gives an unconditional fibre at each
of them, so choosing the fibres in order of $|\sigma|$ gives a fibre
system, and Theorem~\ref{thm:main} applies.
\end{proof}

\subsection{Matroids of corank two}\label{sec:coranktwo}

Everything so far is graphic. Since Theorem~\ref{thm:main} holds for any matroid with a fixed basis,
we close the section with three classes of matroids not contained in the graphic matroids. The
first is dictated by corank rather than rank.

\begin{theorem}\label{thm:coranktwo}
Let $M$ be a matroid of corank two and let $B_0$ be any basis. Then the red pair $R$ is the only
stratum that can be dead, and when it is dead every pure multicomplex in $\N^R$ with degree
sequence $h(K_R)$ is an unconditional fibre at $R$, and such a multicomplex exists. Consequently every basis of a matroid of
corank two carries a fibre system, and Stanley's conjecture holds for these matroids.
\end{theorem}
\begin{proof}
We first remove loops and coloops. A loop $a$ of $M$ is red with $C_a=\{a\}$, so it lies in $U_\tau$ for every stratum $\tau\ni a$ and makes every such stratum live, and the only stratum not containing it is a singleton, which is live. So if $M$ has a loop every stratum is live and Theorem~\ref{thm:main} applies with nothing to choose. A coloop lies in no $C_i$, so no stratum is affected by deleting it, and we delete it. We therefore assume $M$ has neither. The dual $M^*$ then
has rank two and no loops or coloops, so it is a set of $k\ge2$ parallel classes
$P_1,\dots,P_k$, and a basis of $M$ is the complement of a pair from two different classes. Write
$B_0=E\setminus\{a,b\}$ with $a\in P_a$, $b\in P_b$, $P_a\ne P_b$. We argue by induction on $k$.

For a tree edge $e$, the set $(B_0\cup a)\setminus e$ is a basis if and only if $\{e,b\}$ is a basis of
$M^*$ if and only if $e\notin P_b$, so $C_a=E\setminus P_b$ and likewise $C_b=E\setminus P_a$. Hence
$U_R=C_a\,\triangle\,C_b=P_a\cup P_b$, and $A_R=E$. A set is independent in $M$ if and only if its
complement spans $M^*$, so $R$ is dead if and only if the classes other than $P_a,P_b$ span the rank-two
matroid $M^*$. A set spans a rank-two matroid without loops if and only if it meets at least
two parallel classes, so this happens if and only if there are at least two classes other than
$P_a,P_b$, that is, if and only if $k\ge4$. For $k\le3$ every stratum is live, there is nothing to choose, and
Theorem~\ref{thm:main} gives the claim.

Let $k\ge4$. Then $K_R=M/(P_a\cup P_b)$, whose dual $M^*\setminus(P_a\cup P_b)$ has rank two,
no loops and the $k-2$ remaining classes, so $K_R$ has corank two and no coloops. The only proper
strata are the singletons, with $b_{\{a\}}(a)=|C_a|$ and $b_R(a)=|C_a\cap U_R|=|P_a|$. So
$\{a\}$ can be fired from $b_R+m$ only if $m_a\ge|C_a|-|P_a|=|C_a\cap C_b|$. This is the number of elements of
$K_R$ and exceeds $r(K_R)\ge|m|$. So no proper stratum can be fired from $b_R+m$ for an offset $m$ of degree at most $r(K_R)$,
and by Lemma~\ref{lem:baseline} every pure multicomplex in $\N^R$ with degree sequence
$h(K_R)$ is an unconditional fibre.

By induction, for any basis of $K_R$ and any fibre system on it, the relaxed coparking functions of $K_R$
form a pure multicomplex in two variables with degree sequence $h(K_R)$. Renaming the variables
$x_a,x_b$ gives a fibre at $R$, and Theorem~\ref{thm:main} gives the claim.
\end{proof}

This recovers the theorem of De Loera, Kemper and Klee \cite{DLKK} for corank two. The
induction shows why the class suits the model: the dead node of a corank-two matroid is again
of corank two, with two fewer parallel classes in the dual, and every fibre along the way is
unconditional.

\subsection{Matroids of rank at most four}\label{sec:rankfour}

In rank at most four two kinds of dead strata have a fibre to find. The dead pairs have dead
nodes of corank two and receive their fibres from Theorem~\ref{thm:coranktwo}. The dead triples
whose dead node has rank one have fibres made of linears, as at the rank-one nodes of
Section~\ref{sec:htwo}, but without any colour word to read. For a red element $i$ we write
$T_i=C_i\setminus i$ for the tree part of its circuit, a subset of $B_0$.

\begin{theorem}\label{thm:rankfour}
Let $M$ be a matroid of rank at most four and $B_0$ any basis. Then every dead pair admits an
unconditional fibre and every dead triple admits a fibre. Consequently every basis of a matroid
of rank at most four carries a fibre system, and Stanley's conjecture holds for these matroids.
\end{theorem}
\begin{proof}
Every singleton is live. For a dead stratum $\sigma$, Lemma~\ref{lem:rank} and
$\sigma\subseteq U_\sigma$ give $r(K_\sigma)=r(A_\sigma)-|U_\sigma|\le 4-|\sigma|$. So a dead
stratum with at least four elements has rank zero and the fibre $\{0\}$ by
Lemma~\ref{lem:zerofibre}, a dead pair has rank at most two and a dead triple has rank at most
one. By \eqref{eq:nullity} and Lemma~\ref{lem:rank}, $|A_\sigma\setminus U_\sigma|-r(K_\sigma)=|\sigma|$: the corank of $K_\sigma$ is $|\sigma|$.

We first analyse the pairs. Let $\sigma=\{i,j\}$ be dead. Then $K_\sigma$ has corank two and no coloops, so
by Theorem~\ref{thm:coranktwo} and Theorem~\ref{thm:main}, applied to $K_\sigma$ with any of its
bases, the relaxed coparking functions of $K_\sigma$ form a pure multicomplex in two variables
with degree sequence $h(K_\sigma)$. Renaming the variables $x_i,x_j$ gives a pure multicomplex in
$\N^\sigma$ with degree sequence $h(K_\sigma)$. Any such multicomplex is an unconditional fibre.
The only proper strata are the live singletons, with $b_{\{i\}}(i)=|C_i|$. The ground set of
$K_\sigma$ is $C_i\cap C_j$, so $|C_i\cap C_j|=r(K_\sigma)+2$ and
$b_\sigma(i)=|C_i|-|C_i\cap C_j|=|C_i|-r(K_\sigma)-2$. An offset $m$ of the multicomplex has
$m_i\le|m|\le r(K_\sigma)$, so $b_\sigma(i)+m_i\le|C_i|-2<b_{\{i\}}(i)$, and likewise for $j$.
No proper stratum can be fired from $b_\sigma+m$, and Lemma~\ref{lem:baseline} applies.

We now study the triples. A dead triple of rank zero has the fibre $\{0\}$, so let $\sigma=\{i,j,k\}$ be dead with $r(K_\sigma)=1$. The bound above is an
equality, so $|U_\sigma|=3$ and $r(A_\sigma)=4=r(M)$. Hence $U_\sigma=\sigma$, and by
\eqref{eq:nullity} $|A_\sigma|=7$, so $A_\sigma=B_0\cup\sigma$. Every tree edge therefore lies in
at least two of $T_i,T_j,T_k$, the baseline is $b_\sigma=(1,1,1)$, and
$K_\sigma=(M|(B_0\cup\sigma))/\sigma$ is a matroid of rank one on $B_0$ without coloops. Its
loops are the tree edges in the closure of $\sigma$. Writing $\lambda$ for their number, the
non-loops form a parallel class of size $4-\lambda\ge2$ and $h(K_\sigma)=(1,3-\lambda)$.

We determine the proper strata that can be fired from $b_\sigma+\mathbf e_i=(2,1,1)$, the $2$ in
the coordinate $i$. A singleton $\{j\}$ with $j\ne i$ has $b_{\{j\}}(j)=|C_j|\ge2$, since $j$ is
not a loop, and cannot be fired. The singleton $\{i\}$ can be fired if and only if $|C_i|=2$,
that is $T_i=\{y\}$ for a single tree edge $y$. The pair $\{j,k\}$ has baseline
$(|T_j\setminus T_k|+1,\ |T_k\setminus T_j|+1)$ and can be fired from $(1,1)$ if and only if
$T_j=T_k$. It is then dead, because $U_{\{j,k\}}=\{j,k\}$ is independent, and the offset is $0$.
The pair $\{i,j\}$ can be fired from $(2,1)$ if and only if $T_j\subseteq T_i$ and
$|T_i\setminus T_j|\le1$. If $T_j=T_i$ the pair is dead for the same reason and the offset is
$\mathbf e_i$, so condition (ii) asks that $\mathbf e_i\in F(\{i,j\})$. If $T_i=T_j\cup y$ the
offset is $0$ and $U_{\{i,j\}}=\{i,j,y\}$. When this set is independent the pair is dead and
nothing is asked. When it is dependent it is a circuit, because $i$ and $j$ are not parallel and
$y\notin C_j$, and the pair is live. The same holds with $k$ in place of $j$.

By the above, the offset $\mathbf e_i$ violates the coparking condition (ii) in exactly three
situations, each of which is due to a tree edge $y$:
\begin{enumerate}[label=(R\arabic*)]
\item $T_i=\{y\}$, so that $\{i,y\}$ is a circuit;
\item $T_j=T_i$ for some $j\in\sigma\setminus i$ and $\mathbf e_i\notin F(\{i,j\})$;
\item $T_i=T_j\cup y$ for some $j\in\sigma\setminus i$ and $\{i,j,y\}$ is a circuit.
\end{enumerate}
In (R2) put $T=T_i=T_j$. The dead node $K_{\{i,j\}}=(M|(T\cup\{i,j\}))/\{i,j\}$ has ground set
$T$, rank $|T|-2$ and no coloops, so $h_1(K_{\{i,j\}})=2-\ell$ with $\ell$ the number of its
loops, the tree edges $y\in T$ in the closure of $\{i,j\}$. For such a $y$ the set $\{i,j,y\}$
is a circuit, since $y$ parallel to $i$ would force $T=\{y\}$ and $i$ parallel to $j$. The fibre $F(\{i,j\})$ is a pure
multicomplex in $x_i,x_j$ with $h_1(K_{\{i,j\}})$ linears, so exactly
$2-h_1(K_{\{i,j\}})=\ell$ of the two offsets $\mathbf e_i,\mathbf e_j$ violate the coparking
condition at this pair, and we match them to the $\ell$ loops, one each. In all three situations
the tree edge $y$ lies in a circuit of size at most three with elements of $\sigma$, so
$y\in\operatorname{cl}(\sigma)$ and $y$ is a loop of $K_\sigma$.

Distinct violating offsets are due to distinct loops. Suppose $\mathbf e_i$ and $\mathbf e_{i'}$
with $i'\ne i$ both violate the coparking condition due to the same $y$, through circuits $Z$ and $Z'$ of size at most three, each consisting of $y$ and elements
of $\sigma$. If $Z=\{i,y\}$ and $Z'=\{i',y\}$ then $i$ and $i'$ are parallel. If $Z=\{i,y\}$ and
$|Z'|=3$ then $Z'$ does not contain $i$, since a circuit does not properly contain another, so
$i\in\operatorname{cl}(y)\subseteq\operatorname{cl}(Z'\setminus y)$. If $|Z|=|Z'|=3$ and
$Z\ne Z'$ then the two circuits share $y$ and one element of $\sigma$, and the third elements
lie in the closure of those two. If $Z=Z'=\{i,j,y\}$ then both violations come from the pair $\{i,j\}$, which is impossible: in
(R2) the violations are matched to distinct loops and in (R3) only $\mathbf e_i$ violates. In every case three elements
of $\sigma$ lie in a set of rank at most two, contradicting the independence of $\sigma$. Hence
at most $\lambda$ of the offsets $\mathbf e_i$ violate the coparking condition, and at least
$3-\lambda=h_1(K_\sigma)$ do not.

Let $L$ be a set of $h_1(K_\sigma)$ indices $i$ for which $\mathbf e_i$ does not violate the
coparking condition. Then $F(\sigma)=\{0\}\cup\{\mathbf e_i:i\in L\}$ is a pure multicomplex
with degree sequence $h(K_\sigma)$, and it satisfies (ii): for the offset $0$ by
Lemma~\ref{lem:mono}(d), and for the offsets $\mathbf e_i$ by the choice of $L$. So $F(\sigma)$
is a fibre.
\end{proof}

For a matroid of rank at most three no dead triple has rank one and a dead pair has rank at most
one, so all fibres are linears and unconditional. This is the theorem of H\`a, Stokes and Zanello
\cite{HSZ}. Rank four is the theorem of Klee and Samper \cite{KleeSamper}, here with an explicit
pure multicomplex for every basis.

\subsection{Uniform matroids}\label{sec:uniform}

The rank-four case has linear fibres at its triples because those dead nodes have rank one. Uniform matroids show the other extreme: for $M=U_{d,n}$ the dead nodes are again uniform, of every rank up to $d-2$, and the
fibres are full simplices of monomials, compatible with one another rather than unconditional.

\begin{theorem}\label{thm:uniform}
Let $M=U_{d,n}$ with $d\ge2$ and let $B_0$ be any basis. Then $C_i=B_0\cup i$ for every red
element $i$, the dead strata are the sets $\sigma\subseteq R$ with $2\le|\sigma|\le d$, each with
$U_\sigma=\sigma$, $b_\sigma=(1,\dots,1)$ and $K_\sigma\cong U_{d-|\sigma|,d}$, and
\[
  F(\sigma)=\{m\in\N^\sigma:\ |m|\le d-|\sigma|\}
\]
is a fibre system. Hence Stanley's conjecture holds for uniform matroids.
\end{theorem}
\begin{proof}
In a uniform matroid every $d$ elements are independent and every $d+1$ are dependent, so the
fundamental circuit of $i\in R$ is $C_i=B_0\cup i$. For $|\sigma|\ge2$ every element of $B_0$
lies in all $|\sigma|$ circuits of $\sigma$, so $U_\sigma=\sigma$ and $A_\sigma=B_0\cup\sigma$.
Hence $\sigma$ is dead if and only if $2\le|\sigma|\le d$, with $b_\sigma(i)=|C_i\cap\sigma|=1$ for all
$i\in\sigma$, and $K_\sigma=(M|A_\sigma)/\sigma=U_{d,d+|\sigma|}/\sigma\cong U_{d-|\sigma|,d}$
on the ground set $B_0$. Its $h$-vector counts the monomials of degree $k\le d-|\sigma|$ in
$|\sigma|$ variables, so $F(\sigma)=\{m\in\N^\sigma:|m|\le d-|\sigma|\}$ is a pure multicomplex
with degree sequence $h(K_\sigma)$, which is condition (i).

For (ii), let $m\in F(\sigma)$. The only live proper strata inside $\sigma$ are the singletons, with
$b_{\{i\}}(i)=|C_i|=d+1$, and the vector $b_\sigma+m$ has $i$-coordinate $1+m_i\le1+d-|\sigma|<d+1$, so
they cannot be fired. The proper strata that can be fired from $b_\sigma+m$ are therefore
exactly those of size at least two contained in $\sigma$, and indeed all of them can, as $b_\tau=(1,\dots,1)\le(b_\sigma+m)|_\tau$. Each such $\tau$ is
dead, and the offset there is $m|_\tau$, of degree at most
$d-|\sigma|\le d-|\tau|$, hence in $F(\tau)$. So $b_\sigma+m$ is coparking on $\mathring\sigma$ and $F$ is a
fibre system.
\end{proof}

The fibre at $\sigma$ is not unconditional when $3\le|\sigma|<d$: for $m=\mathbf e_j$ and $i\ne j$ the dead stratum $\sigma\setminus i$ can be fired from $b_\sigma+m$ with the nonzero offset $\mathbf e_j$. What makes
the system work is that the fibres are nested in the right way, $F(\sigma)|_\tau\subseteq F(\tau)$
for $\tau\subseteq\sigma$. Uniform matroids are paving, lattice-path and cotransversal matroids and positroids, so this case of Stanley's conjecture is
also covered by the theorems of \cite{MNRV}, \cite{Schweig}, \cite{OhCotransversal} and \cite{Positroid}.

\section{Limitations of the model and further directions}\label{sec:outlook}

Let $G_4$ be the radius-two root with three
fertile hubs $1,2,3$, no barren hubs, and eight children $v_0,\dots,v_7$, with $v_0,v_3,v_6$
children of $1$, $v_1,v_4,v_7$ children of $2$ and $v_2,v_5$ children of $3$. Its red edges are
the path edges $e_i=v_{i-1}v_i$, so that, extending the colour words of Section~\ref{sec:htwo}
to three fertile hubs, the path reads $\word{12312312}$. Its full stratum $\sigma$ is dead with
anchors $1$ and $2$, a triangle node with $h(K_\sigma)=(1,3,4)$ by
Proposition~\ref{prop:startri}. Every prefix ending at a child of $1$ and every suffix starting at a child
of $2$ is a return, killing $x_2,x_3,x_5,x_6$ and every quadratic they divide. Every prefix
ending at a child of $2$ together with a suffix starting at a child of $1$ is a double bridge,
killing $x_1x_4$, $x_1x_7$ and $x_4x_7$ (Section~\ref{sec:killers}). So the only quadratic offsets
that can be spared are $x_1^2,x_4^2,x_7^2$. Three spared quadratics
cannot fill $h_2=4$, so $\sigma$ is obstructed for every fibre system on $G_4$, and $G_4$ does
not carry a fibre system.

The same count works along the family $\word{(123)^q12}$, whose full stratum is a triangle
node with $h(K_\sigma)=(1,q+1,2q)$ by Proposition~\ref{prop:startri}(d), since the third hub has
$q$ internal children. The returns leave the $q+1$ linears $x_1,x_4,\dots,x_{3q+1}$ and the double
bridges kill every product of two of them, so at most $q+1$ quadratics are spared for
$h_2=2q$, a deficit of $q-1$. Consider the roots with three fertile hubs, no barren hubs and a red
path through at most eight children. Up to reversal of the path and renaming of the
hubs, $G_4$ is the only one that carries no fibre system \cite{Notebook}. The Petersen path
$\word{ABCABC}$ has the same dead node $h(K_\sigma)=(1,3,4)$ and admits a fibre, so the failure
is not a property of the dead node but of the way the stratum sits in its graph.

The original goal of this project was to settle the conjecture for all radius-two graphs with
this model, and $G_4$ is where that goal failed: once a graph contains several dead strata
whose dead nodes have rank two or more, we found no way to analyse them together. Which classes
of radius-two graphs beyond the triconed ones carry a fibre system with the breadth-first tree
is open.

Four questions are left open by this paper.

When a stratum $\sigma$ is obstructed, one can still put $F(\sigma)=\{0\}$ there, as in
Section~\ref{sec:truncated}, and continue choosing fibres above it. By the proof of
Proposition~\ref{prop:gap}, $\PP^*(F)$ is then a multicomplex, and the identity of
Theorem~\ref{thm:B} says exactly how much is missing and where, namely
$t^{|b_\sigma|}(h(K_\sigma)-1)E_\sigma(t)$ for each such $\sigma$. A stratum is obstructed
because condition (ii) of Definition~\ref{def:fibre} leaves too few offsets, so the
natural question is whether that condition can be weakened.

\begin{question}
Can the coparking condition (ii) of Definition~\ref{def:fibre}
be relaxed, at the price of a modified Definition~\ref{def:relaxed}, so that
Theorem~\ref{thm:main} still holds and the root $G_4$ carries a fibre system?
\end{question}

With the trivial fibre at every dead stratum, the truncated $h$-polynomial $\tilde h(M,B_0)$ is the
degree sequence of $\PP^*_0$ and hence an $O$-sequence. For $G_0$ with its tree it is even a pure one
(Section~\ref{sec:truncated}). Since $h_1=|E|-r(M)$ for a matroid without loops or coloops, a
graphic matroid with $h$-vector $\tilde h(M(G_0),B_0)=(1,4,10,17,20,13)$ would come from a connected
multigraph on six vertices with nine edges, and none of these has it \cite{Notebook}. So if
$\tilde h(M(G_0),B_0)$ is a matroid $h$-vector at all, the matroid is not graphic.

\begin{question}
Is $\tilde h(M,B_0)$ the $h$-vector of a matroid?
\end{question}

Sections~\ref{sec:coranktwo}
to~\ref{sec:uniform} give three classes on which every basis carries a fibre system.

\begin{question}
Is there a natural class of non-graphic matroids of unbounded rank, beyond the uniform
matroids, in which every basis carries a fibre system?
\end{question}

The relaxation was set up for the fundamental circuits of a basis, and the private
elements (F1) are used throughout.

\begin{question}\label{q:cycles}
Does the model extend to an arbitrary collection of $g$ cycles, the setting of the recursion of
\cite{CycleSystems} and of Definition~\ref{def:gcs}?
\end{question}

\section*{Acknowledgements}
The author thanks Anton Dochtermann, Evan Huang, Kai Mawhinney and Yuchen Xu, the other members of the 2025 REU group at Texas State University, for the discussions on cycle systems that led to the idea of relaxing the coparking functions at the strata where the recursion stops. Many of the tools used here to follow circuits and unique unions through deletion and contraction originated in that project \cite{Regular}, among them the systematic use of the circuit--cocircuit intersection property (F2), used there in its binary form, and the arguments behind Lemmas~\ref{lem:invariants} and~\ref{lem:E}. Over those discussions and others, Anton Dochtermann in particular generously shared many of the ideas and much of the philosophy behind cycle systems. He also raised the question of whether Dhar's burning algorithm has an analogue for cycle systems, which Theorem~\ref{thm:burning} answers. The author also thanks David Perkinson for sharing his notes on deletion--contraction trees beyond cycle systems and Lixing Yi for the observation on biconed
graphs recorded in Example~\ref{ex:biconed-gcs}.

\section*{Use of AI tools}
The author used two AI tools in preparing this paper, Claude (Anthropic, mainly the models Fable 5 and Opus 4.8) and ChatGPT (OpenAI, mainly the model GPT-5.6 Sol), and describes below what each contributed, including to the mathematical reasoning.

AI contributed to the mathematical reasoning in four places. Claude formalised the conditions that a fibre at a dead stratum has to satisfy, Definition~\ref{def:fibre}, and the extension sets of Definition~\ref{def:ext}, from the author's idea of relaxing the coparking conditions at dead strata. Following the author's suggestion that an extension set is the coparking set
of the complement, Claude formalised the reduction to a contraction
(Lemma~\ref{lem:extcontract}) and found the burning-algorithm characterisation that Theorem~\ref{thm:burning} states. These replaced an earlier deletion--contraction argument of the author's, modelled on the proof in \cite{CycleSystems}. ChatGPT found the root $G_4$ of Section~\ref{sec:outlook} on which the model fails, and observed that the Petersen path and $G_4$ share a dead node. Asked by the author to look for classes of matroids in which every basis carries a fibre system, Claude identified the matroids of rank at most three and the uniform matroids as such classes and stated Theorem~\ref{thm:uniform} and the rank-three case of Theorem~\ref{thm:rankfour}.

AI served further supporting roles. The statements of this paper were formulated by Claude, and those of Section~\ref{sec:radius2} by Claude and ChatGPT. Claude drafted the manuscript in two versions, one with proofs and one without, including the examples and figures. The author worked from the version without proofs and wrote all proofs independently. Claude then read the proofs and pointed out gaps and missing cases, which the author closed, after which Claude was used to optimise and reword them. The whole text was revised, edited and checked line by line by the author. Every computation reported in the companion notebook was carried out by Claude or ChatGPT, which wrote and ran the Python code that enumerated strata, fibres, extension sets and relaxed coparking functions. The author specified each computation and the instances it ran over, and reran the code locally or on Google Colab.

The author provided the main ideas of this paper: the direction of generalising cycle systems to radius-two graphs, the idea of relaxing the coparking conditions at dead strata, the idea of unconditional fibres, and the tools developed in the 2023 REU on hypergraph parking functions \cite{BDHOZ} and the 2025 REU on cycle systems \cite{Regular}. The proofs in this paper are the author's, written as described above.

The author checked all mathematical content and all computational results, and takes full responsibility for the paper.

\bibliographystyle{amsplain}
\bibliography{bib}

\end{document}